\documentclass[11pt,preprint]{imsart}
\usepackage[margin=1in]{geometry}
\usepackage[shortlabels]{enumitem}

\usepackage{bm}             
\usepackage{booktabs}
\usepackage{array}
\usepackage{multirow,amsmath}
\usepackage{amsthm}
\usepackage{float}

\usepackage{cancel}

\DeclareMathOperator{\dssnr}{\delta_{\text{S}}}

\DeclareMathOperator{\tr}{tr}

\DeclareMathOperator{\im}{Im}
\DeclareMathOperator{\rank}{rank}

\def\Expect{{\mathbb{E}}}

\def\Zbf{{\bm{Z}}}

\def\R{{\mathbb{R}}}

\def\bA{{\bm{A}}}
\def\bB{{\bm{B}}}
\def\bu{{\bm{u}}}
\def\bw{{\bm{w}}}
\def\bv{{\boldsymbol{v}}}
\def\bM{{\bm{M}}}
\def\bV{{\bm{V}}}
\def\bU{{\bm{U}}}
\def\bx{\bm{x}}

\def\bS{{\bm{S}}}
\def\bX{{\bm{X}}}
\def\bY{{\bm{Y}}}

\def\bI{{\bm{I}}}
\def\bP{{\bm{P}}}
\def\bQ{{\bm{Q}}}
\def\bR{{\bm{R}}}
\def\bSigma{{\bm{\Sigma}}}

\def\calA{{\mathcal{A}}}

\def\calR{{\mathcal{R}}}
\def\calS{{\mathcal{S}}}
\def\calX{{\mathcal{X}}}
\def\psdgeq{\succeq}
\def\psdleq{\preceq}

\providecommand{\dssnr}{}
\renewcommand{\dssnr}{{\delta_{\mathrm{S}}}}

\usepackage{amsfonts,amssymb}
\usepackage[numbers,sort&compress]{natbib}
\usepackage[hidelinks]{hyperref}
\usepackage{graphicx}
\usepackage{xcolor}
\makeatletter
\@ifpackageloaded{ulem}{}{\usepackage[normalem]{ulem}}
\makeatother
\newif\ifmarkchanges
\markchangesfalse
\DeclareRobustCommand{\add}[1]{\ifmarkchanges{\ifvmode\leavevmode\fi\color{blue}#1}\else#1\fi}
\DeclareRobustCommand{\del}[1]{\ifmarkchanges{\ifvmode\leavevmode\fi\color{red}\sout{#1}}\else\ignorespaces\fi}
\makeatletter
\newcommand{\markstrike@}[2]{\setbox\z@\hbox{$\m@th#1#2$}\rlap{\rule[0.45\ht\z@]{\wd\z@}{0.4pt}}\box\z@}
\DeclareRobustCommand{\mdel}[1]{\ifmarkchanges{\color{red}\mathpalette\markstrike@{#1}}\fi}
\DeclareRobustCommand{\delbox}[1]{\ifmarkchanges{\ifvmode\leavevmode\fi\color{red}\sbox\z@{#1}\rlap{\rule[0.5ex]{\wd\z@}{0.4pt}}\usebox\z@}\else\ignorespaces\fi}

\makeatother
\startlocaldefs
\theoremstyle{plain}

\newtheorem{theorem}{Theorem}[section]
\newtheorem{lemma}[theorem]{Lemma}
\newtheorem{corollary}[theorem]{Corollary}
\newtheorem{proposition}[theorem]{Proposition}
\theoremstyle{definition}
\newtheorem{definition}[theorem]{Definition}

\endlocaldefs

\begin{document}

\begin{frontmatter}
\title{Theoretical Guarantees for the Subspace-Constrained Tyler's Estimator
}
\runtitle{Theoretical Guarantees for the Subspace-Constrained Tyler's Estimator}

\begin{aug}
\author[A]{\fnms{Gilad}~\snm{Lerman}\thanksref{t1}\ead[label=e1]{lerman@umn.edu}}
\and
\author[B]{\fnms{Teng}~\snm{Zhang}\thanksref{t1}\ead[label=e3]{teng.zhang@ucf.edu}}
\thankstext{t1}{Both authors are corresponding authors.}
\address[A]{School of Mathematics,
University of Minnesota
\printead[presep={,\ }]{e1}}

\address[B]{Department of Mathematics,
University of Central Florida
\printead[presep={,\ }]{e3}}
\end{aug}

\begin{abstract}
We study the subspace-constrained Tyler's estimator (STE) for recovering a low-dimensional linear subspace from data contaminated by outliers. Under deterministic conditions on inlier spread and outlier alignment, we prove linear convergence to the true subspace from a sufficiently accurate initialization. The guarantee also applies when the dimension-scaled signal-to-noise ratio is below one, where general robust subspace recovery is SSE-hard. When initializing STE from a candidate subspace, we provide a sufficient recovery condition based on the initial angular error and the size of the identity shift applied to the subspace projector. For noisy inliers, we prove a subspace error proportional to the noise level within an explicit iteration bound. Finally, under an asymptotic, anisotropic Gaussian inlier--outlier model, we show that Tyler's M-estimator (TME) provides a successful STE initialization even when TME's own leading eigenspace differs from the target subspace.
\end{abstract}



\end{frontmatter}

\section{Introduction}

Robust subspace recovery (RSR) aims to estimate a low-dimensional linear subspace from data heavily contaminated by outliers~\cite{lerman2018overview}. It plays a central role in geometric computer vision tasks such as fundamental matrix estimation and pose synchronization. Whereas Random Sample Consensus (RANSAC)~\cite{fischler1981random} generates model hypotheses from random subsets, global RSR algorithms update an estimate using the entire dataset. Among global approaches, the subspace-constrained Tyler's estimator (STE) has demonstrated exceptional empirical performance: in fundamental matrix estimation, STE is competitive with RANSAC while outperforming competing global RSR solvers~\cite{Yu2024}. It has also been successfully deployed for outlier filtering in monocular camera pose estimation~\cite{li2025cyclesync}, and its higher-order extension (HOSTE) outperforms standard tensor decompositions in trifocal tensor synchronization~\cite{Miao_trifocal_NEURIPS2024}. The goal of this work is to establish quantitative guarantees for STE's subspace recovery and convergence in outlier-dominated regimes. Earlier versions of some of these results were stated without proof in~\cite{Yu2024}; here we give complete proofs together with new results.

Theoretical analysis of RSR often focuses on the ``noiseless'' setting, where inliers lie exactly on an underlying subspace among arbitrary outliers. This setup exposes a fundamental computational threshold governed by the inlier fraction and subspace dimensions: above this boundary, Tyler's M-estimator (TME)~\cite{tyler1987distribution}---a classical iterative method for robust scatter estimation---guarantees exact subspace recovery under suitable data conditions~\cite{zhang2016robust,lerman2026sharpphasetransitiontylers}. Below this threshold, general RSR is Small Set Expansion (SSE)-hard in the worst case~\cite{hardt2013algorithms}. STE incorporates the target dimension directly into the iterative scatter update: at each step, it retains the leading eigenvalues corresponding to the prescribed subspace dimension while suppressing remaining noise directions with a common shrunken value. We give deterministic conditions under which STE converges linearly to the target subspace from a sufficiently accurate initialization, including below the TME recovery threshold.

\subsection{Previous Relevant Theories}
\label{sec:prior_theory}

To motivate our theoretical contributions, we review existing recovery guarantees for RSR. The standard formulation of the noiseless RSR problem assumes a dataset $\calX$ consisting of $n_1$ inliers lying exactly on an underlying $d$-dimensional linear subspace $L_* \subset \R^D$ (referred to as a $d$-subspace for brevity), and $n_0$ outliers lying outside $L_*$. We refer to such a dataset as a noiseless inlier–outlier dataset. The central question in noiseless RSR is whether one can exactly and efficiently recover $L_*$.

A natural quantity in this setup is the ratio $n_1/n_0$, and one seeks the smallest such ratio that permits exact recovery. Following~\cite{lerman2018overview,maunu19ggd,maunu2019robust}, this is often viewed as a signal-to-noise ratio (SNR), where the inliers represent the signal and the outliers represent the noise. Building on earlier results~\cite{hardt2013algorithms,zhang2016robust}, Yu et al.~\cite{Yu2024} proposed scaling $n_1$ and $n_0$ by the dimension and codimension of the inlier subspace, $d$ and $D-d$, respectively. This leads to the \textit{dimension-scaled SNR}, denoted by $\dssnr$:
\begin{equation*}
    \dssnr := \frac{n_1/d}{n_0/(D-d)}.
\end{equation*}
This quantity characterizes the computational hardness of RSR. Hardt and Moitra~\cite{hardt2013algorithms} showed that when $\dssnr < 1$, the noiseless RSR problem is Small Set Expansion (SSE)-hard. Under the Small Set Expansion Hypothesis of~\cite{raghavendra2010graph}, this hardness rules out a polynomial-time algorithm for arbitrary instances in that regime. In the special case $d = D-1$, Hardt and Moitra invoked a result from~\cite{Khachiyan_complexity1995} to show that the problem is NP-hard.

\begin{theorem}[\cite{hardt2013algorithms}]
\label{thm:sse_hard}
The noiseless RSR problem is SSE-hard if $\dssnr < 1$. Furthermore, if $\dssnr < 1$ and $d = D-1$, then the problem is NP-hard.
\end{theorem}

Hardt and Moitra~\cite{hardt2013algorithms} also showed that subspace-search algorithms can recover $L_*$ when $\dssnr > 1$, assuming certain general position properties. While their algorithms are theoretically guaranteed, they are not practically competitive. In contrast, Zhang~\cite{zhang2016robust} proposed applying TME as a practical method for RSR and established an initial recovery guarantee. The following sharper characterization of the recovery threshold was recently obtained in~\cite{lerman2026sharpphasetransitiontylers}.
\begin{theorem}[\cite{lerman2026sharpphasetransitiontylers}]
\label{thm:zhang16_extend}
Let $\calX$ be a noiseless inlier–outlier dataset that does not contain the origin, with an underlying $d$-subspace $L_*$, and $\calX_{\mathrm{in}}: =\calX\cap L_*$ be the set of inliers. Suppose the following conditions hold:
\begin{enumerate}
    \item $\dssnr \geq 1$. \label{item:dssnr1}
    \item
   $L_*$ is the unique proper nonzero subspace with maximal point density $|\calX\cap L|/\dim(L)$. That is, for every subspace \( L \subsetneq \R^D \) with \( L \neq L_* \) and $\dim(L) \geq 1$, $|\calX \cap L|/\dim(L) < |\calX \cap L_*|/\dim(L_*)$.
\label{item:convergence}
\end{enumerate}
Then trace-normalized TME, applied within $\operatorname{span}(\calX)$ and initialized with an arbitrary positive definite matrix on that span, converges to a singular ambient matrix whose range is exactly $L_*$. When $\operatorname{span}(\calX)\neq\R^D$, the iteration uses the dimension and the inverse on that span, as specified in Section~\ref{sec:ste}.
\end{theorem}

In practical problems, the regime where $\dssnr < 1$ needs to be addressed. For example, in fundamental matrix estimation, which translates to RSR with $D=9$ and $d=8$, the restriction to the regime when $\dssnr \geq 1$ is equivalent to restricting the fraction of outliers to be at most 1/9, which is unreasonable in practice~\cite{Yu2024}.

One approach to the regime $\dssnr < 1$ in algorithmic robust statistics imposes distributional assumptions on the inliers while allowing arbitrary outliers. For example, \cite{Diakonikolas_robust2019} developed polynomial-time algorithms for robust covariance estimation—a problem directly related to robust subspace recovery—by relying on Gaussianity of the inliers. The textbook \cite{Diakonikolas_Kane_2023} on algorithmic statistics reviews many related works. Notably, \cite{Raghavendra2020} and \cite{Bakshi2021} independently developed ``list-decodable'' algorithms for RSR, following the framework of~\cite{charikar2017learning} and using the sum-of-squares hierarchy. These methods output a small list of candidate subspaces, guaranteeing that at least one is close to the true subspace. Bakshi and Kothari \cite{Bakshi2021} use certifiable moment conditions on quadratic measurements for list-decodable recovery, and additionally require certifiable anti-concentration for high-accuracy guarantees. Ivkov and Kothari~\cite{IK22} subsequently obtained exact
list-decodable subspace recovery. These list-decodable
guarantees do not by themselves identify the target among
the candidates, and their high-accuracy running-time bounds
involve polynomial degrees depending on the inlier fraction.
This motivates studying simpler iterative methods under
explicit geometric conditions.

A different line of work assumes generic data-dependent conditions on the sampled points~\cite{lerman15reaper,GMS_2011,maunu19ggd,lerman2025globalconvergenceiterativelyreweighted}. These conditions have been used to analyze iterative algorithms for the $\ell_1$ subspace approximation problem, which minimizes the sum of Euclidean distances from data points to a candidate subspace. Some structural properties of this nonconvex optimization problem were studied in \cite{lp_recovery_part1_11}; furthermore, it was shown to be NP-hard in \cite{CLarkson_Woodruff_L1_NPhard}, prompting polynomial-time $(1+\epsilon)$-approximation schemes via oblivious sketching \cite{Levin+RSR_stream2018} and ridge leverage score coresets \cite{woodruff_icml_2024}. In terms of fast local algorithms under generic conditions, \cite{maunu19ggd} showed that geodesic subgradient descent (or projected subgradient descent) for the $\ell_1$ energy function achieves linear convergence to $L_*$ when well-initialized, providing a generic condition under which PCA yields a suitable initialization. Similarly, \cite{lerman2025globalconvergenceiterativelyreweighted} established generic conditions for the linear convergence of an iteratively reweighted least squares (IRLS) scheme~\cite{lerman2017fast}. For fixed ambient dimensions, these algorithms have per-iteration costs linear in the number of data points. Nevertheless, generic data conditions are typically less interpretable and do not explicitly capture $\dssnr$; proving that they hold when $\dssnr \ll 1$ generally requires restrictive probabilistic assumptions on the data distribution~\cite{lerman15reaper,GMS_2011,maunu19ggd,lerman2025globalconvergenceiterativelyreweighted}.

\subsection{Contributions}

Our main result is a deterministic basin-of-attraction theorem for STE (Theorem~\ref{thm:main}). Under conditions on inlier spread, outlier alignment, and initialization, the estimated subspace converges linearly to $L_*$. The theorem also applies when $\gamma<\dssnr<1$, where $\gamma$ is the STE shrinkage parameter. We further provide explicit error bounds and an asymptotic convergence-factor bound of $\gamma/\dssnr$ (Corollary~\ref{cor:rate}). When initializing STE from a candidate subspace projector, Corollary~\ref{cor:angle_basin} states the sufficient accuracy directly in terms of the initial principal angle and the size of the identity shift applied to the projector. The proof tracks a Schur-complement SNR whose growth controls both subspace error and conditioning.

For inliers with noise of level $\epsilon$, we prove an
$O(\epsilon)$ subspace error within an explicit iteration
bound under the stated assumptions
(Theorem~\ref{thm:noisy} and
Proposition~\ref{prop:noisy_high_init}).
We also analyze TME initialization under the generalized
haystack model~\cite{maunu19ggd}. When the limiting
dimension-scaled SNR is below but sufficiently close to one,
the TME initialization satisfies the STE basin condition
almost surely for all sufficiently large samples. This
includes cases where TME alone fails to recover $L_*$
(Theorem~\ref{thm:haystack}).
An application in Appendix~\ref{sec:initialization} combines
the Bakshi--Kothari list-decodable initializer with STE:
once a candidate satisfies our basin condition, its STE
trajectory converges linearly to the exact target subspace,
without increasing the SoS relaxation degree or
candidate-list size.

\subsection{Inlier--Outlier Models and Notation}
\label{sec:inlier_outlier_models}

Throughout, we assume $1\leq d<D$ and that all data points are nonzero. The deterministic convergence results additionally assume $n_0,n_1>0$. All sums and counts include multiplicities. We denote the given dataset of $N$ points in $\R^D$ by $\calX = \{\bx_i\}_{i=1}^N \subset \R^D$, and the underlying $d$-subspace in $\R^D$ by $L_*$. 
An inlier--outlier model partitions the dataset into
inliers $\calX_{\mathrm{in}}$ and outliers
$\calX_{\mathrm{out}}$, 
 with $n_1 = |\calX_{\mathrm{in}}|$ and $n_0 = |\calX_{\mathrm{out}}|$, so that $N = n_0 + n_1$. In the noiseless model, the inliers lie exactly on $L_*$, that is, $\calX_{\mathrm{in}} := \calX \cap L_*$ and $\calX_{\mathrm{out}} := \calX \setminus \calX_{\mathrm{in}}$.

In the noisy model, the inliers lie inside a cone around $L_*$ while the outliers lie strictly outside it. Specifically, for a relative noise parameter $0<\epsilon\leq1$, the partition satisfies
\begin{equation*}
    \|\bP_{L_*^\perp} \bx\| \leq \epsilon \|\bP_{L_*} \bx\| \quad \forall \bx \in \calX_{\mathrm{in}}, \qquad \text{and} \qquad \|\bP_{L_*^\perp} \bx\| > \epsilon \|\bP_{L_*} \bx\| \quad \forall \bx \in \calX_{\mathrm{out}}.
\end{equation*}

\subsection{Additional Notation}
\label{sec:notation}

We use bold uppercase and lowercase letters to denote matrices and column vectors, respectively. For a matrix $\bA$, we denote its trace by $\tr(\bA)$ and its image (i.e., column space) by $\im(\bA)$. We denote by $S_{+}(D)$ and $S_{++}(D)$ the sets of positive semidefinite (p.s.d.) and positive definite (p.d.) matrices in $\R^{D \times D}$, respectively.
Let $\bI_k$ denote the identity matrix in $\R^{k \times k}$. When clear from context, we abbreviate $\bI_k$, $S_{+}(D)$, and $S_{++}(D)$ by $\bI$, $S_{+}$, and $S_{++}$, respectively. For $k \leq D$, we denote by $O(D,k)$ the set of semi-orthogonal matrices in $\R^{D \times k}$, i.e., $\bU \in O(D,k)$ if and only if $\bU \in \R^{D \times k}$ and $\bU^\top \bU = \bI_{k}$. We also denote by $O(k)$ the set of orthogonal matrices in $\R^{k \times k}$, so $O(k)\equiv O(k,k)$.
For a $d$-subspace $L \subset \R^D$, we denote by $\bP_L$ the $D \times D$ matrix representing the orthogonal projector onto $L$. This matrix is symmetric and satisfies $\bP_L^2 = \bP_L$ and $\im(\bP_L) = L$. We also fix an arbitrary matrix $\bU_L \in O(D,d)$ such that $\bU_L \bU_L^\top = \bP_L$ (such a matrix $\bU_L$ is determined up to right multiplication by an orthogonal matrix in $O(d)$). For $\bx \in \R^k$, we denote by $\|\bx\|$ its Euclidean norm, and for $\bA \in \R^{k\times m}$, we denote by $\|\bA\|$ its spectral norm. We denote the singular values of $\bA \in S_{+}(D)$, which coincide with its eigenvalues, by
$\sigma_1(\bA) \geq \sigma_2(\bA) \geq \cdots \geq \sigma_D(\bA)$. We write $\bA^{-\top}:=(\bA^{\top})^{-1}$, and we denote by $\|\bA\|_F$ the Frobenius norm of $\bA$. For two subspaces $L_1,L_2\in\{L_*,L_*^\perp\}$ and a matrix $\bM\in\R^{D\times D}$, we write $\bM_{L_1,L_2}:=\bU_{L_1}^\top\bM\bU_{L_2}$ for the corresponding block of $\bM$; when $\bM$ is given by a compound expression, we write $[\bM]_{L_1,L_2}$ instead. For two $d$-subspaces $L$ and $L'$, we denote by $\angle(L,L')$ the largest principal angle between them (principal angles are reviewed in Appendix~\ref{sec:calculate_kappa1}); we write $\theta_1$ for $\angle(\hat{L},L_*)$ when $\hat{L}$ is an initial estimate of $L_*$.

\subsection{Structure of the Rest of the Paper}
The remainder of the paper is organized as follows. Section~\ref{sec:ste} reviews the TME and STE algorithms. Section~\ref{sec:theory} presents and motivates our main theoretical results. Section~\ref{sec:supp_theory} provides the proofs of Theorems~\ref{thm:main} and~\ref{thm:haystack} and outlines the proof strategy for the noisy setting. Section~\ref{sec:conclusion} summarizes our conclusions. The appendices contain the remaining technical details: Appendix~\ref{sec:calculate_kappa1} analyzes the auxiliary parameters $\kappa_1$, $\kappa_2$, $\kappa_3$, $\calA$, and $\calR$ under special models; Appendix~\ref{sec:proof_thm_noisy} provides the complete proof of Theorem~\ref{thm:noisy} (omitted from the main text due to its structural similarity to the proof of Theorem~\ref{thm:main}); Appendix~\ref{sec:proof_lemmas} contains the proofs of all supporting technical lemmas; and Appendix~\ref{sec:initialization} gives the list-decodable initialization application and its proof.

\section{Brief Review of the TME and STE Algorithms}
\label{sec:ste}
For a dataset $\calX = \{\bx_i\}_{i=1}^N \subset \R^D$, TME estimates a robust shape matrix, defined up to a positive scalar multiple. For concreteness and numerical stability, practical implementations often fix a scale by imposing a trace constraint.

For TME, if the data do not span $\R^D$, we work in $\operatorname{span}(\calX)$, using its dimension in place of $D$ in the update and taking inverses on that span. The resulting matrices are embedded in $\R^D$ when reporting a subspace. This is the convention in \cite{lerman2026sharpphasetransitiontylers}. For data spanning $\R^D$, the TME algorithm~\cite{tyler1987distribution} starts with an arbitrary $\bSigma^{(0)} \in S_{++}(D)$ and iteratively computes
\begin{equation}
\label{eq:tme_iteration}
\bSigma^{(k)} := \frac{D}{N}\sum_{i=1}^N \frac{\bx_i\bx_i^\top}{\bx_i^\top\left(\bSigma^{(k-1)}\right)^{-1}\bx_i}, \quad k \geq 1,
\end{equation}
where, as noted above, $\bSigma^{(k)}$ is typically further scaled to have trace 1. The factor $D/N$ is included so that the equation remains valid as an unnormalized update. The desired $d$-subspace is estimated by its top $d$ eigenvectors. For full-span data, if every nonzero proper subspace $V\subset\R^D$ satisfies $|\calX\cap V|/N<\dim(V)/D$, the TME iteration converges to a positive definite solution of the fixed-point problem~\cite[Theorem~2.2]{tyler1987distribution}:
\begin{equation}\label{eq:TME_definition}
\bSigma=\frac{D}{N}\sum_{i=1}^N \frac{\bx_i\bx_i^\top}{\bx_i^\top\bSigma^{-1}\bx_i}.
\end{equation}
Under the same strict subspace-mass condition, the positive definite solution is unique up to scalar multiplication~\cite[Theorem~2.1]{tyler1987distribution}. We call a positive definite fixed point a \emph{TME solution} when one exists. Theorem~\ref{thm:zhang16_extend}, by contrast, describes a trace-normalized singular limit, which arises when the strict condition fails at the target subspace.

The STE framework~\cite{Yu2024} is designed to exploit the $d$-subspace structure directly, rather than estimating the full shape matrix. It follows a similar update step as TME:
\begin{equation}
\label{eq:ste_prelim_iteration}
\Zbf^{(k)} := \sum_{i=1}^N \frac{\bx_i\bx_i^\top}{\bx_i^\top\left(\bSigma^{(k-1)}\right)^{-1}\bx_i}, \quad k \geq 1.
\end{equation}
It first computes the top $d$ eigenvalues, $\{\sigma_i\}_{i=1}^d$, and eigenvectors, $\{\bu_i\}_{i=1}^d$, of $\Zbf^{(k)}$. For numerical stability, we use the singular value decomposition to compute these. It then replaces each of the bottom $D - d$ eigenvalues of $\Zbf^{(k)}$ with $\gamma \, \bar{\sigma}_{\mathrm{tail}}$, where $0<\gamma<1$ is a fixed shrinkage parameter and $\bar{\sigma}_{\mathrm{tail}}$ denotes the average of the remaining $D - d$ smallest eigenvalues:
\begin{equation}
\label{eq:last_eig_scaling}
    \bar{\sigma}_{\mathrm{tail}} \equiv \bar{\sigma}_{\mathrm{tail}}(\Zbf^{(k)}):=  \frac{1}{D-d}\sum_{i=d+1}^D\sigma_i.
\end{equation}
The estimator $\bSigma^{(k)}$ is then obtained using the shrinkage parameter $\gamma$:
\begin{equation}
\label{eq:ste_iteration}
\bSigma^{(k)} := \sum_{i=1}^d \sigma_i \cdot \bu_i  \bu_i^\top + \gamma\bar{\sigma}_{\mathrm{tail}} \cdot \left(\bI-\sum_{i=1}^d \bu_i  \bu_i^\top \right),
\end{equation}
where, similar to TME, $\bSigma^{(k)}$ is typically further scaled to have trace 1 for numerical stability.

The algorithm terminates when $\|\bSigma^{(k)}-\bSigma^{(k-1)}\|$ is sufficiently small. The output $d$-subspace is spanned by the top $d$ eigenvectors of $\bSigma^{(K)}$ obtained at the last iteration. Both TME and STE are invariant to multiplying each data point by an arbitrary nonzero scalar.

Both updates are positively homogeneous in the shape matrix, so trace normalization at each step preserves the eigenvectors and the resulting subspace sequence. We analyze the unnormalized STE iteration below. The recalled statements about TME matrix limits use trace normalization.

\section{Review and Discussion of the Recovery Theory}
\label{sec:theory}

The recovery conditions combine geometric properties of the data with spectral quantities of the initialization. We introduce these quantities, state the noiseless and noisy guarantees, and analyze TME initialization. The proofs are given in Section~\ref{sec:supp_theory} and the appendices.

\subsection{A Key Quantity for STE}
\label{sec:k1}
The following spectral ratio measures the separation between the initialization's component along the target subspace and its orthogonal component.

\textbf{The initialization-dependent SNR $\boldsymbol{\kappa_1}$:} We define $\kappa_1$ for the initial stage of STE, depending on $\bSigma^{(0)}$. In Section~\ref{sec:supp_theory}, we generalize this definition to any iteration $k$.

For the initialization matrix $\bSigma^{(0)} \in S_{++}$ and subspaces $L_1, L_2 \in \{L_*, L_*^\perp\}$, recall the block notation $\bSigma^{(0)}_{L_1,L_2} = \bU_{L_1}^\top\bSigma^{(0)}\bU_{L_2}$ of Section~\ref{sec:notation}. We define the SNR associated with $\bSigma^{(0)}$ as
\begin{equation*}
\kappa_1 = \frac{\sigma_d\Big(\bSigma^{(0)}_{L_*,L_*} - \bSigma^{(0)}_{L_*,L_*^\perp}\bSigma^{(0)\,-1}_{L_*^\perp,L_*^\perp}\bSigma^{(0)}_{L_*^\perp,L_*}\Big)}{\sigma_1\Big(\bSigma^{(0)}_{L_*^\perp,L_*^\perp}\Big)}.
\end{equation*}
Mathematically, $\kappa_1$ is the ratio between the smallest eigenvalue of the Schur complement~\cite{Zhang_schur_2005} of $\bSigma^{(0)}_{L_*^\perp,L_*^\perp}$ in $\bSigma^{(0)}$ and the largest eigenvalue of $\bSigma^{(0)}_{L_*^\perp,L_*^\perp}$.

\textbf{Examples of calculating $\boldsymbol{\kappa_1}$:} We first compute $\kappa_1$ for the default initialization $\bSigma^{(0)}=\bI$. In this case, $\kappa_1=1$ since
\begin{equation}
\label{eq:simple_example_I}
\bSigma^{(0)}_{L_*,L_*} = \bI_{d\times d}, \quad \bSigma^{(0)}_{L_*^\perp,L_*^\perp} = \bI_{(D-d)\times (D-d)}, \quad \text{and} \quad \bSigma^{(0)}_{L_*^\perp,L_*} = (\bSigma^{(0)}_{L_*,L_*^\perp})^{\top} = \bm{0}_{(D-d)\times d}.
\end{equation}

Next, consider an example where we start with an initial subspace $\hat{L}$. To form a corresponding $\bSigma^{(0)} \in S_{++}$, we choose a small $\alpha > 0$ and set
\begin{equation}\label{eq:STE_init}
    \bSigma^{(0)} = \bP_{\hat{L}} + \alpha\bI.
\end{equation}
We quantify the distance between $\hat{L}$ and $L_*$ using the largest principal angle, denoted by $\theta_1$ (reviewed in Appendix~\ref{sec:calculate_kappa1}). We claim that in this case
\begin{equation}
\kappa_1 = \frac{\alpha+\alpha^2}{(\sin^2\theta_1 +\alpha)^2},
\label{eq:kappa1_angle}
\end{equation}
and prove this claim in Appendix~\ref{sec:calculate_kappa1}. It follows that
\begin{equation}
\label{eq:example_k1_first}
\text{if } \ \alpha \geq \sin^2\theta_1, \ \text{ then } \ \kappa_1 \geq \frac{\alpha}{(2\alpha)^2} = \frac{1}{4\alpha}.
\end{equation}
Thus, choosing a small positive $\alpha\geq\sin^2\theta_1$ makes $\kappa_1$ large.

\textbf{Intuition and further ideas behind $\boldsymbol{\kappa_1}$:}
To gain better intuition, we express the initial estimator $\bSigma^{(0)}$ as a $2\times 2$ block matrix in the basis of $L_*$ and $L_*^\perp$. Defining $\bSigma' = \bSigma^{(0)}_{L_*,L_*^\perp}\bSigma^{(0)\,-1}_{L_*^\perp,L_*^\perp}\bSigma^{(0)}_{L_*^\perp,L_*}$, we decompose this block matrix as follows:
\begin{equation*}
\begin{pmatrix}
\bSigma^{(0)}_{L_*,L_*} & \bSigma^{(0)}_{L_*,L_*^\perp} \\
\bSigma^{(0)}_{L_*^\perp,L_*} & \bSigma^{(0)}_{L_*^\perp,L_*^\perp}
\end{pmatrix}
=
\begin{pmatrix}
\bSigma' & \bSigma^{(0)}_{L_*,L_*^\perp} \\
\bSigma^{(0)}_{L_*^\perp,L_*} & \bSigma^{(0)}_{L_*^\perp,L_*^\perp}
\end{pmatrix}
+
\begin{pmatrix}
\bSigma^{(0)}_{L_*,L_*} - \bSigma' & \hspace{0.1in} 0 \\
0 & \hspace{0.1in} 0
\end{pmatrix}.
\end{equation*}
Lemma~\ref{lemma:g1}(a), established later, states that both matrices on the right hand side (RHS) are p.s.d.~(i.e., degenerate covariance matrices); the first has rank $D-d$ and the second has rank $d$ (assuming $\bSigma^{(0)} \in S_{++}(D)$).
The second matrix isolates a component supported on $L_*$, while the first contains the remaining blocks. The ratio $\kappa_1$ compares the smallest eigenvalue of the isolated signal block with the largest eigenvalue of the orthogonal block. Unlike the point-count SNR, it measures the spectral separation of the initialization relative to the target subspace. It depends on both $\bSigma^{(0)}$ and $L_*$.

We note that the requirement for $\bSigma^{(0)}$ becomes more restrictive as the required value of $\kappa_1$ increases. For a positive definite initialization, $\kappa_1$ is finite. The formal limiting case in which the orthogonal block vanishes and the signal block remains positive definite corresponds to a matrix with range $L_*$, so that the subspace is already recovered.

In~\eqref{eq:def_kappa_1_2}, we define $\hat{\kappa}_1(\bSigma)$, a technical modification of $\kappa_1$ applicable to any iteration. The core of our proof involves showing that $\hat{\kappa}_1(\bSigma^{(k)})$ approaches infinity as $k \to \infty$ (see~\eqref{eq:overiterations}).

\subsection{Additional Quantities and the Main Condition}
\label{sec:additional_quantities}
We define and motivate additional quantities used in our theory, followed by the required lower bound on $\kappa_1$.

\textbf{The quantity $\boldsymbol{\kappa_2}$}: We define the relative dominance ratio for $\bSigma^{(0)}$ with respect to $L_*^\perp$:
\begin{equation*}
\kappa_2 = \frac{\sigma_1\Big(\bSigma^{(0)}_{L_*^\perp,L_*^\perp}\Big)}{\sigma_D(\bSigma^{(0)})}.
\end{equation*}
The quantity $\kappa_2$ compares the largest eigenvalue of the block on $L_*^\perp$ with the smallest eigenvalue of the full initialization matrix. It controls conditioning estimates in the main proof, where its iteration-dependent counterpart $\hat\kappa_2(\bSigma)$ remains bounded.

\textbf{Examples of calculating $\boldsymbol{\kappa_2}$:}
For the default initialization $\bSigma^{(0)}=\bI$, \eqref{eq:simple_example_I} implies $\kappa_2 = 1$.
For $\bSigma^{(0)}$ as defined in \eqref{eq:STE_init}, we derive a formula for $\kappa_2$ in Appendix~\ref{sec:calculate_kappa1}, which implies:
\begin{equation}
\label{eq:example_k1_second}
\text{if } \ \alpha \geq \sin^2\theta_1, \ \text{ then } \ \kappa_2 \leq 2.
\end{equation}

\textbf{The inliers' condition number $\boldsymbol{\kappa_{\mathrm{in},*}}$}:
For noiseless inliers, whenever a positive definite TME solution exists, fix any such solution and denote it by
\begin{equation}
\label{eq:def_sigmainstar}
\bSigma_{\mathrm{in},*} = \text{a fixed TME solution in } S_{++}(d) \text{ for the set } \{\bU_{L_*}^\top\bx \mid \bx \in \calX_{\mathrm{in}}\} \subset \R^d.
\end{equation}
The inliers' condition number is
\begin{equation} \label{eq:defkinstar}
\kappa_{\mathrm{in},*} = \frac{\sigma_1(\bSigma_{\mathrm{in},*})}{\sigma_d(\bSigma_{\mathrm{in},*})}.
\end{equation}
This is analogous to (25) in \cite{maunu2019robust}, replacing the sample covariance with the TME estimator for the projected inliers. It measures the anisotropy of the inlier TME: it equals one for an isotropic solution, and larger values make the sufficient initialization condition \eqref{eq:kappa1} more demanding.

\textbf{The alignment statistic {$\boldsymbol{\calA}$}}: This outlier-dependent quantity is a variant of the alignment statistic \cite{lerman15reaper,maunu19ggd}:
\begin{equation*}
\calA = \frac{D-d}{n_0} \cdot \left\|\sum_{\bx\in\calX_{\mathrm{out}}} \frac{\bx\bx^\top}{\|\bU_{L_*^{\perp}}^\top\bx\|^2}\right\|.
\end{equation*}
We remark that the scaling factor $(D-d)/n_0$ was omitted in \cite{Yu2024}, but we find it natural as explained in Appendix~\ref{sec:normalize_A}. If the outliers align along a low-dimensional linear subspace (e.g., a line), $\calA$ is large; thus, bounding $\calA$ restricts the alignment of outliers \cite{maunu19ggd}.

Moreover, $\calA$ is large if the outliers are close to $L_*$. This property is specific to our formulation. If outliers concentrate on a $d$-subspace close to $L_*$, a more accurate initialization is required to distinguish the subspaces. Indeed, the lower bound for $\kappa_1$ depends quadratically on $\calA$.

\textbf{The relative alignment statistic {$\boldsymbol{\calR}$}}: We define the relative alignment statistic as:
\begin{equation*}
\calR = \frac{\sigma_1 \left(\sum_{\bx\in\calX_{\mathrm{out}}}\frac{\bx\bx^\top}{\|\bU_{L_*^{\perp}}^\top\bx\|^2}\right)}{\bar{\sigma}_{\mathrm{tail}}\left(\sum_{\bx\in\calX}\frac{\bx\bx^\top}{\|\bx\|^2}\right)},
\end{equation*}
where $\bar{\sigma}_{\mathrm{tail}}$ is defined in \eqref{eq:last_eig_scaling}. The numerator is $\calA$ without the scaling factor; we denote it by $\tilde{\calA}:=\frac{n_0}{D-d}\calA$, and we denote the denominator by $\calS$, so that $\calR=\tilde{\calA}/\calS$. We show in Appendix~\ref{sec:R_bound} that $\calR \geq 1$ for noiseless data, and in~\eqref{eq:calA_R} that for the asymptotic generalized haystack model, $\calR = O(1)$ when $D$ is fixed and $D > d+2$.
Similar to $\calA$, $\calR$ becomes large (specifically its numerator) if outliers cluster near $L_*$ or align on a low-dimensional subspace. Furthermore, $\calR$ can be large if inliers do not permeate $L_*$ sufficiently, causing the denominator to shrink.

\textbf{A key assumption:} Our main theorem requires the following initialization condition:
\begin{equation}\label{eq:kappa1}
\kappa_1 \geq C_{\mathrm b} \, \frac{\kappa_{\mathrm{in},*}\,\calA}{\dssnr}\left(\frac{\calA}{\dssnr-\gamma} + \frac{\kappa_2\calR}{\gamma}(1+\kappa_{\mathrm{in},*})\right),
\end{equation}
where the basin constant $C_{\mathrm b}\equiv C_{\mathrm b}(\dssnr,\gamma)$ is given in \eqref{eq:C_value}.
This condition implies that to keep the requirement on $\kappa_1$ moderate, one needs sufficient inlier permeation, restricted outlier alignment, and sufficient separation of outliers from $L_*$.
The constant $C_{\mathrm b}$ in this sufficient bound has not been optimized. Under the generalized haystack model, Section~\ref{sec:models} shows that \eqref{eq:kappa1} holds for the initialization \eqref{eq:STE_init}, with a suitable $\alpha$, whenever $\sin\theta_1\leq c\,\dssnr$, where $c>0$ does not depend on $\dssnr$.

We emphasize that \eqref{eq:kappa1} is a theoretical basin-of-attraction
condition for initialization values $\bSigma^{(0)}$ rather than a directly verifiable condition from unlabeled data.
Indeed, quantities such as $\kappa_{\mathrm{in},*}$, $\calA$, and
$\calR$ depend on the true subspace $L_*$ and on the inlier--outlier
partition. Its purpose is to characterize how initialization accuracy,
inlier permeation, and outlier alignment determine the basin in which STE
converges.

The initialization $\bSigma^{(0)} = \bI$ leads to the Spherical PCA subspace in the first iteration. Similarly, initializing with the TME solution corresponds to starting with the TME subspace. More generally, if an RSR method produces a subspace $\hat{L}$ sufficiently close to $L_*$, one can construct the initialization $\bSigma^{(0)}$ via \eqref{eq:STE_init}. Corollary~\ref{cor:angle_basin} gives sufficient choices of the angular accuracy and $\alpha$ for this initialization.

\subsection{Generic Theory for the Noiseless Case}
\label{sec:theory_noiseless}
Under the following assumptions on a noiseless inlier--outlier dataset, the STE subspace converges to $L_*$. The theorem also applies when $\gamma<\dssnr<1$, below the threshold in Theorem~\ref{thm:zhang16_extend}. We use the constants:
\begin{equation}\label{eq:C_value}
C_{\mathrm b} = \frac{30\dssnr-2\gamma}{\dssnr-\gamma}
\quad \text{ and } \quad
C_0 = \frac{2 \cdot \dssnr}{\dssnr+\gamma}.
\end{equation}

\begin{theorem}\label{thm:main}
Assume that STE with fixed parameter $0 < \gamma < 1$ and initialization $\bSigma^{(0)}\in S_{++}(D)$ is applied to a noiseless inlier--outlier dataset satisfying $\dssnr > \gamma$. Assume that the inlier TME $\bSigma_{\mathrm{in},*}$ exists in $S_{++}(d)$, that $\calA$ and $\calR$ are finite, and that \eqref{eq:kappa1} holds with the constant $C_{\mathrm b}$ in \eqref{eq:C_value}.
Then the largest principal angle between $L_*$ and the subspace $L^{(k)}$ (spanned by the top $d$ eigenvectors of $\bSigma^{(k)}$) converges $R$-linearly to zero. More precisely, for every $k\geq 1$, we have $\angle(L^{(k)},L_*)<\pi/4$ and
\begin{equation}\label{eq:explicit_rate}
\sin\angle(L^{(k)},L_*)\leq \sin\bigl(2\angle(L^{(k)},L_*)\bigr)\leq \frac{\calA\,\kappa_{\mathrm{in},*}}{\dssnr\,\kappa_1}\, C_0^{-(k-1)}.
\end{equation}
\end{theorem}

We next illustrate Theorem~\ref{thm:main} for the initialization
$\bSigma^{(0)}=\bP_{\hat L}+\alpha\bI$ in \eqref{eq:STE_init}.
Here the initialization has a direct geometric interpretation: $\hat L$
is an initial subspace estimate, and $\alpha>0$ sets the regularization
level. The following corollary expresses the theorem's initialization
condition in terms of the largest principal angle between $\hat L$ and
$L_*$ and the choice of $\alpha$. In this formulation, the auxiliary
initialization quantities $\kappa_1$ and $\kappa_2$ are absorbed into
explicit bounds and need not be computed separately.

\begin{corollary}[Principal-angle basin condition]\label{cor:angle_basin}

Assume the noiseless data setting of Theorem~\ref{thm:main} and
$0<\gamma<\min\{1,\dssnr\}$. Let $\hat L$ be an initial $d$-subspace
estimate, and let $\theta_1=\angle(\hat L,L_*)$. Suppose that $\sin\theta_1\leq 1/\bigl(2\sqrt{B(\calX,L_*,\gamma)}\bigr)$, where
\[
B(\calX,L_*,\gamma)
=
C_{\mathrm b}\frac{\kappa_{\mathrm{in},*}\calA}{\dssnr}
\left(
\frac{\calA}{\dssnr-\gamma}
+
\frac{2\calR}{\gamma}(1+\kappa_{\mathrm{in},*})
\right).
\]
Initialize STE by $\bSigma^{(0)}=\bP_{\hat{L}}+\alpha\bI$, where the regularization parameter $\alpha>0$ satisfies $\sin^2\theta_1\leq \alpha\leq 1/(4B(\calX,L_*,\gamma))$. Then condition \eqref{eq:kappa1} holds, and hence the STE subspace converges linearly to $L_*$ by Theorem~\ref{thm:main}.
\end{corollary}
Indeed, since $\alpha\geq \sin^2\theta_1$, \eqref{eq:example_k1_first} and \eqref{eq:example_k1_second} give $\kappa_1\geq 1/(4\alpha)$ and $\kappa_2\leq 2$. The upper bound $\alpha\leq 1/(4B(\calX,L_*,\gamma))$ therefore yields $\kappa_1\geq B(\calX,L_*,\gamma)$, which is  sufficient for \eqref{eq:kappa1} after using $\kappa_2\leq 2$. The assumed bound on $\theta_1$ guarantees that the interval of admissible values of $\alpha$ is nonempty. Thus, \eqref{eq:kappa1} can be viewed as an explicit principal-angle basin condition.

For this initialization the basin condition can also be written without the intermediate bounds. With $s=\sin^2\theta_1$, the exact identities \eqref{eq:kappa1_angle} and the formula for $\kappa_2$ in Appendix~\ref{sec:calculate_kappa1} turn \eqref{eq:kappa1} into
\begin{equation}\label{eq:angle_basin_exact}
\frac{\alpha(1+\alpha)}{(s+\alpha)^2}
\geq C_{\mathrm b}\frac{\kappa_{\mathrm{in},*}\calA}{\dssnr}
\left(\frac{\calA}{\dssnr-\gamma}
+\frac{\calR}{\gamma}(1+\kappa_{\mathrm{in},*})\left(1+\frac{s}{\alpha}\right)\right).
\end{equation}

In \eqref{eq:explicit_rate}, the bound decreases by the factor $1/C_0=(\dssnr+\gamma)/(2\dssnr)$ at each iteration. The next corollary shows that replacing $C_{\mathrm b}$ in \eqref{eq:kappa1} by a larger constant gives a smaller factor, down to any value above $\gamma/\dssnr$, and that under the assumptions of Theorem~\ref{thm:main} the asymptotic factor is at most $\gamma/\dssnr$.
\begin{corollary}[Convergence factor]\label{cor:rate}
Assume the setting of Theorem~\ref{thm:main}, and let $1<\rho<\dssnr/\gamma$. If \eqref{eq:kappa1} holds with $C_{\mathrm b}$ replaced by
\begin{equation}\label{eq:C_rho}
C_\rho:=\max\left(26,\ \frac{14\dssnr+8\rho(\dssnr-\gamma)}{\dssnr-\rho\gamma}\right),
\end{equation}
then \eqref{eq:explicit_rate} holds with $C_0$ replaced by $\rho$. Moreover, under the assumptions of Theorem~\ref{thm:main}, the SNR in \eqref{eq:def_kappa_1_2} satisfies
\begin{equation}\label{eq:liminf_snr_growth}
\liminf_{k\to\infty}
\frac{\hat\kappa_1(\bSigma^{(k+1)})}{\hat\kappa_1(\bSigma^{(k)})}
\geq\frac{\dssnr}{\gamma},
\end{equation}
and consequently $\limsup_{k\to\infty}\big(\sin\angle(L^{(k)},L_*)\big)^{1/k}\leq \gamma/\dssnr$.
\end{corollary}
Taking $\rho=C_0$ in \eqref{eq:C_rho} gives $C_\rho=C_{\mathrm b}$ and recovers Theorem~\ref{thm:main}. The proof is given at the end of Section~\ref{sec:prove_reduce_noiseless-thm}. The corollary exhibits a trade-off in the choice of the shrinkage parameter: a smaller $\gamma$ gives a smaller asymptotic factor $\gamma/\dssnr$, whereas the basin condition \eqref{eq:kappa1} becomes more demanding through its term $\kappa_2\calR/\gamma$.

\paragraph{Sampling-based candidate generation.}
Corollary~\ref{cor:angle_basin} applies regardless of how the initial subspace is obtained. In a RANSAC-type scheme~\cite{fischler1981random}, random subsets can generate a list of candidate $d$-subspaces, with one STE run on the full dataset initialized from each candidate. Under the corollary's data assumptions, if at least one candidate and its chosen identity shift satisfy the basin condition, at least one trajectory converges linearly to $L_*$. This gives a conditional guarantee for the refinement stage, analogous in structure to locally optimized RANSAC~\cite{ChumMK03}; the probability of generating a suitable candidate and selection among the refined candidates require separate analysis. Appendix~\ref{sec:initialization} gives a complementary application of this candidate-generation-and-refinement strategy, using a SoS algorithm instead of random-subset sampling.

\subsection{Generic Theory for our Noisy Model}
\label{sec:theory_noise}
We now analyze STE under the noisy inlier model described in Section \ref{sec:inlier_outlier_models} with noise parameter $0<\epsilon \le 1$.

\textbf{Preliminary definitions:} We adapt $\bSigma_{\mathrm{in},*}$ to the noisy setting: it denotes the TME solution for the set of projected inliers $\{\bU_{L_*}^\top\bx \mid \bx \in \calX_{\mathrm{in}}\} \subset \R^d$.
We define $T_{\mathrm{in}}: S_{++}(d) \to S_{+}(d)$ as the TME iteration restricted to the projected inliers:
\begin{equation*}
T_{\mathrm{in}}(\bSigma) = \frac{d}{|\calX_{\mathrm{in}}|}\sum_{\bx\in\calX_{\mathrm{in}}}\frac{(\bU_{L_*}^\top\bx)(\bU_{L_*}^\top\bx)^\top}{(\bU_{L_*}^\top\bx)^\top\bSigma^{-1}(\bU_{L_*}^\top\bx)}.
\end{equation*}

\textbf{The expansion constant $\boldsymbol{C_E}$:}
Under the projected-inlier assumptions of Lemma~\ref{lemma:TME_contraction} below, there exists $0<C_E\leq1$ such that for all $\bSigma\in S_{++}(d)$:
\begin{equation}
\label{eq:noisy_assumption1}
\frac{\sigma_d(\bSigma_{\mathrm{in},*}^{-0.5}T_{\mathrm{in}}(\bSigma)\,\bSigma_{\mathrm{in},*}^{-0.5})}{\sigma_d(\bSigma_{\mathrm{in},*}^{-0.5}\bSigma\,\bSigma_{\mathrm{in},*}^{-0.5})} \geq 1 + C_E \cdot \left(1- \frac{\sigma_d(\bSigma_{\mathrm{in},*}^{-0.5}\bSigma\,\bSigma_{\mathrm{in},*}^{-0.5})}{\sigma_1(\bSigma_{\mathrm{in},*}^{-0.5}\bSigma\,\bSigma_{\mathrm{in},*}^{-0.5})}\right).
\end{equation}
$C_E$ quantifies the expansion of the smallest eigenvalue (which improves numerical conditioning). The existence of $C_E \ge 0$ follows from Lemma 2.1 of \cite{tyler1987distribution} or Theorem 4.4 in \cite{sra_hosseini15}. The existence of a strictly positive $C_E$ is summarized below (proof in Appendix~\ref{sec:proof_lemmas}).

\begin{lemma}\label{lemma:TME_contraction}
If the TME solution $\bSigma_{\mathrm{in},*}$ of the projected inliers is nonsingular and $\{\bU_{L_*}^\top\bx:\bx\in\calX_{\mathrm{in}}\}$ is not contained in the union of two $(d-1)$-subspaces of $\R^d$, then there exists $0<C_E\leq1$ such that \eqref{eq:noisy_assumption1} holds.
\end{lemma}

\textbf{Definition of $\boldsymbol{\kappa_3}$:} Using $\bSigma' = \bSigma^{(0)}_{L_*,L_*^\perp}\bSigma^{(0)\,-1}_{L_*^\perp,L_*^\perp}\bSigma^{(0)}_{L_*^\perp,L_*}$, we define:
\begin{equation*}
\kappa_3 := \frac{\sigma_1(\bSigma^{(0)}_{L_*,L_*})}{\sigma_d(\bSigma^{(0)}_{L_*,L_*} - \bSigma')}.
\end{equation*}
Like $\kappa_1$ and $\kappa_2$, this depends on the initialization relative to $L_*$. It controls the conditioning of the signal block relative to its Schur complement. For $\bSigma^{(0)}=\bI$, $\kappa_3=1$. For the initialization in \eqref{eq:STE_init}, Appendix~\ref{sec:calculate_kappa1} shows that if $\sin^2\theta_1 \leq \alpha$, then $\kappa_3 \leq 2$.

\textbf{Theorem Formulation:}  Our next result states that STE approximates $L_*$ with an error of order $O(\epsilon)$. The theorem relies on constants $C_{\mathrm b}$, $C_{\kappa_1}$, and $C_{\kappa_3}$; we provide explicit formulas for these constants that ensure the validity of the result in the beginning of Appendix~\ref{sec:proof_thm_noisy}.
\begin{theorem}\label{thm:noisy}
Assume STE with $0 < \gamma < 1$ and initialization $\bSigma^{(0)}\in S_{++}(D)$ is applied to a noisy inlier--outlier dataset with $\dssnr > \gamma$ and noise parameter $0<\epsilon \leq 1$. Assume that the projected inliers satisfy the assumptions of Lemma~\ref{lemma:TME_contraction}, so that $0<C_E\leq 1$ in \eqref{eq:noisy_assumption1} is well defined, and that \eqref{eq:kappa1} holds with $C_{\mathrm b}$ given by \eqref{eq:Cvalue_noisy_final} instead of \eqref{eq:C_value}. Let $C_{\kappa_1}>0$ and $C_{\kappa_3}>0$ be the constants specified in \eqref{eq:C_kappa_1} and \eqref{eq:c_kappa3}, and assume that $\kappa_2\le 7$, $\kappa_3\le C_{\kappa_3}/\kappa_{\mathrm{in},*}$ and $\kappa_1\leq C_{\kappa_1}/\kappa_{\mathrm{in},*}$. Then there exists an integer $k_0$ with $1\leq k_0\leq \lfloor \log_{C_0}(C_{\kappa_1}/\kappa_1)\rfloor+1$ such that the following two properties hold:
\begin{enumerate}[(a)]
\item for every $1\leq k\leq k_0$,
\[
\sin\angle(L^{(k)},L_*)\leq \frac{10\,\calA\,\kappa_{\mathrm{in},*}}{7\,\dssnr\,\kappa_1}\,C_0^{-(k-1)};
\]
\item
\[
\sin\angle(L^{(k_0)},L_*)
\le
\frac{\kappa_{\mathrm{in},*}}{C_{\kappa_1}}\Big(\frac{13\,C_{\kappa_3}\,\calR}{\gamma}+\calA\Big).
\]
\end{enumerate}
Moreover, the choice of $C_{\kappa_1}$ in
\eqref{eq:C_kappa_1} satisfies
$C_{\kappa_1}=\Theta(\epsilon^{-1})$ as $\epsilon\downarrow0$
when the remaining problem parameters are fixed.
Consequently,
$\sin\angle(L^{(k_0)},L_*)=O(\epsilon).$\end{theorem}

\textbf{On the form of the guarantee:} Theorem~\ref{thm:noisy} gives a geometrically decaying principal-angle bound through iteration $k_0$ and $O(\epsilon)$ accuracy at that iteration. The proof tracks the iteration-dependent SNR $\hat\kappa_1$, which first exceeds $C_{\kappa_1}/\sigma_1(\bSigma_{\mathrm{in},*})$ at $k_0$. Below this threshold, the noisy inliers are a controlled perturbation of exact inliers. For the initialization \eqref{eq:STE_init} with $\alpha\geq\sin^2\theta_1$, the conditions $\kappa_2\leq7$ and $\kappa_3\leq C_{\kappa_3}/\kappa_{\mathrm{in},*}$ hold automatically: their left sides are at most $2$, whereas $C_{\kappa_3}/\kappa_{\mathrm{in},*}\geq26$.
The upper bound on $\kappa_1$ separates two proof regimes. Initializations satisfying this bound reach $O(\epsilon)$ accuracy by Theorem~\ref{thm:noisy}; the following proposition shows that initializations exceeding it already achieve $O(\epsilon)$ accuracy after one STE step, under the remaining assumptions.

\begin{proposition}[One-step accuracy above the noisy threshold]\label{prop:noisy_high_init}
Assume all hypotheses of Theorem~\ref{thm:noisy} except its upper bound on $\kappa_1$, and instead assume $\kappa_1>C_{\kappa_1}/\kappa_{\mathrm{in},*}$.
Then there is a constant $M>0$ such that $\sin\angle(L^{(1)},L_*)\leq M\epsilon$. Here $M$ depends only on $D,d,\gamma,\dssnr,\kappa_{\mathrm{in},*},\calA$, and $C_E$. With these quantities fixed, it is independent of $\epsilon$ and the initialization. An explicit admissible value of $M$ is given in the proof in Appendix~\ref{sec:proof_noisy_high_init}.
\end{proposition}

Theorem~\ref{thm:noisy} and Proposition~\ref{prop:noisy_high_init} remain valid for any partition of $\calX$ in which the inliers satisfy the cone bound and the outliers lie outside $L_*$, provided $\calA$ and $\calR$ are finite.

\subsection{Implications for the Generalized Haystack Model}
\label{sec:haystack}
We analyze the basin condition and TME initialization under the asymptotic generalized haystack model. Section~\ref{sec:models} bounds the geometric quantities and relates the certified basin to initialization accuracy. Section~\ref{sec:models_theorem} shows that TME can initialize STE successfully when $\dssnr<1$, even when TME alone does not recover the target subspace.

The generalized haystack model is due to~\cite{maunu19ggd}. We assume Gaussian components and keep the dimensions and covariance matrices fixed. Let $\{\bx_i^{\mathrm{in}}\}_{i\geq1}$ and $\{\bx_j^{\mathrm{out}}\}_{j\geq1}$ be independent infinite sequences, sampled i.i.d.~from $\mathcal N(\bm0,\bSigma^{(\mathrm{in})}/d)$ and $\mathcal N(\bm0,\bSigma^{(\mathrm{out})}/D)$, respectively, where $\bSigma^{(\mathrm{in})}\in S_+(D)$ has rank $d$ and range $L_*$, and $\bSigma^{(\mathrm{out})}\in S_{++}(D)$. The dataset of size $N$ consists of the first $n_1(N)$ inliers and the first $n_0(N)$ outliers, with $n_1(N)+n_0(N)=N$. Assume that
\[
\dssnr_N:=\frac{n_1(N)/d}{n_0(N)/(D-d)}
\longrightarrow\dssnr\in(0,\infty)\qquad(N\to\infty).
\]
Theorem~\ref{thm:haystack} below addresses the regime $\dssnr<1$, in which the limiting fraction of inliers is below $d/D$. In this subsection, population statements use the limiting value $\dssnr$, while each finite-dataset application of Theorem~\ref{thm:main} uses $\dssnr_N$.

We define the following condition numbers of inliers (in $L_*$) and outliers:
\begin{equation}
\kappa_{\mathrm{in}}={\sigma_1(\bSigma^{(\mathrm{in})})}/{\sigma_d(\bSigma^{(\mathrm{in})})}\,\,\ \ \text{ and } \ \kappa_{\mathrm{out}}={\sigma_1(\bSigma^{(\mathrm{out})})}/{\sigma_D(\bSigma^{(\mathrm{out})})}\,.
\label{eq:kappa_in_out}
\end{equation}
Almost surely, $\kappa_{\mathrm{in},*}\to\kappa_{\mathrm{in}}$ as $N\to\infty$; for finite samples they generally differ.

\subsubsection{Population bounds and basin scaling} \label{sec:models}
We analyze the RHS of \eqref{eq:kappa1} under the asymptotic generalized haystack model with $D-d > 2$, and assuming the initialization in \eqref{eq:STE_init} with $\alpha \geq \sin^2\theta_1$. We prove in Appendix~\ref{sec:proof_example_haystack} that
\begin{equation}\label{eq:calA_R}
\lim_{n_0\to\infty}{\calA} \leq \kappa_{\mathrm{out}} \cdot \frac{D-d}{D-d-2}
\quad \text{ and } \quad
\lim_{n_0\to\infty}{\calR} \leq \kappa_{\mathrm{out}}^2 \cdot \frac{D}{D-d-2}.
\end{equation}

We first take the large-sample limit with $d$, $D$, and $\dssnr$ fixed. After obtaining the corresponding population bounds, we examine how the certified basin depends on $\dssnr$ across a family of models in which $\kappa_{\mathrm{in}}$ and $\kappa_{\mathrm{out}}$ remain bounded. Define
$\bar{\calA}:=\kappa_{\mathrm{out}}(D-d)/(D-d-2)$ and
$\bar{\calR}:=\kappa_{\mathrm{out}}^2D/(D-d-2)$.
Fix $0<c_*<1/2$ and choose $0<\gamma<1$ so that
$c_*\dssnr\leq \gamma\leq (1-c_*)\dssnr$.
Then $\gamma\geq c_*\dssnr$, $\dssnr-\gamma\geq c_*\dssnr$, and the constant $C_{\mathrm b}$ in \eqref{eq:C_value} satisfies $C_{\mathrm b}\leq(28+2c_*)/c_*$. Consequently, if $B(\calX,L_*,\gamma)$ is the quantity in Corollary~\ref{cor:angle_basin}, then along the generalized haystack sequence,
\begin{equation}\label{eq:haystack_B_bound}
\limsup_{N\to\infty}\dssnr^2 B(\calX,L_*,\gamma)
\leq
K_*:=\frac{\kappa_{\mathrm{in}}\bar{\calA}}{c_*}
\frac{28+2c_*}{c_*}
\left(\bar{\calA}+2\bar{\calR}(1+\kappa_{\mathrm{in}})\right).
\end{equation}
Thus, for every $K>K_*$ and all sufficiently large $N$,
$B(\calX,L_*,\gamma)\leq K(\dssnr)^{-2}$.
Corollary~\ref{cor:angle_basin} therefore applies whenever
\begin{equation}\label{eq:haystack_angle_basin}
\sin\theta_1\leq \frac{\dssnr}{2\sqrt K},
\end{equation}
and one may choose any regularization parameter satisfying
$\sin^2\theta_1\leq \alpha\leq (\dssnr)^2/(4K)$.
In particular, the certified principal-angle basin scales linearly in $\dssnr$, while the admissible regularization scale is quadratic in $\dssnr$.

\textbf{On the scale of the basin:} By \eqref{eq:haystack_angle_basin}, initial subspaces with $\sin\theta_1$ of order $\dssnr$ are admissible when the dimensions, $\kappa_{\mathrm{in}}$ and $\kappa_{\mathrm{out}}$ are fixed and $c_*\dssnr\leq\gamma\leq(1-c_*)\dssnr$. We do not know whether this order is optimal.

A possible refinement would use the outliers' cross-alignment,
\[
\left\|\sum_{\bx\in\calX_{\mathrm{out}}}
\frac{\bP_{L_*}\bx\bx^\top\bP_{L_*^\perp}}
{\|\bU_{L_*^\perp}^\top\bx\|^2}\right\|,
\]
which directly measures the interaction between $L_*$ and $L_*^\perp$.
When the outlier distribution is invariant under reflection across
$L_*$, the corresponding population cross-alignment vanishes, provided
the expectation exists. A general outlier covariance in the generalized
haystack model need not have this symmetry. Determining whether
cross-alignment yields a larger certified basin is left to future work.

\subsubsection{The effectiveness of the TME initialization}
\label{sec:models_theorem}
We study STE initialized by the empirical TME, denoted by TME+STE. The theorem separates the sample-size limit from the STE iteration limit: almost surely, every sufficiently large sample gives a successful initialization.
\begin{theorem}\label{thm:haystack}
Consider the generalized haystack samples defined above, with $d\leq(1-\lambda)D-2$ for a fixed $\lambda>0$. For any $0<c_0<1/2$, there exists $\delta_0:=\delta_0(\kappa_{\mathrm{in}},\kappa_{\mathrm{out}},c_0,\lambda)\in(2c_0,1)$ such that, if the limiting ratio satisfies $\delta_0\leq\dssnr<1$, then almost surely the following holds for all sufficiently large $N$: the empirical trace-normalized TME solution exists in $S_{++}(D)$ and, used as $\bSigma^{(0)}$, satisfies \eqref{eq:kappa1} with the finite-sample ratio $\dssnr_N$ for every $\gamma$ satisfying $c_0<\gamma<\delta_0-c_0$. Consequently, for each such $N$ and $\gamma$, the STE subspace converges linearly to $L_*$ as the iteration index tends to infinity. If $\bSigma^{(\mathrm{out})}_{L_*,L_*^\perp}\neq\bm0$, the limiting population TME matrix does not have $L_*$ as its leading $d$-dimensional eigenspace, and the empirical TME subspace cannot converge to $L_*$ as $N\to\infty$.

\end{theorem}
Together with Theorem~\ref{thm:zhang16_extend}, this result describes three regimes in the asymptotic setting:
\begin{itemize}
    \item \textbf{TME recovery ($\dssnr_N\geq1$):} For finite haystack samples satisfying the assumptions of Theorem~\ref{thm:zhang16_extend}, trace-normalized TME recovers $L_*$ as the iteration index tends to infinity, so STE refinement is unnecessary.
    \item \textbf{Intermediate Regime ($\delta_0 \leq \dssnr < 1$):} TME generally does not recover $L_*$, but its full-rank solution provides an initialization from which STE recovers $L_*$. This highlights the benefit of the subspace constraint.
    \item \textbf{Uncertified Regime ($\dssnr < \delta_0$):} TME+STE is not guaranteed to succeed by Theorem~\ref{thm:haystack}. STE may nevertheless recover the subspace if provided with a different initialization satisfying condition~\eqref{eq:kappa1}.
\end{itemize}

The proof requires the two margins $\gamma$ and $\dssnr-\gamma$ to stay uniformly away from zero. The parameter $c_0$ enforces these margins uniformly over the range $\dssnr\geq\delta_0$: indeed, $\gamma>c_0$ and $\dssnr-\gamma\geq\delta_0-\gamma>c_0$. As $c_0\downarrow0$, or as $\kappa_{\mathrm{in}}$ or $\kappa_{\mathrm{out}}$ becomes large, or $\lambda\downarrow0$, the constants in the proof deteriorate and the sufficient estimates may require choosing a cutoff $\delta_0$ closer to $1$. We also remark that the restriction involving $\lambda$ appears to be an artifact of the proof technique rather than a fundamental limitation of the algorithm.

\textbf{An explicit sufficient condition in a block-isotropic model:} Consider the following special case of the model, with $D-d>2$ and $0<\dssnr<1$. Assume that $\bSigma^{(\mathrm{in})}=\bP_{L_*}$, so that the inliers are isotropic within $L_*$ and $\kappa_{\mathrm{in}}=1$ (and $\kappa_{\mathrm{in},*}\to 1$ as $N\to\infty$), and that the outlier covariance satisfies
\begin{equation}\label{eq:block_isotropic}
\bSigma^{(\mathrm{out})}_{L_*,L_*}-\bSigma^{(\mathrm{out})}_{L_*,L_*^\perp}\big(\bSigma^{(\mathrm{out})}_{L_*^\perp,L_*^\perp}\big)^{-1}\bSigma^{(\mathrm{out})}_{L_*^\perp,L_*}=a\,\bI_d \quad\text{and}\quad \bSigma^{(\mathrm{out})}_{L_*^\perp,L_*^\perp}=b\,\bI_{D-d}
\end{equation}
for some $a,b>0$, while the off-diagonal block $\bSigma^{(\mathrm{out})}_{L_*,L_*^\perp}\neq\bm{0}$ is arbitrary. The detailed calculation is given in Appendix~\ref{sec:block_isotropic_details}. It shows that, in the limit $N\to\infty$, for $0<\gamma<\dssnr<1$ and $D-d>2$, TME+STE recovers $L_*$ whenever
\begin{equation}\label{eq:explicit_threshold}
\frac{a}{b}\cdot\frac{D-(1-\dssnr)(d+2)}{(1-\dssnr)(D-d)}\ \geq\ C_{\mathrm b}\,\frac{\bar{\calA}}{\dssnr}\Big(\frac{\bar{\calA}}{\dssnr-\gamma}+\frac{2(1+\kappa_{\mathrm{out}})^2\bar{\calR}}{\gamma}\Big),
\end{equation}
with $C_{\mathrm b}$ given by \eqref{eq:C_value}. For fixed $\gamma$ and $c_0>0$, the right-hand side of \eqref{eq:explicit_threshold} is bounded on $[\gamma+c_0,1)$, while the left-hand side grows like $(1-\dssnr)^{-1}$. Hence there is a $\delta_0<1$ such that \eqref{eq:explicit_threshold} holds for all $\dssnr\in[\delta_0,1)$, giving a sufficient cutoff for this fixed $\gamma$. Here, the TME initialization ratio takes the form $\kappa_1 = (a/b)\omega(\dssnr)$, where $\omega(\dssnr)$ is a dimensionless root derived from the TME population fixed-point equation (Appendix~\ref{sec:block_isotropic_details}). The ratio $a/b$ compares the outliers' variance within $L_*$, after removing the component explained by $L_*^\perp$, to their variance in $L_*^\perp$. The lower bound in \eqref{eq:explicit_threshold} arises by substituting a closed-form lower bound on $\omega(\dssnr)$ into $\kappa_1$, together with upper bounds on the remaining geometric parameters.

 \section{Proofs of the Main Results}
\label{sec:supp_theory}
We prove the main results below and defer several auxiliary proofs to the appendices.

\subsection{Proof of Theorem~\ref{thm:main}}
\label{proof:main_noiseless}
The proof is organized as follows:
Section~\ref{sec:notation_thm1} lists definitions and notation;
Section~\ref{sec:tech_lemma_thm1} presents several technical estimates used throughout the proof;
Section~\ref{sec:reduce_noiseless-thm} reduces the proof to verification of a technical proposition; and
Section~\ref{sec:prove_reduce_noiseless-thm} proves this reduction.

The proof decomposes each iterate using a Schur complement and tracks the relative strength of its blocks along $L_*$ and $L_*^\perp$. This ratio grows by at least $C_0$ at each iteration, while the required conditioning bounds remain valid. A perturbation estimate converts this growth into the principal-angle bound in \eqref{eq:explicit_rate}.

\subsubsection{Notation and Definitions}
\label{sec:notation_thm1}
We denote by $\psdleq$ the Loewner order: For $\bX$, $\bX' \in \R^{D \times D}$,  $\bX\psdleq\bX'$ if and only if $\bX'-\bX$ is p.s.d. We denote by
$\sigma_i(\cdot)$ the $i$-th eigenvalue of a corresponding matrix. We also use $\bu_i(\cdot)$ to denote the $i$-th eigenvector of a corresponding matrix. When the matrix is clear from the context, we may just write $\sigma_i$ and $\bu_i$.

If some eigenvalues coincide, we choose the eigenvectors by any fixed rule that depends only on the eigenspaces. For the iterates in our guarantees, $\sigma_d>\sigma_{d+1}$, so their leading $d$-dimensional eigenspaces are unique.
Let $\Pi_d(\bSigma)$ denote the best approximation of $\bSigma$ of rank at most $d$ in Frobenius norm:  $$\Pi_d(\bSigma)=\sum_{i=1}^d\bu_i(\bSigma)\sigma_i(\bSigma)\bu_i(\bSigma)^\top$$
and $\bP_{\mathrm{tail}}$ denote the projector onto the span of the  smallest $D-d$ eigenvectors of $\bSigma$:
$$\bP_{\mathrm{tail}}(\bSigma)=\sum_{i=d+1}^D\bu_i(\bSigma)\bu_i(\bSigma)^\top.$$
We use boldface notation for the operator $\bP$, since we identify it with its matrix representation. On the other hand, $\Pi$ is not a linear operator.
We recall our notation $\bar{\sigma}_{\mathrm{tail}}(\bSigma)=\frac{1}{D-d}\sum_{i=d+1}^D\sigma_i(\bSigma)$ for the average of the smallest $D-d$ eigenvalues of $\bSigma$.
We further denote
$$\Pi_{\mathrm{tail}} (\bSigma)=\sum_{i=d+1}^D\bu_i(\bSigma)\sigma_i(\bSigma)\bu_i(\bSigma)^\top.$$

Throughout the rest of the paper we propose definitions for $\bSigma \in S_{+}$, where these definitions involve expressions with $\bSigma^{-1}$ or $\bSigma_{L_*^\perp,L_*^\perp}^{-1}$, which may not exist. In such cases, one should replace $\bSigma^{-1}$  with $(\bSigma + \epsilon_0 \bI)^{-1}$ and take its limit as $\epsilon_0 \to 0$. We similarly replace $\bSigma_{L_*^\perp,L_*^\perp}^{-1}$. For example, if $\bSigma$ is singular with range $L$, then
$$
\frac{\bx\bx^\top}{\bx^\top\bSigma^{-1}\bx} = \begin{cases}
 \frac{\bx\bx^\top}{\bx^\top\bU_L(\bU_L^\top\bSigma\bU_L)^{-1}\bU_L^\top\bx},\,\,&\text{if $\bx\in L$;}\\
 0,\,\,&\text{if $\bx\not\in L$}.
\end{cases}
$$

We define the following operators from $S_{+}$ to $S_+$:
\begin{align*}
    &T_1(\bSigma)  = \sum_{\bx\in\calX}\frac{\bx\bx^\top}{\bx^\top\bSigma^{-1}\bx}, \quad
    T_2(\bSigma;\gamma)  = \Pi_d(\bSigma)+\gamma  \bar{\sigma}_{\mathrm{tail}}(\bSigma)  \bP_{\mathrm{tail}}(\bSigma), \ \text{ and} \\
&T(\bSigma)=T_2(T_1(\bSigma),\gamma).
\end{align*}
The STE Algorithm is a fixed-point iteration of $T$ up to a scaling factor (see \eqref{eq:ste_prelim_iteration}-\eqref{eq:ste_iteration}). Since we only care about the span of the top $d$ eigenvectors of $T(\bSigma)$, which is invariant to scaling, we can ignore the scaling factor in our analysis.

We further define
\begin{equation}
    \label{eq:Sigma+}
\bSigma_{+,\mathrm{in}}=\sum_{\bx\in\calX_{\mathrm{in}}}\frac{\bx\bx^\top}{\bx^\top\bSigma^{-1}\bx},
\quad
\bSigma_{+,\mathrm{out}}=\sum_{\bx\in\calX_{\mathrm{out}}}\frac{\bx\bx^\top}{\bx^\top\bSigma^{-1}\bx}
\quad \text{and} \quad
\bSigma_+=\bSigma_{+,\mathrm{in}}+\bSigma_{+,\mathrm{out}}.
\end{equation}
We note that
\begin{equation*}\bSigma_+=T_1(\bSigma) \ \text{ and } \ T(\bSigma)=T_2(\bSigma_+,\gamma).\end{equation*}

We also define $g_2: S_+\rightarrow S_+$ and $g_1: S_+\rightarrow S_+$ by
\begin{equation}
\label{eq:defg2}g_2(\bSigma)=\begin{pmatrix}
\bSigma_{L_*,L_*^\perp}\bSigma^{-1}_{L_*^\perp,L_*^\perp}\bSigma_{L_*^\perp,L_*} & \hspace{0.05in} \bSigma_{L_*,L_*^\perp}  \\
\bSigma_{L_*^\perp,L_*} & \hspace{0.05in} \bSigma_{L_*^\perp,L_*^\perp}
\end{pmatrix}\equiv \begin{pmatrix}
 \bSigma_{L_*,L_*^\perp} \\
 \bSigma_{L_*^\perp,L_*^\perp}
\end{pmatrix}\Big(\bSigma^{-1}_{L_*^\perp,L_*^\perp}\bSigma_{L_*^\perp,L_*},\bI\Big),\hspace{0.2in}
\end{equation}
\begin{equation}
    \label{eq:defg1}
g_1(\bSigma)=\bSigma-g_2(\bSigma)=\begin{pmatrix}
\bSigma_{L_*,L_*}-\bSigma_{L_*,L_*^\perp}\bSigma^{-1}_{L_*^\perp,L_*^\perp}\bSigma_{L_*^\perp,L_*} & \hspace{0.05in} 0 \\
0 & \hspace{0.05in} 0
\end{pmatrix}.
\end{equation}

The reader may remember these notations as follows: $g_1(\bSigma)$ is a matrix whose top left block is the Schur complement of $\bSigma_{L_*^\perp,L_*^\perp}$ in $\bSigma$  and the rest of its blocks are zero, whereas $g_2(\bSigma)=\bSigma-g_1(\bSigma)$.
Recalling our above convention, we note, for example, that
\[
\text{if } \ \bSigma=\begin{pmatrix}
 \bSigma_{L_*,L_*} & 0\\
0& 0
\end{pmatrix}, \ \text{ then } \ g_1(\bSigma)=\bSigma \ \text{ and }
\ g_2(\bSigma)=0.
\]

We recall the definition of  $\bSigma_{\mathrm{in},*}$ in \eqref{eq:def_sigmainstar}, and define
\begin{equation}
\label{eq:def_f}
h(\bSigma)=\sigma_d(\bSigma_{\mathrm{in},*}^{-0.5}[g_1(\bSigma)]_{L_*,L_*}\bSigma_{\mathrm{in},*}^{-0.5}).
\end{equation}

For $\bSigma\in S_+(D)$ with positive denominators, we define
\begin{equation}
\label{eq:def_kappa_1_2}
\hat\kappa_1(\bSigma)=\frac{h(\bSigma)}{\sigma_1(\bSigma_{L_*^\perp,L_*^\perp})},
\qquad
\hat\kappa_2(\bSigma)=\frac{\sigma_1(\bSigma_{L_*^\perp,L_*^\perp})}{\sigma_D(\bSigma)}.
\end{equation}
Both ratios are finite on $S_{++}(D)$, and we only evaluate them at positive definite matrices, since all STE iterates are positive definite (see Section~\ref{sec:reduce_noiseless-thm}).
We note that
$
\kappa_2=\hat{\kappa}_2(\bSigma^{(0)})$. Furthermore, applying Lemma~\ref{lemma:matrixeigenvalue} to the $d\times d$ signal block $[g_1(\bSigma^{(0)})]_{L_*,L_*}$ gives
\begin{equation}\label{eq:hatkappa1_relation}
\frac{\kappa_1}{\sigma_d(\bSigma_{\mathrm{in},*})}
\geq\hat\kappa_1(\bSigma^{(0)})
\geq\frac{\kappa_1}{\sigma_1(\bSigma_{\mathrm{in},*})}.
\end{equation}
When $\|\bSigma_{+,\mathrm{out}}\|>0$, we also define the spectral dominance ratio
\begin{equation}
\label{eq:def_nu}
\nu(\bSigma)=\frac{\sigma_d(\bSigma_{+,\mathrm{in}})}{\|\bSigma_{+,\mathrm{out}}\|}.
\end{equation}

\subsubsection{Auxiliary Propositions}
\label{sec:tech_lemma_thm1}
We formulate and prove several results that we will use in the main proof.
The first lemma provides some basic properties of $g_1(\bSigma)$ and $g_2(\bSigma)$.
\begin{lemma}
\label{lemma:g1}
The following properties hold:\\
(a) If $\bSigma \in S_+$, then $g_1(\bSigma) \in S_+$,  $g_2(\bSigma) \in S_+$, and consequently, $g_1(\bSigma)\psdleq \bSigma$. Moreover,
$\rank(g_1(\bSigma))\leq d$ and $\rank(g_2(\bSigma))\leq D-d$. If $\,\bSigma \in S_{++}$, then $\rank(g_1(\bSigma))= d$ and $\rank(g_2(\bSigma))= D-d$.
Therefore,  $\sigma_d(g_1(\bSigma))>0$ and ${\kappa}_1>0$ if $\bSigma^{(0)} \in S_{++}$.\\
(b) If $\bX$, $\bSigma \in S_+$ and the range of $\bSigma-\bX$ is  contained in $L_*$, then $g_1(\bSigma)\psdgeq\bSigma-\bX$.
\end{lemma}

The next lemma provides matrix bounds that will be used in our analysis.
\begin{lemma}\label{lemma:xplus}
For $\bSigma\in S_+$, with $\bSigma_{+,\mathrm{in}}$, $\bSigma_{+,\mathrm{out}}$ and $\bSigma_+$ defined in \eqref{eq:Sigma+}, the following  hold:
\begin{align}\label{eq:xin}
h(\bSigma_{+,\mathrm{in}})\geq \frac{n_1}{d}h(\bSigma),
\end{align}
\begin{align}\label{eq:xout1}
\bSigma_{+,\mathrm{out}}\psdleq \sigma_1(\bSigma_{L_*^{\perp},L_*^{\perp}})\sum_{\bx\in\calX_{\mathrm{out}}}\frac{\bx\bx^\top}{\|\bU_{L_*^{\perp}}^\top\bx\|^2},
\end{align}
and consequently,
\begin{align}\label{eq:xout2}\|\bSigma_{+,\mathrm{out}}\|\leq \tilde{\calA}\sigma_1(\bSigma_{L_*^{\perp},L_*^{\perp}}) \ \text{ and } \ \tr([\bSigma_{+,\mathrm{out}}]_{L_*^{\perp},L_*^{\perp}})\leq  {n_0}\sigma_1(\bSigma_{L_*^{\perp},L_*^{\perp}}), \end{align}
where $\tilde{\calA}=\frac{n_0}{D-d}\calA$ was defined in Section~\ref{sec:additional_quantities}.
Furthermore,
\begin{align}\label{eq:xout3}
\bSigma_{+}\psdgeq\sigma_D(\bSigma)\sum_{\bx\in\calX}\frac{\bx\bx^\top}{\|\bx\|^2}.
\end{align}
\end{lemma}

\begin{proof}[Proof of Lemma~\ref{lemma:xplus}]
\label{sec:proof_lemma_3}
We prove it in four parts according to the four equations.\\
(a) Proof of \eqref{eq:xin}. The definition of $g_1$ implies that $[g_1(\bSigma_{+,\mathrm{in}})]_{L_*,L_*}=[\bSigma_{+,\mathrm{in}}]_{L_*,L_*}$ and thus   $h(\bSigma_{+,\mathrm{in}})=\sigma_d(\bSigma_{\mathrm{in},*}^{-0.5}[g_1(\bSigma_{+,\mathrm{in}})]_{L_*,L_*}\bSigma_{\mathrm{in},*}^{-0.5})=\sigma_d(\bSigma_{\mathrm{in},*}^{-0.5}[\bSigma_{+,\mathrm{in}}]_{L_*,L_*}\bSigma_{\mathrm{in},*}^{-0.5})$.
Therefore, to conclude \eqref{eq:xin} we will prove that $\bSigma_{\mathrm{in},*}^{-0.5}[\bSigma_{+,\mathrm{in}}]_{L_*,L_*}\bSigma_{\mathrm{in},*}^{-0.5}
\psdgeq h(\bSigma)\frac{n_1}{d}\bI$.

Recall that  $\bSigma_{+,\mathrm{in}}=\sum_{\bx\in\calX_{\mathrm{in}}}{\bx\bx^\top}/{\bx^\top\bSigma^{-1}\bx}$, and denote
$\tilde{\calX}_{\mathrm{in}}=\{\bU_{L_*}^\top\bx:\bx\in\calX_{\mathrm{in}}\}\subset\R^d$. We derive the following relationships, which we clarify below:
\begin{align*}
&\bSigma_{\mathrm{in},*}^{-0.5}[\bSigma_{+,\mathrm{in}}]_{L_*,L_*}\bSigma_{\mathrm{in},*}^{-0.5}=\sum_{\bx\in\calX_{\mathrm{in}}}\frac{\bSigma_{\mathrm{in},*}^{-0.5}[\bx\bx^\top]_{L_*,L_*}\bSigma_{\mathrm{in},*}^{-0.5}}{\bx^\top(\bSigma)^{-1}\bx}\\
\nonumber
\psdgeq& \sum_{\bx\in\calX_{\mathrm{in}}}\!\frac{\bSigma_{\mathrm{in},*}^{-0.5}[\bx\bx^\top]_{L_*,L_*}\bSigma_{\mathrm{in},*}^{-0.5}}{\bx^\top(g_1(\bSigma))^{-1}\bx}
=\sum_{\tilde{\bx}\in\tilde{\calX}_{\mathrm{in}}}\frac{\bSigma_{\mathrm{in},*}^{-0.5}\tilde{\bx}\tilde{\bx}^\top\bSigma_{\mathrm{in},*}^{-0.5}}{\tilde{\bx}^\top\bSigma_{\mathrm{in},*}^{-0.5}(\bSigma_{\mathrm{in},*}^{-0.5}[ g_1(\bSigma)]_{L_*,L_*}\bSigma_{\mathrm{in},*}^{-0.5})^{-1}\bSigma_{\mathrm{in},*}^{-0.5}\tilde{\bx}}
\\\psdgeq& h(\bSigma)\sum_{\tilde{\bx}\in\tilde{\calX}_{\mathrm{in}}}\frac{\bSigma_{\mathrm{in},*}^{-0.5}\tilde{\bx}\tilde{\bx}^\top\bSigma_{\mathrm{in},*}^{-0.5}}{\tilde{\bx}^\top\bSigma_{\mathrm{in},*}^{-1}\tilde{\bx}}=h(\bSigma)\frac{n_1}{d}\bI. \nonumber
\end{align*}
To verify the first Loewner order (which in fact holds with equality), we note that for $\bx\in L_*$, $\bx^\top\bSigma^{-1}\bx=(\bU_{L_*}^\top\bx)^\top[\bSigma^{-1}]_{L_*,L_*}(\bU_{L_*}^\top\bx)$, and that $[\bSigma^{-1}]_{L_*,L_*}=([g_1(\bSigma)]_{L_*,L_*})^{-1}$ by the formula for the inverse of a $2\times 2$ block matrix~\cite[p.~86]{bhatia2009positive}; by our convention for the inverse of the singular matrix $g_1(\bSigma)$, whose range is $L_*$, the latter quantity equals $\bx^\top(g_1(\bSigma))^{-1}\bx$.
The second equality follows from our convention of taking the inverse of a singular matrix and the facts that $g_1(\bSigma)$ is a singular matrix with range $L_*$, and $\bx\in\calX_{\mathrm{in}}$ lie in $L_*$.
The second Loewner order  follows from  $\bSigma_{\mathrm{in},*}^{-0.5}[g_1(\bSigma)]_{L_*,L_*}\bSigma_{\mathrm{in},*}^{-0.5}\psdgeq h(\bSigma)\bI$, which is a direct consequence of the definition $h(\bSigma) := \sigma_d(\bSigma_{\mathrm{in},*}^{-0.5}[g_1(\bSigma)]_{L_*,L_*}\bSigma_{\mathrm{in},*}^{-0.5})$.
To derive the last equality, we use the assumed existence of the positive definite inlier TME $\bSigma_{\mathrm{in},*}$ for $\tilde{\calX}_{\mathrm{in}}$.
We then use \eqref{eq:TME_definition}
for our setting with the TME solution $\bSigma_{\mathrm{in},*}$ and dataset $\tilde{\calX}_{\mathrm{in}}$ of $n_1$ points in a subspace of ambient dimension $d$.

(b) Proof of \eqref{eq:xout1}. We first present the main arguments of the proof and then justify each step below:
\begin{align*}
\bSigma_{+,\mathrm{out}}&=\sum_{\bx\in\calX_{\mathrm{out}}}\frac{\bx\bx^\top}{\bx^\top\bSigma^{-1}\bx}\psdleq \lim_{t\rightarrow\infty} \sum_{\bx\in\calX_{\mathrm{out}}}\frac{\bx\bx^\top}{\bx^\top(t\bP_{L_*}+\bSigma)^{-1}\bx} =\sum_{\bx\in\calX_{\mathrm{out}}}\frac{\bx\bx^\top}{\bx^\top\bU_{L_*^\perp}\bSigma_{L_*^\perp,L_*^\perp}^{-1}\bU_{L_*^\perp}^\top\bx}\\
\nonumber
\psdleq &
\sum_{\bx\in\calX_{\mathrm{out}}}\frac{\bx\bx^\top}{\bx^\top\bU_{L_*^\perp}\frac{\bI}{\sigma_1(\bSigma_{L_*^{\perp},L_*^{\perp}})}\bU_{L_*^\perp}^\top\bx}
=\sigma_1(\bSigma_{L_*^{\perp},L_*^{\perp}})\sum_{\bx\in\calX_{\mathrm{out}}}\frac{\bx\bx^\top}{\|\bU_{L_*^\perp}^\top\bx\|^2}.
\end{align*}
The first equality follows from the definition of $\bSigma_{+,\mathrm{out}}$.  The first Loewner order follows from the fact that $t\bP_{L_*}+\bSigma\psdgeq \bSigma$ for $t>0$ and thus $\bx^\top\bSigma^{-1}\bx\geq \bx^\top(t\bP_{L_*}+\bSigma)^{-1}\bx$.
The second equality follows from observing that $\lim_{t\rightarrow\infty}(t\bP_{L_*}+\bSigma)^{-1}=\bU_{L_*^\perp}\bSigma_{L_*^\perp,L_*^\perp}^{-1}\bU_{L_*^\perp}^\top$. Indeed, by letting $t \to \infty$, the three blocks  associated with $(L_*,L_*)$, $(L_*,L_*^\perp)$ and $(L_*^\perp,L_*)$ zero out when taking the inverse, where a rigorous proof can be obtained by using the formula for the inverse of a block matrix (see e.g., (2.2) of \cite{LU2002119}).

(c) Proof of \eqref{eq:xout2}. We  conclude it by applying the norm and trace to both sides of \eqref{eq:xout1}.

(d) Proof of \eqref{eq:xout3}. Using the fact  $\bSigma\psdgeq \sigma_D(\bSigma)\bI$,
\begin{align*}
&\bSigma_{+}=\sum_{\bx\in\calX}\frac{\bx\bx^\top}{\bx^\top\bSigma^{-1}\bx}\psdgeq \sum_{\bx\in\calX}\frac{\bx\bx^\top}{\bx^\top(\sigma_D(\bSigma)\bI)^{-1}\bx} =\sigma_D(\bSigma)\sum_{\bx\in\calX}\frac{\bx\bx^\top}{\|\bx\|^2}.
\qedhere
\end{align*}
\end{proof}

The next lemma summarizes two basic and well-known properties of matrices.
\begin{lemma}\label{lemma:matrixperturbation}
The following two properties hold:\\
(a) For any semi-orthogonal matrix $\bU_0\in O(D,k)$ and $\bSigma \in S_+(D)$,
\[
\sigma_1(\bSigma)\geq \sigma_1(\bU_0^\top\bSigma\bU_0)\geq \sigma_{D-k+1}(\bSigma),\hspace{0.2in} \hspace{0.2in} \sigma_{k}(\bSigma)\geq \sigma_k(\bU_0^\top\bSigma\bU_0)\geq \sigma_{D}(\bSigma).
\]
(b) For any two symmetric matrices $\bA, \bB \in \R^{D\times D}$, $|\sigma_i(\bA)-\sigma_i(\bB)| \leq \|\bA-\bB\|$, \ $1 \leq i \leq D$.
\end{lemma}

Classical Davis--Kahan bounds \cite{DavisKahan1970,bhatia1996matrix} compare a matrix with an ideal signal matrix. Here the dimensionally compatible comparison matrix is
\[
\bM=\bU_{L_*}\bSigma_{L_*,L_*}\bU_{L_*}^\top.
\]
Let $\hat L$ be the span of the top $d$ eigenvectors of $\bSigma$. For example, if $\|\bSigma-\bM\|<\sigma_d(\bSigma_{L_*,L_*})$, then $\sin\angle(\hat L,L_*)\leq\|\bSigma-\bM\|/(\sigma_d(\bSigma_{L_*,L_*})-\|\bSigma-\bM\|)$.
The next lemma instead uses only the orthogonal block $\bSigma_{L_*^\perp,L_*^\perp}$, which is convenient for tracking the iterates.

\begin{lemma}\label{lemma:subspace}
Let $\bSigma\in S_+(D)$, let $\hat L$ be the span of its top $d$ eigenvectors, and suppose
$\sigma_d(\bSigma_{L_*,L_*})>0$. Then
\[
\sin\angle(\hat L,L_*)\leq
\min\left\{1,\sqrt{\frac{\sigma_1(\bSigma_{L_*^\perp,L_*^\perp})}{\sigma_d(\bSigma_{L_*,L_*})}}\right\}.
\]
\end{lemma}

For the noiseless argument we use the additional fact that the outlier contribution is positive semidefinite.
\begin{lemma}[Positive semidefinite perturbation bound]\label{lemma:psd_angle}
Let $\bA\in S_+(D)$ have rank $d$ and range $L$, and let $\bB\in S_+(D)$ satisfy $\|\bB\|\leq\sigma_d(\bA)/2$. Then the leading $d$-subspace $\hat L$ of $\bA+\bB$ is unique, $\angle(\hat L,L)<\pi/4$ and $\sin(2\angle(\hat L,L))\leq \|\bB\|/\sigma_d(\bA)$.
\end{lemma}
\begin{proof}
Since $\bB\psdgeq0$, $\sigma_d(\bA+\bB)\geq \sigma_d(\bA)$, while $\sigma_{d+1}(\bA+\bB)\leq \|\bB\|$ by Weyl's inequality. Thus the leading $d$-subspace is unique. Choose an orthonormal basis $\bU$ of that subspace consisting of eigenvectors, and let $\boldsymbol\Lambda$ be the diagonal matrix of its eigenvalues. Projecting the eigenvector equation onto $L^\perp$ after subtracting $\|\bB\|\bI/2$ gives
\[
\bP_{L^\perp}\bU(\boldsymbol\Lambda-\|\bB\|\bI_d/2)
=\bP_{L^\perp}(\bB-\|\bB\|\bI_D/2)\bU.
\]
The bounds $0\psdleq\bB\psdleq \|\bB\|\bI_D$ and $\boldsymbol\Lambda\psdgeq \sigma_d(\bA)\bI_d$ imply
\[
\sin\angle(\hat L,L)=\|\bP_{L^\perp}\bU\|
\leq\frac{\|\bB\|/2}{\sigma_d(\bA)-\|\bB\|/2}\leq\frac13.
\]
Hence $\theta:=\angle(\hat L,L)<\pi/4$.
If $\theta>0$, choose a principal-vector pair represented by unit orthogonal vectors $\mathbf{e}\in L$, $\mathbf{f}\in L^\perp$ and
\[
\mathbf{u}=\cos\theta\,\mathbf{e}+\sin\theta\,\mathbf{f}\in\hat L,
\qquad
\mathbf{v}=-\sin\theta\,\mathbf{e}+\cos\theta\,\mathbf{f}\in\hat L^\perp.
\]
Invariance of $\hat L$ under $\bA+\bB$ gives $\mathbf{u}^\top(\bA+\bB)\mathbf{v}=0$. Since $\bA\mathbf{f}=0$,
\[
\sigma_d(\bA)\sin\theta\cos\theta
\leq\sin\theta\cos\theta\,\mathbf{e}^\top\bA\mathbf{e}
=\mathbf{u}^\top\bB\mathbf{v}
=\mathbf{u}^\top(\bB-\|\bB\|\bI_D/2)\mathbf{v}
\leq \|\bB\|/2.
\]
Therefore $\sin(2\theta)\leq \|\bB\|/\sigma_d(\bA)$. The case $\theta=0$ is immediate.
\end{proof}

The following variant of the Davis--Kahan theorem is first order in the off-diagonal block $\bSigma_{L_*,L_*^\perp}$. We use it in the noisy setting, where the perturbation of the inlier term need not be positive semidefinite. Like Lemma~\ref{lemma:psd_angle}, it improves the geometric bound from $C_0^{-k/2}$ to $C_0^{-k}$ compared with Lemma~\ref{lemma:subspace}.
\begin{lemma}\label{lemma:blockDK}
Let $\bSigma\in S_+(D)$. If $\sigma_d(\bSigma_{L_*,L_*})>\sigma_1(\bSigma_{L_*^\perp,L_*^\perp})$, then $\sigma_d(\bSigma)>\sigma_{d+1}(\bSigma)$, and the span $\hat L$ of its top $d$ eigenvectors satisfies \[
\sin\angle(\hat L,L_*)\leq\frac{\|\bSigma_{L_*,L_*^\perp}\|}{\sigma_d(\bSigma_{L_*,L_*})-\sigma_1(\bSigma_{L_*^\perp,L_*^\perp})}.
\]
\end{lemma}
\begin{proof}
Compression interlacing gives
$\sigma_d(\bSigma)\geq \sigma_d(\bSigma_{L_*,L_*})>\sigma_1(\bSigma_{L_*^\perp,L_*^\perp})\geq\sigma_{d+1}(\bSigma)$.
Set $\theta:=\angle(\hat L,L_*)$. If $\theta=0$, the claim is immediate.
Otherwise, choose unit principal vectors $\bu\in L_*$ and
$\bv\in L_*^\perp$ such that
\[
\mathbf{q}=\cos\theta\,\bu+\sin\theta\,\bv\in\hat L,
\qquad
\mathbf{q}'= -\sin\theta\,\bu+\cos\theta\,\bv\in\hat L^\perp.
\]
The two invariant eigenspaces are orthogonal, hence
\[
0=\mathbf{q}'^\top\bSigma\mathbf{q}
=-\tfrac12(\bu^\top\bSigma\bu-\bv^\top\bSigma\bv)\sin(2\theta)+\bu^\top\bSigma\bv\cos(2\theta).
\]
Expanding $\mathbf{q}^\top\bSigma\mathbf{q}-\mathbf{q}'^\top\bSigma\mathbf{q}'$ and using this identity yields $(\mathbf{q}^\top\bSigma\mathbf{q}-\mathbf{q}'^\top\bSigma\mathbf{q}')\cos(2\theta)=\bu^\top\bSigma\bu-\bv^\top\bSigma\bv$. Since $\mathbf{q}\in\hat L$ and $\mathbf{q}'\in\hat L^\perp$, the first factor is at least $\sigma_d(\bSigma)-\sigma_{d+1}(\bSigma)>0$, while the right-hand side is at least $\sigma_d(\bSigma_{L_*,L_*})-\sigma_1(\bSigma_{L_*^\perp,L_*^\perp})>0$. In addition, $|\bu^\top\bSigma\bv|\leq\|\bSigma_{L_*,L_*^\perp}\|$.
Thus $\theta<\pi/4$ and
\[
\tan(2\theta)=\frac{2\bu^\top\bSigma\bv}{\bu^\top\bSigma\bu-\bv^\top\bSigma\bv}\leq\frac{2\|\bSigma_{L_*,L_*^\perp}\|}{\sigma_d(\bSigma_{L_*,L_*})-\sigma_1(\bSigma_{L_*^\perp,L_*^\perp})}.
\]
Hence
$\sin\theta\leq\theta\leq\tan(2\theta)/2$, which proves the claim.
\end{proof}

The next lemma yields a lower bound on $\sigma_D(\bSigma)$, using   $\bSigma_{L_*^{\perp},L_*^{\perp}}$, $\bSigma_{L_*,L_*}$,
and $\bSigma_{L_*^{\perp},L_*}$.
\begin{lemma}\label{lemma:sigmaD}
Assume $\bSigma \in S_{+}(D)$. If $\|\bSigma_{L_*,L_*^{\perp}}\|\leq\sqrt{\sigma_{D-d}(\bSigma_{L_*^{\perp},L_*^{\perp}})\,\sigma_{d}(\bSigma_{L_*,L_*})}/2$, then
\[
\sigma_D(\bSigma)\geq\min\{\sigma_{D-d}(\bSigma_{L_*^{\perp},L_*^{\perp}}),\sigma_{d}(\bSigma_{L_*,L_*})\}/2.
\]

\end{lemma}

The following lemma establishes an  explicit formula of $\bSigma_{L_*^{\perp},L_*^{\perp}}$ in terms of  $\bSigma_{L_*^{\perp},L_*}$, $\bSigma_{L_*,L_*}$ and $\bSigma_{L_*,L_*^{\perp}},$ under the condition that $\rank(\bSigma)=d$.
\begin{lemma}\label{lemma:lowrank}
If $\bSigma\in S_+$, $\rank(\bSigma)=d$ and $\bSigma_{L_*,L_*}$ is invertible, then $$\bSigma_{L_*^{\perp},L_*^{\perp}}=\bSigma_{L_*^{\perp},L_*}\bSigma_{L_*,L_*}^{-1}\bSigma_{L_*,L_*^{\perp}}.$$
\end{lemma}

The next lemma describes various technical properties of the operators defined in Section~\ref{sec:notation_thm1} (e.g., $T(\bSigma)$,  $\Pi_d(\bSigma_+)$ and $\bSigma_{+,\mathrm{in}}$).
\begin{lemma}\label{lemma:bound_sigma_out}
Let $\bSigma \in S_+$. Then the following properties hold:
\begin{enumerate}[(a)]

\item We have
\begin{align}\label{eq:pertubation}
\|T(\bSigma)-\bSigma_{+,\mathrm{in}}\|
&\leq  2\|\bSigma_{+,\mathrm{out}}\|,
\\[0.5em]\label{eq:pertubation2}
\|\Pi_d(\bSigma_+)-\bSigma_{+,\mathrm{in}}\|
&\leq 2\|\bSigma_{+,\mathrm{out}}\|.
\end{align}
Furthermore,
\begin{equation}
\label{eq:widetildeT2}
\|[{T}(\bSigma)]_{L_*,L_*^\perp}\|\leq 2\|\bSigma_{+,\mathrm{out}}\|\,\,\text{ and }\,\,\|[{T}(\bSigma)]_{L_*^\perp,L_*^\perp}\|\leq 2\|\bSigma_{+,\mathrm{out}}\|.
\end{equation}

\item Suppose $\nu > 2$, where $\nu$ is defined in \eqref{eq:def_nu}. Then
\begin{equation}\label{eq:lemmab}\sigma_1([ T(\bSigma)]_{L_*^\perp,L_*^\perp})\leq \frac{4}{\nu-2}\|\bSigma_{+,\mathrm{out}}\| +\gamma \frac{n_0}{D-d}\sigma_1\left(\bSigma_{L_*^{\perp},L_*^{\perp}}\right),\end{equation}
\begin{equation}\label{eq:lemmae}\frac{\|[T(\bSigma)]_{L_*^{\perp},L_*^{\perp}}\|}{\sigma_d([T(\bSigma)]_{L_*,L_*})}\leq \frac{2}{\nu-2},\end{equation}
and, if $\bar{\sigma}_{\mathrm{tail}}(\bSigma_+)>0$,
\begin{align}\label{eq:3rd_bound_lemma8_b}
\frac{\|[T(\bSigma)]_{L_*^{\perp},L_*^{\perp}}\|}{\sigma_{D-d}([T(\bSigma)]_{L_*^{\perp},L_*^{\perp}})}
\leq
 1+\frac{4\|\bSigma_{+,\mathrm{out}}\|}{(\nu-2)\,\gamma \bar{\sigma}_{\mathrm{tail}}(\bSigma_+)}.
\end{align}

\item If $\sigma_{D-d}([T(\bSigma)]_{L_*^\perp,L_*^\perp})>0$, then
\begin{equation}\label{eq:widetildeT}
\|[T(\bSigma)]_{L_*,L_*}-[g_1(T(\bSigma))]_{L_*,L_*}\|
\leq
\frac{\|[T(\bSigma)]_{L_*,L_*^\perp}\|^2}{\sigma_{D-d}([T(\bSigma)]_{L_*^\perp,L_*^\perp})}.
\end{equation}

\item  Recalling the notation  $\calS=\bar{\sigma}_{\mathrm{tail}}\left(\sum_{\bx\in\calX}\frac{\bx\bx^\top}{\|\bx\|^2}\right)$, we have
\begin{align}\label{eq:sigmaD-d}
&\sigma_{D-d}\big([T(\bSigma)]_{L_*^\perp,L_*^\perp}\big)
\geq \gamma\sigma_D(\bSigma)\, \calS.
\end{align}
\end{enumerate}
\end{lemma}
\begin{proof}[Proof of Lemma~\ref{lemma:bound_sigma_out}]
(a) We first note that the  eigenvalues of
$\gamma \, \bar{\sigma}_{\mathrm{tail}}(\bSigma_+) \, \bP_{\mathrm{tail}}(\bSigma_+)-\Pi_{\mathrm{tail}}(\bSigma_+)$ are
$\gamma\bar{\sigma}_{\mathrm{tail}}(\bSigma_+)-\sigma_i(\bSigma_+)$, $d+1\leq i\leq D$, and they satisfy: \begin{align*}\gamma\bar{\sigma}_{\mathrm{tail}}(\bSigma_+)-\sigma_i(\bSigma_+)\leq \gamma\bar{\sigma}_{\mathrm{tail}}(\bSigma_+)\leq \sigma_{d+1}(\bSigma_+), \\ \gamma\bar{\sigma}_{\mathrm{tail}}(\bSigma_+)-\sigma_i(\bSigma_+)\geq -\sigma_i(\bSigma_+)\geq -\sigma_{d+1}(\bSigma_+).\end{align*}
Using this observation, we obtain the following bound:
\begin{multline*}
\|T(\bSigma)-\bSigma_+\|=\|T_2(\bSigma_+,\gamma)-\bSigma_+\|
=
\|\Pi_d(\bSigma_+)+\gamma \, \bar{\sigma}_{\mathrm{tail}}(\bSigma_+) \, \bP_{\mathrm{tail}}(\bSigma_+)-\bSigma_+\|
\\
=\|\gamma \, \bar{\sigma}_{\mathrm{tail}}(\bSigma_+) \, \bP_{\mathrm{tail}}(\bSigma_+)-\Pi_{\mathrm{tail}}(\bSigma_+)\|
\leq \sigma_{d+1}(\bSigma_+).
\end{multline*}
Next, we apply Lemma
\ref{lemma:matrixperturbation}(a), more specifically the second inequality in the first equation of this lemma with $k=D-d$ and $\bU_0 = \bU_{L_*^\perp}$
and note that
$$\sigma_{d+1}(\bSigma_+)\leq \sigma_1([\bSigma_+]_{L_*^\perp,L_*^\perp})=\sigma_1([\bSigma_{+,\mathrm{out}}]_{L_*^\perp,L_*^\perp})\leq \sigma_1(\bSigma_{+,\mathrm{out}}).$$

Applying twice the triangle inequality and the above two equations, we obtain the desired estimate as follows:
\begin{align*}
&\|T(\bSigma)-\bSigma_{+,\mathrm{in}}\|\leq \|T(\bSigma)-\bSigma_+\|+\|\bSigma_+-\bSigma_{+,\mathrm{in}}\|\leq \sigma_{d+1}(\bSigma_+)+\|\bSigma_{+,\mathrm{out}}\|\leq 2\|\bSigma_{+,\mathrm{out}}\|.
\end{align*}
Equation~\eqref{eq:pertubation2} is similarly obtained:
\begin{align*}
\|\Pi_d(\bSigma_+)-\bSigma_{+,\mathrm{in}}\|\leq \sigma_{d+1}(\bSigma_+)+\|\bSigma_+-\bSigma_{+,\mathrm{in}}\|\leq 2\|\bSigma_{+,\mathrm{out}}\|.
\end{align*}

Finally,
we conclude the upper bound of $\|[{T}(\bSigma)]_{L_*,L_*^\perp}\|$ in \eqref{eq:widetildeT2}:
\begin{equation*}
\|[{T}(\bSigma)]_{L_*,L_*^\perp}\|=\|[{T}(\bSigma)-\bSigma_{+,\mathrm{in}}]_{L_*,L_*^\perp}\|\leq \|{T}(\bSigma)-\bSigma_{+,\mathrm{in}}\|\leq 2\|\bSigma_{+,\mathrm{out}}\|.\end{equation*}
The proof of the upper bound of $\|[{T}(\bSigma)]_{L_*^\perp,L_*^\perp}\|$ in \eqref{eq:widetildeT2} is exactly the same.

(b) We derive the following inequalities, which we clarify below:
\begin{align}\label{eq:sigma1perp}
&\sigma_1([ T(\bSigma)]_{L_*^\perp,L_*^\perp})\leq  \sigma_1([ \Pi_d(\bSigma_+)]_{L_*^{\perp},L_*^{\perp}})+\gamma \, \bar{\sigma}_{\mathrm{tail}}(\bSigma_+) \, \sigma_1([\bP_{\mathrm{tail}}(\bSigma_+)
]_{L_*^{\perp},L_*^{\perp}})
\\\nonumber
&\leq  \sigma_1([ \Pi_d(\bSigma_+)]_{L_*^{\perp},L_*^{\perp}})+\gamma\bar{\sigma}_{\mathrm{tail}}(\bSigma_+)
\leq \frac{\|[\Pi_d(\bSigma_+)]_{L_*^\perp,L_*}\|^2}{\sigma_d([\Pi_d(\bSigma_+)]_{L_*,L_*})} +\gamma \bar{\sigma}_{\mathrm{tail}}(\bSigma_+)
\\
\nonumber
&\leq \frac{4\|\bSigma_{+,\mathrm{out}}\|^2}{\sigma_d(\bSigma_{+,\mathrm{in}})-2\|\bSigma_{+,\mathrm{out}}\|} +\gamma \bar{\sigma}_{\mathrm{tail}}(\bSigma_+)
= \frac{4}{\nu-2}\|\bSigma_{+,\mathrm{out}}\| +\gamma \bar{\sigma}_{\mathrm{tail}}(\bSigma_+).
\end{align}
The first inequality follows from the definition of $T$ and the fact that for $\bA \in S_+$, $\sigma_1(\bA)=\|\bA\|$.  The second inequality follows from
$\sigma_1([\bP_{\mathrm{tail}}(\bSigma_+)
]_{L_*^{\perp},L_*^{\perp}})\leq 1$, which holds since $\bP_{\mathrm{tail}}(\bSigma_+)$ is an orthogonal projector. The third inequality
is clarified next using  Lemma~\ref{lemma:lowrank} and the fact that $\Pi_d(\bSigma_+)$ is of rank $d$:
\begin{align}
\nonumber
&\sigma_1([ \Pi_d(\bSigma_+)]_{L_*^{\perp},L_*^{\perp}})= \sigma_1\left([ \Pi_d(\bSigma_+)]_{L_*^{\perp},L_*}([ \Pi_d(\bSigma_+)]_{L_*,L_*})^{-1}[ \Pi_d(\bSigma_+)]_{L_*,L_*^{\perp}}\right)\\ &\leq\left\|[ \Pi_d(\bSigma_+)]_{L_*^{\perp},L_*}\right\|^2
\cdot\left\|([ \Pi_d(\bSigma_+)]_{L_*,L_*})^{-1}\right\|=\frac{\left\|[\Pi_d(\bSigma_+)]_{L_*^\perp,L_*}\right\|^2}{\sigma_d([\Pi_d(\bSigma_+)]_{L_*,L_*})}.\label{eq:bound_sigma_1_multiplication}
\end{align}
Next, we verify the fourth inequality in \eqref{eq:sigma1perp}. The bound for the numerator applies     \eqref{eq:pertubation2}:
$$\|[\Pi_d(\bSigma_+)]_{L_*^\perp,L_*}\|=\|[\Pi_d(\bSigma_+)-\bSigma_{+,\mathrm{in}}]_{L_*^\perp,L_*}\|\leq \|\Pi_d(\bSigma_+)-\bSigma_{+,\mathrm{in}}\|\leq 2\|\bSigma_{+,\mathrm{out}}\|.$$
The bound for the denominator applies Lemma~\ref{lemma:matrixperturbation}(b) and \eqref{eq:pertubation2}:
\[
|\sigma_d([\Pi_d(\bSigma_+)]_{L_*,L_*})-\sigma_d(\bSigma_{+,\mathrm{in}})|\leq\|[\Pi_d(\bSigma_+)-\bSigma_{+,\mathrm{in}}]_{L_*,L_*}\|\leq2\|\bSigma_{+,\mathrm{out}}\|,
\]
where the first inequality also uses that $\bSigma_{+,\mathrm{in}}$ has range in $L_*$, so that its nonzero eigenvalues are those of $[\bSigma_{+,\mathrm{in}}]_{L_*,L_*}$. In particular, $\sigma_d([\Pi_d(\bSigma_+)]_{L_*,L_*})\geq\sigma_d(\bSigma_{+,\mathrm{in}})-2\|\bSigma_{+,\mathrm{out}}\|>0$ when $\nu>2$, so $[\Pi_d(\bSigma_+)]_{L_*,L_*}$ is invertible, as required by Lemma~\ref{lemma:lowrank}.
The last equality of \eqref{eq:sigma1perp} follows from the definition of $\nu$.

The second term in the RHS of \eqref{eq:sigma1perp} is bounded as follows (see explanations below):
\begin{align}
\nonumber
&\bar{\sigma}_{\mathrm{tail}}(\bSigma_+)=\frac{1}{D-d}\sum_{i=d+1}^{D}\sigma_{i}(\bSigma_+) =\frac{1}{D-d}\min_{\bV\in O(D,D-d)}\tr(\bV^\top\bSigma_+\bV)\\
&\leq \frac{1}{D-d}\tr(\bU_{L_*^{\perp}}^\top \bSigma_{+}\bU_{L_*^{\perp}})
=\frac{1}{D-d}
\tr(\bU_{L_*^{\perp}}^\top \bSigma_{+,\mathrm{out}}\bU_{L_*^{\perp}})\leq \frac{n_0}{D-d}\sigma_1(\bSigma_{L_*^{\perp},L_*^{\perp}}).\label{eq:sigma1perp2}
\end{align}
The second equality follows from the extremal partial trace formula in~\cite[Proposition 1.3.4]{tao2012topics}.
The third equality is because the inliers lie on $L_*$ and thus
$\bU_{L_*^{\perp}}^\top \bSigma_{+,\mathrm{in}}\bU_{L_*^{\perp}} = \bm{0}$. Lastly, the second inequality applies the trace bound in \eqref{eq:xout2}.
Combining \eqref{eq:sigma1perp} and \eqref{eq:sigma1perp2}, \eqref{eq:lemmab} is proved.

To prove \eqref{eq:lemmae}, we note that \begin{equation*}
\frac{\|[T(\bSigma)]_{L_*^{\perp},L_*^{\perp}}\|}{\sigma_d([T(\bSigma)]_{L_*,L_*})}\leq \frac{2\|\bSigma_{+,\mathrm{out}}\|}{\sigma_d([\bSigma_{+,\mathrm{in}}]_{L_*,L_*})-2\|\bSigma_{+,\mathrm{out}}\|}=\frac{2}{\nu-2}.
\end{equation*}
For the numerator in the  inequality we used
\eqref{eq:widetildeT2}.
For the denominator in the inequality
we used Lemma~\ref{lemma:matrixperturbation}(b)
and
\eqref{eq:pertubation}.
We can divide the two bounds as long as the denominator is positive, which holds when $\nu >2$. The last equality uses the definition of $\nu$.

Finally, we conclude \eqref{eq:3rd_bound_lemma8_b}. We bound the  numerator, $\|[T(\bSigma)]_{L_*^{\perp},L_*^{\perp}}\|\equiv \sigma_1([T(\bSigma)]_{L_*^{\perp},L_*^{\perp}})$, using \eqref{eq:sigma1perp}. For the denominator, we note that all eigenvalues of $T(\bSigma)=T_2(\bSigma_+,\gamma)$ are at least $\gamma\bar{\sigma}_{\mathrm{tail}}(\bSigma_+)$ (its top $d$ eigenvalues are those of $\bSigma_+$, which are at least $\sigma_{d+1}(\bSigma_+)\geq\bar{\sigma}_{\mathrm{tail}}(\bSigma_+)$, and its remaining eigenvalues equal $\gamma\bar{\sigma}_{\mathrm{tail}}(\bSigma_+)$), so that Lemma~\ref{lemma:matrixperturbation}(a) (with $k=D-d$ and $\bU_0=\bU_{L_*^\perp}$) yields
\begin{equation}\label{eq:tail_block_lower}
\sigma_{D-d}([T(\bSigma)]_{L_*^\perp,L_*^\perp})\geq \sigma_D(T(\bSigma))=\gamma\bar{\sigma}_{\mathrm{tail}}(\bSigma_+).
\end{equation}
Dividing \eqref{eq:sigma1perp} by \eqref{eq:tail_block_lower} results in \eqref{eq:3rd_bound_lemma8_b}.

(c) Note that $\|[T(\bSigma)-g_1(T(\bSigma))]_{L_*,L_*}\|=
\|
[T(\bSigma)]_{L_*,L_*^\perp}([T(\bSigma)]_{L_*^\perp,L_*^\perp})^{-1}[T(\bSigma)]_{L_*^\perp,L_*}
\|$, where this equality directly follows from the definition of $g_1$ in \eqref{eq:defg1}. We bound the RHS of this equality using similar arguments as those in \eqref{eq:bound_sigma_1_multiplication} (i.e., the sub-multiplicative property of the spectral norm and for $\bA \in S_{++}(D-d)$, $\|\bA^{-1}\|=\sigma_1(\bA^{-1})=\sigma_{D-d}^{-1}(\bA)$).

(d)
Applying \eqref{eq:tail_block_lower}, then \eqref{eq:xout3} together with the monotonicity of $\bar{\sigma}_{\mathrm{tail}}$ (the average of the $D-d$ smallest eigenvalues) with respect to the Loewner order, and at last the definition of $\calS$, results in
\begin{align*}
&\sigma_{D-d}\big([T(\bSigma)]_{L_*^\perp,L_*^\perp}\big)\geq \gamma\bar{\sigma}_{\mathrm{tail}}(\bSigma_+)
\geq \gamma\,\bar{\sigma}_{\mathrm{tail}}\Big(\sigma_D(\bSigma)\sum_{\bx\in\calX}\frac{\bx\bx^\top}{\|\bx\|^2}\Big)= \gamma\sigma_D(\bSigma)\, \calS.
\qedhere
\end{align*}
\end{proof}

The last lemma of this section provides lower and upper bounds for  the smallest eigenvalue of a certain product of p.s.d.~matrices.
\begin{lemma}\label{lemma:matrixeigenvalue}
If $\bA$, $\bB\in S_+(d)$, then
\[
\sigma_d(\bA\bB\bA)\geq \sigma_d^2(\bA)\sigma_d(\bB)
\ \text{ and } \
\sigma_d(\bA\bB\bA)\leq \sigma_1^2(\bA)\sigma_d(\bB).
\]
\end{lemma}

\subsubsection{Reduction of the Theorem}
\label{sec:reduce_noiseless-thm}
To prove the theorem, we begin by reducing the task to proving an alternative statement. We first introduce some notation, recalling from Section~\ref{sec:additional_quantities} that $\tilde{\calA}=\frac{n_0}{D-d}{\calA}$, $\calS=\bar{\sigma}_{\mathrm{tail}}(\sum_{\bx\in\calX}\bx\bx^\top/\|\bx\|^2)$ and $\calR=\tilde{\calA}/\calS$. We define \begin{align}\label{eq:tildekappa1}\tilde{\kappa}_1 =&C_{\mathrm b} \, \frac{\calA}{\dssnr\cdot \sigma_d(\bSigma_{\mathrm{in},*})}\left(\frac{\calA}{\dssnr-\gamma}+\frac{\kappa_2\calR}{\gamma}(1+\kappa_{\mathrm{in},*})\right)\\=& C_{\mathrm b} \, \frac{d\,\tilde{\calA}}{n_1\,\sigma_d(\bSigma_{\mathrm{in},*})}\left(\frac{\tilde{\calA}}{\frac{n_1}{d}-\gamma\frac{n_0}{D-d}}+\frac{\kappa_2\tilde{\calA}}{\gamma \calS}(1+\kappa_{\mathrm{in},*})\right).\nonumber\end{align}
We note that in view of
\eqref{eq:kappa1} and the definitions of $\tilde{\kappa}_1$, $\dssnr$ and $\kappa_{\mathrm{in},*}$, 
$\tilde{\kappa}_1 \leq \kappa_1 /  \sigma_1(\bSigma_{\mathrm{in},*})$.

We recall the definition of $C_0$ in \eqref{eq:C_value}
and note that the assumption $\dssnr > \gamma$ implies that $C_0>1$.
Using these notations and definitions and the ones from Section~\ref{sec:notation_thm1}, we formulate the alternative statement as follows:
\begin{multline}
\label{eq:overiterations}
\text{If }  \hat{\kappa}_1(\bSigma)\geq \tilde{\kappa}_1   \text{ and }  \hat{\kappa}_2(\bSigma)\leq \max(\kappa_2,3),\\   \text{ then }
\hat{\kappa}_1(T(\bSigma))\geq C_0 \hat{\kappa}_1(\bSigma)  \text{ and } \hat{\kappa}_2(T(\bSigma))\leq \max(\kappa_2,3).
\end{multline}
We will show next that \eqref{eq:overiterations} implies the conclusion of the theorem.

Since $\calR$ is finite and $\tilde{\calA}>0$, we have $\calS>0$. The initialization is positive definite, and \eqref{eq:xout3} implies a positive tail average at every step; hence $T_2$ produces a positive definite matrix. By induction all finite iterates are positive definite, so their spectral ratios have positive denominators.

We first note that $\hat\kappa_1(\bSigma^{(0)})\geq\tilde\kappa_1$. Indeed, \eqref{eq:hatkappa1_relation} and the initialization condition give
\begin{equation}\label{eq:k1hatlargerk1tilde}
\hat\kappa_1(\bSigma^{(0)})
\geq\frac{\kappa_1}{\sigma_1(\bSigma_{\mathrm{in},*})}
\geq\tilde\kappa_1.
\end{equation}
Furthermore, since $\hat{\kappa}_2(\bSigma^{(0)})=\kappa_2$, both conditions of \eqref{eq:overiterations} are satisfied for $\bSigma=\bSigma^{(0)}$. We obtain by induction that the conclusion of \eqref{eq:overiterations} holds for $\bSigma^{(k+1)}=T\left(\bSigma^{(k)}\right)$, $k \geq 0$.

We note that \eqref{eq:overiterations} and \eqref{eq:k1hatlargerk1tilde} imply that $\hat{\kappa}_1\left(\bSigma^{(k)}\right)\rightarrow \infty$ and it grows at least geometrically. We further note that
\[
\bSigma_{L_*,L_*}-[g_1(\bSigma)]_{L_*,L_*}
=\bSigma_{L_*,L_*^\perp}\bSigma_{L_*^\perp,L_*^\perp}^{-1}\bSigma_{L_*^\perp,L_*}\psdgeq0,
\]
so $\bSigma_{L_*,L_*}\psdgeq[g_1(\bSigma)]_{L_*,L_*}$,
and consequently,
\begin{equation}\label{eq:reduction2}
\sigma_d\left(\bSigma_{L_*,L_*}\right)/\sigma_1\left(\bSigma_{L_*^\perp,L_*^\perp}\right)\geq \sigma_d(g_1(\bSigma))/\sigma_1\left(\bSigma_{L_*^\perp,L_*^\perp}\right)\geq \sigma_d(\bSigma_{\mathrm{in},*})\,\hat{\kappa}_1(\bSigma),
\end{equation}
where the last inequality follows from the definitions of $h$ and $\hat{\kappa}_1$ in \eqref{eq:def_f} and \eqref{eq:def_kappa_1_2} and the second inequality of Lemma~\ref{lemma:matrixeigenvalue} (applied with $\bA=\bSigma_{\mathrm{in},*}^{-0.5}$ and $\bB=[g_1(\bSigma)]_{L_*,L_*}$), which yields $h(\bSigma)\leq \sigma_d(g_1(\bSigma))/\sigma_d(\bSigma_{\mathrm{in},*})$.
Combining \eqref{eq:reduction2} with Lemma~\ref{lemma:subspace}, we conclude that $\sin\angle({L}^{(k)},L_*)\leq 1/\sqrt{\sigma_d(\bSigma_{\mathrm{in},*})\hat{\kappa}_1\left(\bSigma^{(k)}\right)}$. This observation and the fact $\hat{\kappa}_1(T(\bSigma))\geq C_0 \hat{\kappa}_1(\bSigma)$ (according to \eqref{eq:overiterations}) imply that  $\angle({L}^{(k)},L_*)$ converges $R$-linearly to zero {in the sense that $\angle(L^{(k)},L_*)\leq c \cdot C_0^{-k/2}$}.

To obtain the explicit bound \eqref{eq:explicit_rate}, fix $j\geq0$ and put $\bSigma=\bSigma^{(j)}$. The decomposition $\bSigma_+=\bSigma_{+,\mathrm{in}}+\bSigma_{+,\mathrm{out}}$ has a rank-$d$ positive semidefinite inlier term with range $L_*$ and a positive semidefinite outlier term. The proof of \eqref{eq:overiterations} gives $\nu(\bSigma^{(j)})\geq C_{\mathrm b}\geq26$ for every $j$ (see \eqref{eq:est_nu}-\eqref{eq:est_nv_asumption}). Hence Lemma~\ref{lemma:psd_angle} applies, and the top $d$ eigenvalues of $\bSigma_+$ are separated from the remaining ones. The operator $T_2$ preserves this leading eigenspace, so it is $L^{(j+1)}$. We obtain $\angle(L^{(j+1)},L_*)<\pi/4$ and
\begin{equation}\label{eq:angle_via_nu}
\sin\angle(L^{(j+1)},L_*)\leq\sin\bigl(2\angle(L^{(j+1)},L_*)\bigr)\leq\frac1{\nu(\bSigma^{(j)})}
\leq\frac{\calA}{\dssnr\,\sigma_d(\bSigma_{\mathrm{in},*})\,\hat\kappa_1(\bSigma^{(j)})},
\end{equation}
where the last inequality is \eqref{eq:est_nu}. By \eqref{eq:overiterations} and \eqref{eq:k1hatlargerk1tilde},
\[
\hat\kappa_1(\bSigma^{(j)})\geq C_0^j\kappa_1/\sigma_1(\bSigma_{\mathrm{in},*})\quad\text{and}\quad\nu(\bSigma^{(j)})\geq\frac{\dssnr\,\kappa_1}{\calA\,\kappa_{\mathrm{in},*}}\,C_0^j.
\]
These inequalities prove \eqref{eq:explicit_rate}.

\subsubsection{Proof of \eqref{eq:overiterations}}
\label{sec:prove_reduce_noiseless-thm}
We separately prove the two parts of \eqref{eq:overiterations}.

\textbf{Part I: Proof that $\boldsymbol{\hat{\kappa}_1(T(\bSigma))\geq C_0 \hat{\kappa}_1(\bSigma)}$: }
Recall that
\begin{equation}\hat{\kappa}_1(T(\bSigma))=\frac{h({T}(\bSigma))}{\sigma_1([ T(\bSigma)]_{L_*^\perp,L_*^\perp})}
=
\frac{\sigma_d(\bSigma_{\mathrm{in},*}^{-0.5}[g_1(T(\bSigma))]_{L_*,L_*}\bSigma_{\mathrm{in},*}^{-0.5})}{\sigma_1([ T(\bSigma)]_{L_*^\perp,L_*^\perp})}.
\label{eq:kappa1_proof}\end{equation}
We estimate the denominator of $\hat{\kappa}_1(T(\bSigma))$ under a special condition, then its numerator, and at last use these estimates to conclude that $\hat{\kappa}_1(T(\bSigma))\geq C_0 \hat{\kappa}_1(\bSigma)$.

\textbf{A lower bound on $\nu$:} 
We bound $\nu$ from below and later use  the derived condition in different bounds. For this purpose, we  obtain an upper bound for $h(\bSigma_{+,\mathrm{in}})$.
Applying the fact that $g_1(\bSigma_{+,\mathrm{in}}) = \bSigma_{+,\mathrm{in}}$ (in view of our convention of the inverse of a singular matrix) and Lemma~\ref{lemma:matrixeigenvalue}, while noting that the matrices $\bSigma_{\mathrm{in},*}^{-0.5}$ and $[\bSigma_{+,\mathrm{in}}]_{L_*,L_*}$ are in $S_+(d)$, yields
\begin{equation}
\label{eq:upper_bound_h}
h(\bSigma_{+,\mathrm{in}})=\sigma_d\big(\bSigma_{\mathrm{in},*}^{-0.5}[\bSigma_{+,\mathrm{in}}]_{L_*,L_*}\bSigma_{\mathrm{in},*}^{-0.5}\big)\leq  \frac{\sigma_d\big([\bSigma_{+,\mathrm{in}}]_{L_*,L_*}\big)}{\sigma_d(\bSigma_{\mathrm{in},*})}=\frac{\sigma_d\big(\bSigma_{+,\mathrm{in}}\big)}{\sigma_d(\bSigma_{\mathrm{in},*})}.
\end{equation}

We lower bound  $\nu=\sigma_d(\bSigma_{+,\mathrm{in}})/\|\bSigma_{+,\mathrm{out}}\|$ (see \eqref{eq:def_nu})
as follows (see explanations below):
\begin{equation}\label{eq:est_nu}
\nu \geq \frac{h(\bSigma_{+,\mathrm{in}})\sigma_d(\bSigma_{\mathrm{in},*})}{\|\bSigma_{+,\mathrm{out}}\|}\geq \frac{\frac{n_1}{d} \, h(\bSigma) \, \sigma_d(\bSigma_{\mathrm{in},*})}{\tilde{\calA}\,\sigma_1\left(\bSigma_{L_*^\perp,L_*^\perp}\right)}
=\hat{\kappa}_1(\bSigma)\frac{\dssnr \cdot \sigma_d(\bSigma_{\mathrm{in},*})}{{\calA}}.
\end{equation}
The first inequality applies \eqref{eq:upper_bound_h} and the definition of $\nu$.  The second one applies \eqref{eq:xin} for the numerator and the first inequality in \eqref{eq:xout2} for the denominator. The first equality uses the definition of $\hat{\kappa}_1$ in  \eqref{eq:def_kappa_1_2} and the last one the definitions of $\dssnr$ and $\tilde{\calA}$.

We further show that $\nu\geq C_{\mathrm b}$, where $C_{\mathrm b}$ is defined in \eqref{eq:C_value}, and, consequently, $\nu \geq 26$. Applying \eqref{eq:kappa1} (and \eqref{eq:tildekappa1} for the second inequality), together with the facts that $\kappa_2\geq 1$ (by Lemma~\ref{lemma:matrixperturbation}(a)), $\calR\geq 1$ (see Appendix~\ref{sec:R_bound}), $\gamma<1$ and $\kappa_{\mathrm{in},*}\geq 1$, which imply that the expression in parentheses in \eqref{eq:tildekappa1} is at least $\kappa_2\calR(1+\kappa_{\mathrm{in},*})/\gamma\geq 1+\kappa_{\mathrm{in},*}\geq 2$,
we note that
\begin{equation}
\label{eq:est_nv_asumption}
\kappa_1\geq
C_{\mathrm b} \, \frac{\calA \cdot \kappa_{\mathrm{in},*}}{\dssnr} =
C_{\mathrm b} \, \frac{\calA \cdot
\sigma_1(\bSigma_{\mathrm{in},*})
}{\dssnr\cdot \sigma_d(\bSigma_{\mathrm{in},*})}
\ \text{ and } \
\tilde{\kappa}_1 \geq C_{\mathrm b} \, \frac{\calA}{\dssnr\cdot \sigma_d(\bSigma_{\mathrm{in},*})}.
\end{equation}
Finally, using the assumption $\hat{\kappa}_1(\bSigma)\geq \tilde{\kappa}_1$ (see \eqref{eq:overiterations}),  \eqref{eq:est_nu} and~\eqref{eq:est_nv_asumption}, we conclude that $\nu\geq C_{\mathrm b}$.

\textbf{Estimate of the denominator of $\boldsymbol{\hat{\kappa}_1}$:}
We bound the denominator of $\hat{\kappa}_1(T(\bSigma))$ from above.  Combining the bound on $\|\bSigma_{+,\mathrm{out}}\|$ in \eqref{eq:xout2} and \eqref{eq:lemmab}   (note we verified above its required condition, that is, $\nu> 2$), we conclude the upper bound:
\begin{equation}\label{eq:kappadenominator}
\sigma_1([ T(\bSigma)]_{L_*^\perp,L_*^\perp})\leq \frac{4}{\nu-2} \tilde{\calA}\sigma_1\left(\bSigma_{L_*^\perp,L_*^\perp}\right)+\gamma \frac{n_0}{D-d}\sigma_1\left(\bSigma_{L_*^\perp,L_*^\perp}\right).
\end{equation}

\textbf{Estimate of the numerator of $\boldsymbol{\hat{\kappa}_1}$:}
We bound $h(T(\bSigma))$ using arguments detailed below:
\begin{multline}
\label{eq:kappa1_numerator}
h(T(\bSigma))=
\sigma_d(\bSigma_{\mathrm{in},*}^{-0.5}[g_1(T(\bSigma))]_{L_*,L_*}\bSigma_{\mathrm{in},*}^{-0.5}) \geq\\
 \sigma_d(\bSigma_{\mathrm{in},*}^{-0.5}[\bSigma_{+,\mathrm{in}}]_{L_*,L_*}\bSigma_{\mathrm{in},*}^{-0.5})-\big\|\bSigma_{\mathrm{in},*}^{-0.5}\,([\bSigma_{+,\mathrm{in}}]_{L_*,L_*}-[g_1(T(\bSigma))]_{L_*,L_*})\,\bSigma_{\mathrm{in},*}^{-0.5}\big\| \geq
\\h(\bSigma_{+,\mathrm{in}})\!-\!\frac{\|[\bSigma_{+,\mathrm{in}}]_{L_*,L_*}-[g_1(T(\bSigma))]_{L_*,L_*}\|}{\sigma_d(\bSigma_{\mathrm{in},*})}\geq h(\bSigma_{+,\mathrm{in}})\!-\!
\frac{2\|\bSigma_{+,\mathrm{out}}\|\!+\!\frac{\|[T(\bSigma)]_{L_*,L_*^\perp}\|^2}{\sigma_{D-d}([T(\bSigma)]_{L_*^\perp,L_*^\perp})}}{\sigma_d(\bSigma_{\mathrm{in},*})}.
\end{multline}
The first inequality follows from Lemma \ref{lemma:matrixperturbation}(b). The second inequality applies the definition of $h$ to its first term and the following argument, which is similar to the one in \eqref{eq:upper_bound_h}, to its second term:
\[
\|\bSigma_{\mathrm{in},*}^{-0.5}\bX\bSigma_{\mathrm{in},*}^{-0.5}\|\leq \|\bSigma_{\mathrm{in},*}^{-0.5}\|\cdot\|\bX\|\cdot\|\bSigma_{\mathrm{in},*}^{-0.5}\| ={\|\bX\|}/{\sigma_d(\bSigma_{\mathrm{in},*})}.
\]
The last inequality applies the triangle inequality, \eqref{eq:pertubation} and \eqref{eq:widetildeT}.

Applying \eqref{eq:widetildeT2} and \eqref{eq:sigmaD-d},  and then the bound of $\|\bSigma_{+,\mathrm{out}}\|$ in \eqref{eq:xout2} yields
\begin{align*}
&\frac{\|[T(\bSigma)]_{L_*,L_*^\perp}\|^2}{\sigma_{D-d}([T(\bSigma)]_{L_*^{\perp},L_*^{\perp}})}
\leq  \frac{(2\|\bSigma_{+,\mathrm{out}}\|)^2}{\gamma\sigma_D(\bSigma)\calS} \leq  \frac{4\|\bSigma_{+,\mathrm{out}}\|\tilde{\calA}}{\gamma\calS}\cdot\frac{\sigma_1\left(\bSigma_{L_*^\perp,L_*^\perp}\right)}{\sigma_D(\bSigma)} =  \frac{4\|\bSigma_{+,\mathrm{out}}\|\tilde{\calA}}{\gamma\calS} \, \hat{\kappa}_2(\bSigma).
\end{align*}
Plugging the above bound in \eqref{eq:kappa1_numerator} results in the following estimate of the numerator of \eqref{eq:kappa1_proof}:
\begin{equation}
\label{eq:bound_nuemrator_final}
h(T(\bSigma))\geq h(\bSigma_{+,\mathrm{in}})-
\frac{2\|\bSigma_{+,\mathrm{out}}\|\!+\!\frac{4\|\bSigma_{+,\mathrm{out}}\|  \tilde{\calA}  \hat{\kappa}_2(\bSigma)}{\gamma\calS}}{\sigma_d(\bSigma_{\mathrm{in},*})}.
\end{equation}

\textbf{Conclusion of the proof that $\boldsymbol{\hat{\kappa}_1(T(\bSigma))\geq C_0 \hat{\kappa}_1(\bSigma)}$}:
Under the hypotheses of \eqref{eq:overiterations}, the bounds \eqref{eq:kappa1nu} and \eqref{eq:aux_eq_to_figure_const}, derived below directly from those hypotheses, give
\[
1-\left(2+\frac{4\tilde\calA\hat\kappa_2(\bSigma)}{\gamma\calS}\right)\frac{\kappa_{\mathrm{in},*}}{\nu}\geq1-\frac{14}{C_{\mathrm b}}>0.
\]
Thus the last multiplicative substitution in the following bound preserves its direction:
\begin{multline}
\hat{\kappa}_1(T(\bSigma))\geq\frac{h(\bSigma_{+,\mathrm{in}})-\Big({2+\frac{4\tilde{\calA}\hat{\kappa}_2(\bSigma)}{\gamma\calS}}\Big)\frac{\|\bSigma_{+,\mathrm{out}}\|}{\sigma_d(\bSigma_{\mathrm{in},*})}}{\frac{4}{\nu-2}\tilde{\calA}\sigma_1(\bSigma_{L_*^{\perp},L_*^{\perp}}) +\gamma \frac{n_0}{D-d}\sigma_1(\bSigma_{L_*^{\perp},L_*^{\perp}})}\\
= \frac{h(\bSigma_{+,\mathrm{in}})-\Big({2+\frac{4\tilde{\calA}\hat{\kappa}_2(\bSigma)}{\gamma\calS}}\Big)\frac{{\sigma_d(\bSigma_{+,\mathrm{in}})}}{\nu\sigma_d(\bSigma_{\mathrm{in},*})}}{\frac{4}{\nu-2}\tilde{\calA}\sigma_1(\bSigma_{L_*^{\perp},L_*^{\perp}}) +\gamma \frac{n_0}{D-d}\sigma_1(\bSigma_{L_*^{\perp},L_*^{\perp}})}
\geq\frac{h(\bSigma_{+,\mathrm{in}})-(2+\frac{4\tilde{\calA}\hat{\kappa}_2(\bSigma)}{\gamma\calS}) \frac{\kappa_{\mathrm{in},*}
\, h(\bSigma_{+,\mathrm{in}})}{\nu}}{\frac{4}{\nu-2}\tilde{\calA}\sigma_1(\bSigma_{L_*^{\perp},L_*^{\perp}}) +\gamma \frac{n_0}{D-d}\sigma_1(\bSigma_{L_*^{\perp},L_*^{\perp}})}\\
=\frac{1-(2+\frac{4\tilde{\calA}\hat{\kappa}_2(\bSigma)}{\gamma\calS}) \frac{\kappa_{\mathrm{in},*}}
{\nu}}
{\frac{4}{\nu-2}\tilde{\calA} +\gamma \frac{n_0}{D-d}}\frac{h(\bSigma_{+,\mathrm{in}})}{\sigma_1(\bSigma_{L_*^{\perp},L_*^{\perp}})}\geq \frac{1-(2+\frac{4\tilde{\calA}\hat{\kappa}_2(\bSigma)}{\gamma\calS}) \frac{\kappa_{\mathrm{in},*}}
{\nu}}{\frac{4}{\nu-2} \, \tilde{\calA} +\gamma \frac{n_0}{D-d}}\frac{n_1}{d}\hat{\kappa}_1(\bSigma)\label{eq:est_kappa1}.
\end{multline}
The first inequality follows by combining \eqref{eq:kappa1_proof} with \eqref{eq:kappadenominator} and \eqref{eq:bound_nuemrator_final}, where the RHS of \eqref{eq:bound_nuemrator_final} has been simplified.
The first equality uses the definition of $\nu$ (see \eqref{eq:def_nu}) in the last multiplicative term of the numerator. The second  inequality applies $\sigma_d(\bSigma_{+,\mathrm{in}})\leq h(\bSigma_{+,\mathrm{in}}) \sigma_1(\bSigma_{\mathrm{in},*})$ (which follows from the first inequality in  Lemma~\ref{lemma:matrixeigenvalue}, the definition of $h$ and the fact that $\sigma_1(\bSigma_{\mathrm{in},*})=\sigma_d^{-1}(\bSigma_{\mathrm{in},*}^{-1})$), and the definition of $\kappa_{\mathrm{in},*}$ in \eqref{eq:defkinstar}. The last inequality follows from \eqref{eq:xin} and the definition $\hat{\kappa}_1(\bSigma)=h(\bSigma)/{\sigma_1(\bSigma_{L_*^{\perp},L_*^{\perp}})}$.

We further simplify the above bound. We first note that combining  \eqref{eq:tildekappa1}, the assumption $\hat{\kappa}_1(\bSigma)\geq \tilde{\kappa}_1$  (see \eqref{eq:overiterations}), and \eqref{eq:est_nu} yields
\begin{equation}\label{eq:kappa1nu}
\nu\geq C_{\mathrm b}\left(\frac{\tilde{\calA}}{\frac{n_1}{d}-\gamma\frac{n_0}{D-d}}+\frac{\kappa_2\tilde{\calA}}{\gamma \calS}(1+\kappa_{\mathrm{in},*})\right).
\end{equation}
We note that the second term in parentheses is at least $\kappa_{\mathrm{in},*}$, since $\kappa_2\geq 1$, $\tilde{\calA}/\calS=\calR\geq 1$ and $\gamma<1$; in particular, $\nu\geq C_{\mathrm b}\kappa_{\mathrm{in},*}$. Using this bound, while focusing on the two terms of its RHS, and the inequalities $\nu \geq 4$ and $\hat{\kappa}_2(\bSigma)\leq \max(\kappa_2,3) \leq 3 \kappa_2$ (see \eqref{eq:overiterations}),
we obtain the estimates
\begin{equation}
\frac{4}{\nu-2}\tilde{\calA}\leq \frac{8}{\nu}\tilde{\calA} \leq \frac{8}{C_{\mathrm b}}\Big(\frac{n_1}{d}-\gamma\frac{n_0}{D-d}\Big), \ \
2 \cdot \frac{\kappa_{\mathrm{in},*}}{\nu}
\leq \frac{2}{C_{\mathrm b}},
\ \text{ and } \ \frac{4\tilde{\calA}\hat{\kappa}_2(\bSigma)}{\gamma\calS} \cdot\frac{\kappa_{\mathrm{in},*}}{\nu}\leq \frac{12}{C_{\mathrm b}}.
\label{eq:aux_eq_to_figure_const}
\end{equation}
Combining \eqref{eq:est_kappa1}-\eqref{eq:aux_eq_to_figure_const} leads to the bound
\begin{equation*}\hat{\kappa}_1(T(\bSigma))\geq
\frac{1-\frac{(2+12)}{C_{\mathrm b}}}{\frac{8}{C_{\mathrm b}}\Big(\frac{n_1}{d}-\gamma\frac{n_0}{D-d}\Big)+\gamma\frac{n_0}{D-d}}
\cdot
\frac{n_1}{d}
\cdot
\hat{\kappa}_1(\bSigma).\end{equation*}

To conclude the proof, we show that
\begin{equation}
\label{eq:conditions_C_aux}
    \frac{1-\frac{14}{C_{\mathrm b}}}{\frac{8}{C_{\mathrm b}}\Big(\frac{n_1}{d}-\gamma\frac{n_0}{D-d}\Big)+\gamma\frac{n_0}{D-d}}\frac{n_1}{d}
\geq C_0.
\end{equation}
Using the definition of $\dssnr$ and the constant $C_0$ in \eqref{eq:C_value}, we rewrite this equation as follows:
\begin{equation}
\label{eq:conditions_C}
\frac{1-\frac{14}{C_{\mathrm b}}}{\frac{8}{C_{\mathrm b}}\left(\frac{\dssnr}{\gamma}-1\right)+1}\frac{\dssnr}{\gamma} \geq  2\cdot \frac{\dssnr}{\dssnr+\gamma}.
\end{equation}
One can note that \eqref{eq:conditions_C} holds when
$C_{\mathrm b}\geq{(30\dssnr-2\gamma)}/{(\dssnr-\gamma)}$, and that the constant $C_{\mathrm b}$ specified in \eqref{eq:C_value} satisfies this inequality. Consequently, $\hat{\kappa}_1(T(\bSigma))\geq  C_0\hat{\kappa}_1(\bSigma)$. More generally, writing $x=\dssnr/\gamma$, the left-hand side of \eqref{eq:conditions_C} is at least $\rho$ for a given $\rho\in(1,x)$ as soon as $C_{\mathrm b}\geq (14x+8\rho(x-1))/(x-\rho)$; the choice $\rho=C_0=2x/(x+1)$ recovers the above requirement on $C_{\mathrm b}$. This observation is used in Corollary~\ref{cor:rate}.

\textbf{Part II: Proof that $\boldsymbol{\hat{\kappa}_2(T(\bSigma))\leq}$ max$\boldsymbol{(\kappa_2,3)}$: }
First, we bound the denominator of $\hat{\kappa}_2(T(\bSigma))$, that is, ${\sigma_D(T(\bSigma))}$, where the main tool is Lemma~\ref{lemma:sigmaD}. Next, we use this estimate and Lemma~\ref{lemma:bound_sigma_out}(b) to conclude that ${\hat{\kappa}_2(T(\bSigma))\leq \max(\kappa_2,3)}$.

\textbf{Estimating the denominator of $\boldsymbol{\hat{\kappa}_2(T(\bSigma))}$: }
We first bound from above $\|[T(\bSigma)]_{L_*,L_*^{\perp}}\|$:
\begin{equation}
\label{eq:bound_of_T_for_lemma6}
\|[T(\bSigma)]_{L_*,L_*^{\perp}}\|
\leq \sqrt{\sigma_{D-d}([T(\bSigma)]_{L_*^{\perp},L_*^{\perp}})\sigma_{d}([T(\bSigma)]_{L_*,L_*})}/2.
\end{equation}
This equation is
the condition of Lemma~\ref{lemma:sigmaD} for $T(\bSigma)$. This lemma thus implies the following lower bound on the denominator of $\hat{\kappa}_2(T(\bSigma))$:
\begin{align}\label{eq:kappa2denominator}
&\sigma_{D}( T(\bSigma))
\geq \frac{1}{2}\min\Big(\sigma_{D-d}([T(\bSigma)]_{L_*^{\perp},L_*^{\perp}}),\sigma_{d}([T(\bSigma)]_{L_*,L_*})\Big).
\end{align}

In order to derive \eqref{eq:bound_of_T_for_lemma6}, we first note that
\begin{align}\label{eq:kappa2nu}&
(\nu-2)\gamma\calS = \frac{\nu-2}{\nu} \cdot \nu\gamma\calS
\geq \frac{\nu-2}{\nu} \cdot 2 C_{\mathrm b} \kappa_2 \tilde{\calA}
\geq \frac{C_{\mathrm b}-2}{C_{\mathrm b}}\cdot 2C_{\mathrm b}\kappa_2\tilde{\calA}=
 2(C_{\mathrm b}-2) \tilde{\calA}\kappa_2\\\geq& \frac{2(C_{\mathrm b}-2)}{3} \tilde{\calA}\max(\kappa_2,3) \geq 16\tilde{\calA}\hat{\kappa}_2(\bSigma).\nonumber
\end{align}
The first inequality follows from $\nu\geq 2C_{\mathrm b}\frac{\kappa_2\tilde{\calA}}{\gamma \calS}$, which uses \eqref{eq:kappa1nu} (with only the second term in the sum as a lower bound of $\nu$) and the fact $\kappa_{\mathrm{in},*}\geq 1$.
The second inequality follows from the fact verified earlier that $\nu \geq C_{\mathrm b}$.
The fourth and last inequality follows from $C_{\mathrm b} \geq 26$ and the
condition
$\hat{\kappa}_2(\bSigma)\leq \max(\kappa_2,3)$
in \eqref{eq:overiterations}.

Next, we verify the following bound:
\begin{align}
\nonumber
&(2\|[T(\bSigma)]_{L_*,L_*^{\perp}}\|)^2 \leq
16 \, \|\bSigma_{+,\mathrm{out}}\|^2= 16\,\|\bSigma_{+,\mathrm{out}}\|\cdot \|\bSigma_{+,\mathrm{out}}\|\\
&\leq 16\,\tilde{\calA} \sigma_1\left(\bSigma_{L_*^\perp,L_*^\perp}\right)\cdot \|\bSigma_{+,\mathrm{out}}\| \leq \frac{(\nu-2)\,\gamma\,\calS}{\hat{\kappa}_2(\bSigma)}{\sigma_1\left(\bSigma_{L_*^\perp,L_*^\perp}\right)}\cdot \|\bSigma_{+,\mathrm{out}}\|
\label{eq:kappa2nu2}
\\
&=
(\nu-2) \,\gamma\,\calS\,\sigma_D(\bSigma) \,\|\bSigma_{+,\mathrm{out}}\|=\Big((\nu-2)\|\bSigma_{+,\mathrm{out}}\|\Big)\Big(\gamma\sigma_D(\bSigma)\calS\Big).
\nonumber
\end{align}
The first inequality follows from
\eqref{eq:widetildeT2}.
The second inequality
applies \eqref{eq:xout2}.
The third inequality applies \eqref{eq:kappa2nu}. The next equality applies the definition of $\hat{\kappa}_2(\bSigma)$.

We bound the first multiplicative term in the RHS of \eqref{eq:kappa2nu2} using the definition of $\nu$ in \eqref{eq:def_nu}, and then Lemma~\ref{lemma:matrixperturbation}(b) and~\eqref{eq:pertubation}:
\[
(\nu-2)\|\bSigma_{+,\mathrm{out}}\|
=
\sigma_d([\bSigma_{+,\mathrm{in}}]_{L_*,L_*})-2\|\bSigma_{+,\mathrm{out}}\|
\leq
\sigma_{d}([T(\bSigma)]_{L_*,L_*}).
\]
The second multiplicative term in the RHS of \eqref{eq:kappa2nu2}
was already bounded in
\eqref{eq:sigmaD-d}. Applying this and the above bound to \eqref{eq:kappa2nu2} concludes the desired bound in \eqref{eq:bound_of_T_for_lemma6}.

\textbf{Conclusion of the proof that $\boldsymbol{\hat{\kappa}_2(T(\bSigma))\leq}$max($\boldsymbol{\kappa_2}$,3): } We first note that \eqref{eq:kappa2denominator} implies
\begin{align}\label{eq:bound_with_3max}
\hat{\kappa}_2(T(\bSigma))
=\frac{\|[T(\bSigma)]_{L_*^{\perp},L_*^{\perp}}\|}{\sigma_D(T(\bSigma))}
\leq
2\max\left(
\frac{\|[T(\bSigma)]_{L_*^{\perp},L_*^{\perp}}\|}{\sigma_{D-d}([T(\bSigma)]_{L_*^{\perp},L_*^{\perp}})}\,,\,\frac{\|[T(\bSigma)]_{L_*^{\perp},L_*^{\perp}}\|}{\sigma_d([T(\bSigma)]_{L_*,L_*})}\right).
\end{align}
We bound the
first argument of the  maximum in \eqref{eq:bound_with_3max}, using  \eqref{eq:3rd_bound_lemma8_b}, and the
second argument  using \eqref{eq:lemmae}, where we recall we verified that $\nu>2$.

To further bound the RHS of \eqref{eq:3rd_bound_lemma8_b}, we bound $\bar{\sigma}_{\mathrm{tail}}(\bSigma_+)$ from below and $\|\bSigma_{+,\mathrm{out}}\|$ from above. For the former, we apply \eqref{eq:xout3}, the monotonicity of $\bar{\sigma}_{\mathrm{tail}}$ with respect to the Loewner order, and the definition of $\calS$:
\[
\bar{\sigma}_{\mathrm{tail}}(\bSigma_+)\geq \bar{\sigma}_{\mathrm{tail}}\Big(\sigma_D(\bSigma)\sum_{\bx\in\calX}\frac{\bx\bx^\top}{\|\bx\|^2}\Big)=\sigma_D(\bSigma)\bar{\sigma}_{\mathrm{tail}}\Big(\sum_{\bx\in\calX}\frac{\bx\bx^\top}{\|\bx\|^2}\Big) = \sigma_D(\bSigma)\calS.
\]
For the latter, we apply the first inequality in \eqref{eq:xout2}, which we write as  $\|\bSigma_{+,\mathrm{out}}\|\leq\|\bSigma_{L_*^{\perp},L_*^{\perp}}\|\,\tilde{\calA}$.
Applying the bounds $\nu \geq 4$, the bound $ \hat{\kappa}_2(\bSigma) \leq \max(\kappa_2,3) \leq 3 \kappa_2$ (see  \eqref{eq:overiterations}), and the inequality $\nu\geq 2C_{\mathrm b}\frac{\kappa_2\tilde{\calA}}{\gamma \calS}$
(already explained when justifying the first inequality in \eqref{eq:kappa2nu}), results in
\begin{equation}\label{eq:kappa2_bound}
\frac{4}{\nu-2}\hat{\kappa}_2(\bSigma)\tilde{\calA}\leq \frac{8}{\nu}\hat{\kappa}_2(\bSigma)\tilde{\calA} \leq \frac{8}{\nu}\max(\kappa_2,3)\tilde{\calA} \leq  \frac{12}{C_{\mathrm b}}\gamma\calS.
\end{equation}
Combining these two bounds with the definition of $\hat{\kappa}_2(\bSigma)$ and \eqref{eq:kappa2_bound} yields
\begin{equation}
\label{eq:check_a_b_ratio}
\frac{4\|\bSigma_{+,\mathrm{out}}\|}{(\nu-2)\gamma\bar{\sigma}_{\mathrm{tail}}(\bSigma_+)}\leq \frac{\frac{4}{\nu-2}\|\bSigma_{L_*^\perp,L_*^\perp}\|\tilde{\calA}}{\gamma \sigma_{D}(\bSigma)\calS}=\frac{\frac{4}{\nu-2}\hat{\kappa}_2(\bSigma)\tilde{\calA}}{\gamma \calS}\leq \frac{12}{C_{\mathrm b}}.
\end{equation}

Applying both \eqref{eq:lemmae} and \eqref{eq:3rd_bound_lemma8_b}
 to  \eqref{eq:bound_with_3max}, and then \eqref{eq:check_a_b_ratio}, results in
\begin{align}\label{eq:est_kappa2}
&\hat{\kappa}_2(T(\bSigma))
\leq
2\max\Big(1+\frac{12}{C_{\mathrm b}},\frac{2}{\nu-2}\Big).
\end{align}
The RHS of \eqref{eq:est_kappa2} can be
further bounded using the estimate $\nu \geq C_{\mathrm b}$ (so  $\nu$ can be replaced by $C_{\mathrm b}$). Applying the bound  $C_{\mathrm b}\geq 26$, the resulting RHS is bounded by $2\cdot(1+12/26)<3$ and consequently by $\max(3,\kappa_2)$. Therefore, the desired bound and the proof of the theorem are concluded. \qed

\begin{proof}[Proof of Corollary~\ref{cor:rate}]
The constant $C_{\mathrm b}$ enters the above proof of \eqref{eq:overiterations} only through the requirement $\hat{\kappa}_1(\bSigma)\geq\tilde{\kappa}_1$, that is, through \eqref{eq:est_nv_asumption} and \eqref{eq:kappa1nu}, and through the requirements $C_{\mathrm b}\geq 26$ (Part~II) and \eqref{eq:conditions_C} (Part~I). As noted after \eqref{eq:conditions_C}, the latter holds with $C_0$ replaced by $1<\rho<\dssnr/\gamma$ as soon as $C_{\mathrm b}\geq[14\dssnr+8\rho(\dssnr-\gamma)]/(\dssnr-\rho\gamma)$, which is guaranteed by the definition of $C_\rho$ in \eqref{eq:C_rho}. Hence, if \eqref{eq:kappa1} holds with $C_\rho$ in place of $C_{\mathrm b}$, then \eqref{eq:overiterations} holds with $C_0$ replaced by $\rho$, and the reduction of Section~\ref{sec:reduce_noiseless-thm} (including \eqref{eq:angle_via_nu}) yields \eqref{eq:explicit_rate} with $C_0$ replaced by $\rho$.

For the second statement, let $\tilde{\kappa}_1^{\rho}$ denote the quantity \eqref{eq:tildekappa1} with $C_{\mathrm b}$ replaced by $C_\rho$. By \eqref{eq:overiterations} and \eqref{eq:k1hatlargerk1tilde}, $\hat{\kappa}_1(\bSigma^{(k)})\geq C_0^k\kappa_1/\sigma_1(\bSigma_{\mathrm{in},*})\to\infty$, while $\hat{\kappa}_2(\bSigma^{(k)})\leq\max(\kappa_2,3)$ for all $k$. Let $k_\rho$ be such that $\hat{\kappa}_1(\bSigma^{(k)})\geq\tilde{\kappa}_1^{\rho}$ for all $k\geq k_\rho$. For such $k$, the hypotheses of \eqref{eq:overiterations} with $C_\rho$ in place of $C_{\mathrm b}$ are satisfied by $\bSigma^{(k)}$, and the argument of this section (with $C_\rho$ and $\rho$ in place of $C_{\mathrm b}$ and $C_0$) gives $\hat{\kappa}_1(\bSigma^{(k+1)})\geq\rho\,\hat{\kappa}_1(\bSigma^{(k)})$. Since the eventual growth bound holds for every $1<\rho<\dssnr/\gamma$, letting $\rho\uparrow\dssnr/\gamma$ proves \eqref{eq:liminf_snr_growth}. Finally, \eqref{eq:angle_via_nu} implies $\sin\angle(L^{(k+1)},L_*)\leq \calA/(\dssnr\sigma_d(\bSigma_{\mathrm{in},*})\hat{\kappa}_1(\bSigma^{(k)}))\leq c_\rho\,\rho^{-k}$ for $k\geq k_\rho$ and a constant $c_\rho$, so $\limsup_k\big(\sin\angle(L^{(k)},L_*)\big)^{1/k}\leq 1/\rho$. Letting $\rho\uparrow\dssnr/\gamma$ concludes the proof.
\end{proof}

\subsection{Main Ideas of the Proof of Theorem~\ref{thm:noisy}}
\label{sec:proof_thm_noisy_prelim}
The proof largely follows the proof of Theorem~\ref{thm:main}, with a few necessary modifications. In
Section~\ref{sec:reduce_noisy-thm} we reduce Theorem~\ref{thm:noisy} to verifying a technical property, stated in \eqref{eq:overiterations_noisy}, together with a one-step upper bound on the growth of $\hat{\kappa}_1$ (Lemma~\ref{lemma:growth_upper}).
In Section~\ref{sec:proof_idea_or_overiter_noisy}, we  explain, on a non-technical level, the basic idea behind proving \eqref{eq:overiterations_noisy} in light of the proof of \eqref{eq:overiterations}.
We complete the proof in Appendix~\ref{sec:proof_thm_noisy}, where we also specify choices for the different constants, which we refer to here.

\subsubsection{Reduction of the Theorem}
\label{sec:reduce_noisy-thm}
We first introduce some notation and definitions. We define $\phi:S_{+}(d) \rightarrow S_{+}(d)$:
\begin{equation}
\phi(\bA)=\bSigma_{\mathrm{in},*}^{-0.5}\bA\,\bSigma_{\mathrm{in},*}^{-0.5} \ \text{ for } \ \bA \in S_{+}(d),
\label{eq:def_phi}
\end{equation}
and consequently rewrite $\hat{\kappa}_1$ (originally defined in \eqref{eq:def_kappa_1_2} using the function $h$ defined in \eqref{eq:def_f}):
\begin{equation}
\hat{\kappa}_1(\bSigma)={\sigma_d(\phi([g_1(\bSigma)]_{L_*,L_*}))}/{\sigma_1\left(\bSigma_{L_*^\perp,L_*^\perp}\right)}.
\label{eq:def_kappa_1_2_new}
\end{equation}
When $\sigma_d(\phi([g_1(\bSigma)]_{L_*,L_*}))>0$, we further define
\begin{equation}\label{eq:def_kappa_3_new} \hat{\kappa}_3(\bSigma)={\sigma_1(\phi\left(\bSigma_{L_*,L_*}\right))}/{\sigma_d(\phi([g_1(\bSigma)]_{L_*,L_*}))}.\end{equation}

We note that for p.s.d.~matrices $\bX\in\R^{d\times d}$
\begin{equation}\label{eq:phi_relation}
\frac{\sigma_i(\bX)}{\sigma_1(\bSigma_{\mathrm{in},*})}
\leq \sigma_i(\phi(\bX))
\leq \frac{\sigma_i(\bX)}{\sigma_d(\bSigma_{\mathrm{in},*})},
\qquad 1\leq i\leq d.
\end{equation}
These eigenvalue bounds follow by comparing
$\bX^{1/2}\bSigma_{\mathrm{in},*}^{-1}\bX^{1/2}$ with scalar
multiples of $\bX$; this matrix and $\phi(\bX)$ have the same eigenvalues.
Consequently,
$
{\kappa}_3 / {\kappa_{\mathrm{in},*}} \leq \hat{\kappa}_3(\bSigma^{(0)})\leq \kappa_{\mathrm{in},*}\cdot{\kappa}_3$.
We remark that we use this more complicated version of $\hat{\kappa}_3(\bSigma)$ (instead of replacing  $\phi(\bSigma)$ by $\bSigma$ and thus obtaining that  $\kappa_3 = \hat{\kappa}_3(\bSigma^{(0)})$) in order to simplify the proof. On the other hand, we prefer leaving $\kappa_3$ in the main presentation (instead of replacing $\bSigma$ by $\phi(\bSigma)$) since it is more intuitive and easier to understand. We summarize the following relationships, which we refer to throughout the proof, where we recall that $\hat{\kappa}_2$ was defined in \eqref{eq:def_kappa_1_2}:
\begin{align}\label{eq:init_noisy}\sigma_1^{-1}(\bSigma_{\mathrm{in},*}){\kappa}_1\leq \hat{\kappa}_1(\bSigma^{(0)})\leq \sigma_d^{-1}(\bSigma_{\mathrm{in},*}){\kappa}_1 \ \text{(see \eqref{eq:hatkappa1_relation})},\quad \hat{\kappa}_2(\bSigma^{(0)})= {\kappa}_2,\quad \hat{\kappa}_3(\bSigma^{(0)})\leq \kappa_{\mathrm{in},*} \cdot {\kappa}_3.\end{align}

Next, we formulate the property we reduce the theorem to. It uses the above constants $\kappa_2$ and $\kappa_3$,
$C_0$ defined in \eqref{eq:C_value}, $\tilde{\kappa}_1$ defined in \eqref{eq:tildekappa1}, and the constants  $C_{\kappa_1}$, $C_{\kappa_2}$ and $C_{\kappa_3}$, specified in \eqref{eq:C_kappa_2}-\eqref{eq:C_kappa_1}.  This property replaces \eqref{eq:overiterations} in the proof of Theorem~\ref{thm:main} as follows:
\begin{align}\label{eq:overiterations_noisy}
\text{If }  C_{\kappa_1}/\sigma_1(\bSigma_{\mathrm{in},*})\geq \hat{\kappa}_1(\bSigma)\geq \tilde{\kappa}_1,  \hat{\kappa}_2(\bSigma)\leq C_{\kappa_2},    &\text{ and }
\hat{\kappa}_3(\bSigma)\leq C_{\kappa_3}, \\
\text{ then }
\hat{\kappa}_1(T(\bSigma))\geq C_0\hat{\kappa}_1(\bSigma), \hat{\kappa}_2(T(\bSigma))\leq C_{\kappa_2},    &\text{ and } \hat{\kappa}_3(T(\bSigma))\leq C_{\kappa_3}.\nonumber \end{align}
Thus $\hat\kappa_2$ and $\hat\kappa_3$ remain bounded through iteration $k_0$, while $\hat\kappa_1$ grows by at least $C_0$ at each step before the threshold $C_{\kappa_1}/\sigma_1(\bSigma_{\mathrm{in},*})$ is first exceeded. We define
\begin{equation}
\label{eq:property_k0}
k_0:=\min\Big\{k\geq 0:\ \hat{\kappa}_1(\bSigma^{(k)})>C_{\kappa_1}/\sigma_1(\bSigma_{\mathrm{in},*})\Big\}.
\end{equation}
This set is nonempty. Indeed, suppose that the threshold is never exceeded. The initialization bounds \eqref{eq:init_noisy}, the hypotheses of Theorem~\ref{thm:noisy}, and \eqref{eq:k1hatlargerk1tilde} then ensure that the hypotheses of \eqref{eq:overiterations_noisy} hold at iteration zero. By induction they hold at every subsequent iteration, so $\hat\kappa_1(\bSigma^{(k)})\geq C_0^k\kappa_1/\sigma_1(\bSigma_{\mathrm{in},*})\to\infty$, contradicting the fixed upper threshold. Thus $k_0$ is well defined.

Lastly, we show that \eqref{eq:overiterations_noisy} indeed implies the theorem. We first verify that $1\leq k_0\leq\lfloor\log_{C_0}(C_{\kappa_1}/\kappa_1)\rfloor+1$. By \eqref{eq:hatkappa1_relation} and the assumption $\kappa_1\leq C_{\kappa_1}/\kappa_{\mathrm{in},*}$,
\[
\hat{\kappa}_1(\bSigma^{(0)})\leq \frac{\kappa_1}{\sigma_d(\bSigma_{\mathrm{in},*})}=\frac{\kappa_1\kappa_{\mathrm{in},*}}{\sigma_1(\bSigma_{\mathrm{in},*})}\leq \frac{C_{\kappa_1}}{\sigma_1(\bSigma_{\mathrm{in},*})},
\]
so $k_0\geq 1$. By \eqref{eq:init_noisy}, the assumptions $\kappa_2\leq 7=C_{\kappa_2}$ and $\kappa_3\leq C_{\kappa_3}/\kappa_{\mathrm{in},*}$, \eqref{eq:k1hatlargerk1tilde} and the definition of $k_0$, the hypotheses of \eqref{eq:overiterations_noisy} hold for $\bSigma^{(0)}$, and by induction they hold for $\bSigma^{(j)}$ for every $0\leq j\leq k_0-1$. Consequently,
\begin{equation}
\label{eq:growth_until_k0}
\hat{\kappa}_1(\bSigma^{(j)})\geq C_0^{j}\hat{\kappa}_1(\bSigma^{(0)})\geq C_0^{j}\kappa_1/\sigma_1(\bSigma_{\mathrm{in},*}), \quad 0\leq j\leq k_0,
\end{equation}
where the last inequality uses \eqref{eq:hatkappa1_relation}. Applying \eqref{eq:growth_until_k0} with $j=k_0-1$ and $\hat{\kappa}_1(\bSigma^{(k_0-1)})\leq C_{\kappa_1}/\sigma_1(\bSigma_{\mathrm{in},*})$ yields $C_0^{k_0-1}\kappa_1\leq C_{\kappa_1}$, that is, $k_0-1\leq\log_{C_0}(C_{\kappa_1}/\kappa_1)$, which is the claimed upper bound on $k_0$.

Next, we prove part (a) of the theorem. Fix $1\leq k\leq k_0$ and let $\bSigma=\bSigma^{(k-1)}$, so that the hypotheses of \eqref{eq:overiterations_noisy} hold for $\bSigma$. We use the decomposition $\bSigma_+=\bSigma_{\diamond,\mathrm{in}}+\bSigma_{\diamond,\mathrm{out}}$ and the corresponding ratio $\nu=\sigma_d(\bSigma_{\diamond,\mathrm{in}})/\|\bSigma_{\diamond,\mathrm{out}}\|$, as redefined in Section~\ref{sec:proof_idea_or_overiter_noisy}. Since $\bSigma_{\diamond,\mathrm{in}}$ has range in $L_*$ and $[\bSigma_{\diamond,\mathrm{out}}]_{L_*,L_*}=[\bSigma_{+,\mathrm{out}}]_{L_*,L_*}\psdgeq \bm{0}$, we have $[\bSigma_+]_{L_*,L_*^\perp}=[\bSigma_{\diamond,\mathrm{out}}]_{L_*,L_*^\perp}$, $[\bSigma_+]_{L_*^\perp,L_*^\perp}=[\bSigma_{\diamond,\mathrm{out}}]_{L_*^\perp,L_*^\perp}$ and $\sigma_d([\bSigma_+]_{L_*,L_*})\geq \sigma_d(\bSigma_{\diamond,\mathrm{in}})$. Moreover, \eqref{eq:est_nu_noisy}, \eqref{eq:noisy_conclusion} and $1.1^3<1.4$ give
\begin{equation}
\label{eq:nu_diamond_lower}
\nu\geq \frac{\hat{\kappa}_1(\bSigma)\,\dssnr\,\sigma_d(\bSigma_{\mathrm{in},*})}{1.4\,\calA},
\end{equation}
and \eqref{eq:kappa1nu_noisy}, \eqref{eq:R_lower_noisy}, and $C_{\mathrm b}\geq70$ give $\nu\geq70/1.4=50$.
The subspace $L^{(k)}$ is the span of the top $d$ eigenvectors of $\bSigma_+$. We apply Lemma~\ref{lemma:blockDK} to $\bSigma_+$: we have $\sigma_d([\bSigma_+]_{L_*,L_*})\geq\sigma_d(\bSigma_{\diamond,\mathrm{in}})$, while $\|[\bSigma_+]_{L_*,L_*^\perp}\|$ and $\sigma_1([\bSigma_+]_{L_*^\perp,L_*^\perp})$ are at most $\|\bSigma_{\diamond,\mathrm{out}}\|<\sigma_d(\bSigma_{\diamond,\mathrm{in}})$. Lemma~\ref{lemma:blockDK} therefore gives $\sin\angle(L^{(k)},L_*)\leq1/(\nu-1)$, which together with the bound $\nu\geq50$ implies
\begin{equation}
\label{eq:angle_via_nu_noisy}
\sin\angle(L^{(k)},L_*)\leq \frac{1}{\nu-1}\leq\frac{50}{49\nu}\leq \frac{10\,\calA}{7\,\dssnr\,\sigma_d(\bSigma_{\mathrm{in},*})\,\hat{\kappa}_1(\bSigma^{(k-1)})}.
\end{equation}
The last inequality uses \eqref{eq:nu_diamond_lower} and $1.4\cdot50/49=10/7$.
Combining \eqref{eq:angle_via_nu_noisy} with \eqref{eq:growth_until_k0} (for $j=k-1$) and the definition of $\kappa_{\mathrm{in},*}$ in \eqref{eq:defkinstar} proves part (a).

Finally, we prove part (b). Lemma~\ref{lemma:growth_upper} in Appendix~\ref{sec:growth_upper}, applied to $\bSigma=\bSigma^{(k_0-1)}$, gives $\hat{\kappa}_1(\bSigma^{(k_0)})\leq G\,\hat{\kappa}_1(\bSigma^{(k_0-1)})$ with $G=\dssnr\big(9C_{\kappa_3}\calR/(\gamma\calA)+0.11\big)$. Since $\hat{\kappa}_1(\bSigma^{(k_0)})>C_{\kappa_1}/\sigma_1(\bSigma_{\mathrm{in},*})$ by the definition of $k_0$, \eqref{eq:angle_via_nu_noisy} with $k=k_0$ yields
\[
\sin\angle(L^{(k_0)},L_*)\leq \frac{10\,\calA\,G}{7\,\dssnr\,\sigma_d(\bSigma_{\mathrm{in},*})\,\hat{\kappa}_1(\bSigma^{(k_0)})}
<\frac{10\,\calA\,G\,\kappa_{\mathrm{in},*}}{7\,\dssnr\,C_{\kappa_1}}
=\frac{\kappa_{\mathrm{in},*}}{C_{\kappa_1}}\Big(\frac{90\,C_{\kappa_3}\calR}{7\gamma}+\frac{11}{70}\calA\Big),
\]
which implies part (b). For fixed remaining parameters, the first entry in the minimum in \eqref{eq:C_kappa_1} tends to infinity as $\epsilon\downarrow0$, while the second is a positive constant. Thus, for all sufficiently small $\epsilon$, $C_{\kappa_1}=n_0\min\{\dssnr/\gamma-1,1\}\,\epsilon^{-1}/(100C_{\kappa_3}n_1(D-d))$, which proves the last claim of the theorem.
\subsubsection{Proof strategy}\label{sec:proof_idea_or_overiter_noisy}

The proof of \eqref{eq:overiterations_noisy} mostly follows the proof of \eqref{eq:overiterations}, which decomposes $\bSigma_{+}=\sum_{\bx\in\calX}{\bx\bx^\top}/{\bx^\top\bSigma^{-1}\bx}$ into
$\bSigma_{+,\mathrm{in}}$ and $\bSigma_{+,\mathrm{out}}$ (see \eqref{eq:Sigma+}).
We recall that the major idea of the proof of \eqref{eq:overiterations} is based on two properties of this decomposition. First, $\bSigma_{+,\mathrm{in}}$ is a p.s.d.~matrix whose range lies in $L_*$. Second, the lower bound on $\nu = \sigma_d(\bSigma_{+,\mathrm{in}}) / \|\bSigma_{+,\mathrm{out}}\|$, established in \eqref{eq:est_nu}, ensures that $\bSigma_{+,\mathrm{out}}$ is ``small'' relative to $\bSigma_{+,\mathrm{in}}$. However, we cannot directly apply the proof of \eqref{eq:overiterations} to conclude \eqref{eq:overiterations_noisy}, since the first property does not hold in the current case.

To adapt the proof of \eqref{eq:overiterations}, we apply a slightly different decomposition:
\[
\bSigma_{\diamond,\mathrm{in}}=\begin{pmatrix}
[\bSigma_{+,\mathrm{in}}]_{L_*,L_*} & 0\\
0 & 0
\end{pmatrix},\qquad
\bSigma_{\diamond,\mathrm{out}}=\bSigma_+-\bSigma_{\diamond,\mathrm{in}}.
\]
The matrix $\bSigma_{\diamond,\mathrm{in}}$ is positive semidefinite with range contained in $L_*$, so the first property holds. The remainder $\bSigma_{\diamond,\mathrm{out}}$ need not be positive semidefinite, although both of its diagonal blocks are. The argument below controls its spectral norm.
We establish the second property by showing that
the spectral norm of $\bSigma_{\diamond,\mathrm{out}}$ is controlled by the smallest eigenvalue of  $\bSigma_{\diamond,\mathrm{in}}$.
Our proof of this property relies on Lemma~\ref{lemma:xplus_noisy}, whose statement and proof are given in Appendix~\ref{sec:proof_lemma_10}.
This lemma adapts parts (a) and (b) of Lemma~\ref{lemma:xplus} to the noisy setting.
We remark that the total-data lower bound \eqref{eq:xout3} in Lemma~\ref{lemma:xplus}, as well as all other preliminary lemmas in Section~\ref{sec:tech_lemma_thm1} except Lemma~\ref{lemma:bound_sigma_out}, do not involve $\bSigma_{+,\mathrm{in}}$ or $\bSigma_{+,\mathrm{out}}$ and thus remain unchanged.
Moreover, a careful examination of the proof of Lemma~\ref{lemma:bound_sigma_out} shows that it still holds when $\bSigma_{+,\mathrm{in}}$ and $\bSigma_{+,\mathrm{out}}$ are replaced with their noisy counterparts $\bSigma_{\diamond,\mathrm{in}}$ and $\bSigma_{\diamond,\mathrm{out}}$ (and $\nu$ is redefined accordingly): the proof only uses that $\bSigma_{\diamond,\mathrm{in}}\in S_+$ has range in $L_*$ and that $\bSigma_+=\bSigma_{\diamond,\mathrm{in}}+\bSigma_{\diamond,\mathrm{out}}$, except for the last term of \eqref{eq:lemmab}, whose derivation uses the trace bound in \eqref{eq:xout2} and therefore acquires the factor $1+c_4(\bSigma,\epsilon)$ of its noisy counterpart \eqref{eq:xout2_noisy}.

\subsection{Proof of Theorem~\ref{thm:haystack}}
\label{sec:thm_haystack}
We first justify the passage from population quantities to sample quantities. Let $P$ be the limiting Gaussian mixture and let $Q$ be the distribution of $\bx\bx^\top/\|\bx\|^2$ for $\bx\sim P$. The convergence of $\dssnr_N$ implies that the inlier fraction converges to $d\dssnr/(D-d+d\dssnr)$ and that both sample counts tend to infinity. Since $\dssnr<1$, this limiting fraction is less than $d/D$. Every nonzero proper subspace $V$ has $P(V)=0$ unless it contains $L_*$, in which case $P(V)<d/D\leq\dim(V)/D$. Thus the strict subspace-mass condition for existence and uniqueness of the population TME holds; see~\cite[Theorem~4.9, Case~0]{DPS15}. The empirical angular matrix distributions converge weakly to $Q$ almost surely along the two fixed sample streams. Applying~\cite[Theorem~6.1, Case~0]{DPS15} on this probability-one event gives eventual existence and convergence of the determinant-normalized empirical TME; continuous rescaling gives convergence with trace normalization as well. Applying the same argument in the intrinsic coordinates of $L_*$ proves consistency of the inlier TME. Finally, $D-d>2$ ensures the integrability of the outlier alignment summands used in Appendix~\ref{sec:proof_example_haystack}. The strong law therefore gives almost-sure convergence of $\calA$ and $\calR$ to their population values.

It thus remains to study the population TME, denoted by $\bSigma_*$. We prove that it supplies a successful STE initialization when $\dssnr$ is sufficiently close to one, and that its leading $d$-dimensional eigenspace differs from $L_*$ when $\bSigma^{(\mathrm{out})}_{L_*,L_*^\perp}\neq\bm0$.

\textbf{TME does not recover $L_*$ if $\bSigma^{(\mathrm{out})}_{L_*,L_*^\perp}\neq 0$ and $\dssnr< 1$:}
We first state a lemma, which clarifies properties of the TME solution under the generalized haystack model, and is proved in Appendix~\ref{sec:proof_lemmas}.  Its first part is used in the current case and the second part is used below in the last case. We recall that the TME solutions are uniquely defined up to an arbitrary scaling.
\begin{lemma}\label{lemma:special}
For the population generalized haystack model with fixed dimensions and covariance matrices, suppose that $0<\dssnr<1$ and $\bSigma^{(\mathrm{out})}_{L_*,L_*^\perp}=\bm0$. The following properties hold:\\
(a) Any TME solution $\bSigma_*$ (defined up to scaling) is full-rank and satisfies
\[
[\bSigma_*]_{L_*,L_*^\perp}=\bm{0} \ \text{ and } \ [\bSigma_*]_{L_*^\perp,L_*}=\bm{0}.
\]
(b) There exist constants $c_{\dssnr,0}<1$ and $c_{\dssnr}>0$, depending only on $d$, $D$ and the condition numbers of $\bSigma^{(\mathrm{in})}$ and $\bSigma^{(\mathrm{out})}$, such that for $\dssnr\in [c_{\dssnr,0},1)\cap(0,1)$, any TME solution $\bSigma_*$ satisfies the following conditions:
$$[\bSigma_*]_{L_*^\perp,L_*^\perp}=c\,\bSigma^{(\mathrm{out})}_{L_*^\perp,L_*^\perp}, \hspace{0.1in}  \ \text{ and } \ \frac{\sigma_d([\bSigma_*]_{L_*,L_*})}{\sigma_1([\bSigma_*]_{L_*^\perp,L_*^\perp})}\geq \frac{c_{\dssnr}}{1-\dssnr},
$$
where $c>0$ is a scalar that depends on the arbitrary scaling of $\bSigma_*$.
\end{lemma}

We let $\mathbf{J}_0=\bSigma^{(\mathrm{out})}_{L_*,L_*^\perp}(\bSigma^{(\mathrm{out})}_{L_*^\perp,L_*^\perp})^{-1}$, \[
\mathbf{J}=\begin{pmatrix}
\bI_{d\times d} &\hspace{0.1in} -\mathbf{J}_0 \\
\bm{0}_{(D-d)\times d} &\hspace{0.1in} \bI_{(D-d)\times (D-d)}
\end{pmatrix},
\text{ so } \
\mathbf{J}^{-1}=\begin{pmatrix}
\bI_{d\times d} &\hspace{0.1in} \mathbf{J}_0 \\
\bm{0}_{(D-d)\times d} &\hspace{0.1in} \bI_{(D-d)\times (D-d)}
\end{pmatrix},
\]
and $\tilde{\calX}=\{\mathbf{J}\bx:\bx\in\calX\}$. We view $\tilde{\calX}$ as a set of inliers and outliers, with  ``inlying covariance'' $\tilde{\bSigma}^{(\mathrm{in})}=\mathbf{J}\bSigma^{(\mathrm{in})}\mathbf{J}^\top$ and ``outlying covariance''  $\tilde{\bSigma}^{(\mathrm{out})}=\mathbf{J}\bSigma^{(\mathrm{out})}\mathbf{J}^\top$. We note that
\begin{equation}\label{eq:A_transform}
\mathbf{J}\bSigma\mathbf{J}^\top = \begin{pmatrix}
\bSigma_{L_*,L_*}-\bSigma_{L_*,L_*^\perp}\mathbf{J}_0^\top-\mathbf{J}_0\bSigma_{L_*^\perp,L_*} +\mathbf{J}_0\bSigma_{L_*^\perp,L_*^\perp}\mathbf{J}_0^\top &\hspace{0.2in} \bSigma_{L_*,L_*^\perp}-\mathbf{J}_0\bSigma_{L_*^\perp,L_*^\perp}  \\
\bSigma_{L_*^\perp,L_*}-\bSigma_{L_*^\perp,L_*^\perp}\mathbf{J}_0^\top &\hspace{0.2in} \bSigma_{L_*^\perp,L_*^\perp}
\end{pmatrix}.
\end{equation}
Applying  $\tilde{\bSigma}^{(\mathrm{out})}=\mathbf{J}{\bSigma}^{(\mathrm{out})}\mathbf{J}^\top$ and the above definition of $\mathbf{J}_0$ results in
\begin{equation}\label{eq:TME_diagonal}\text{$\tilde{\bSigma}^{(\mathrm{out})}_{L_*,L_*^\perp}=\bm{0}_{d\times (D-d)}$ and $\tilde{\bSigma}^{(\mathrm{out})}_{L_*^\perp,L_*}=\bm{0}_{(D-d)\times d}$.}\end{equation}

Applying Lemma~\ref{lemma:special}(a),  the  TME solution for the set $\tilde{\calX}$, denoted by $\tilde{\bSigma}_*$, also has a block diagonal structure, that is, $[\tilde{\bSigma}_*]_{L_*,L_*^\perp}=\bm{0}_{d\times (D-d)}$ and $[\tilde{\bSigma}_*]_{L_*^\perp,L_*}=\bm{0}_{(D-d)\times d}$.

The $\mathbf{J}$-equivariance property of TME (discussed later in \eqref{eq:tme_equivariance}) implies ${\bSigma}_*=\mathbf{J}^{-1}\tilde{\bSigma}_*\mathbf{J}^{-\top}$, and a similar calculation to \eqref{eq:A_transform} shows that $[{\bSigma}_*]_{L_*,L_*^\perp}=[\tilde{\bSigma}_*]_{L_*,L_*^\perp}+\mathbf{J}_0[{\bSigma}_*]_{L_*^\perp,L_*^\perp}=\mathbf{J}_0[{\bSigma}_*]_{L_*^\perp,L_*^\perp}$. We note that $\mathbf{J}_0[{\bSigma}_*]_{L_*^\perp,L_*^\perp}$ is nonzero since Lemma~\ref{lemma:special}(a) implies that $\tilde{\bSigma}_*$ is full rank and the assumption  $\bSigma^{(\mathrm{out})}_{L_*^\perp,L_*}\neq 0$ implies that  $\mathbf{J}_0$ is nonzero.

If the top $d$ eigenvectors of a symmetric matrix span $L_*$, its cross block relative to $L_*\oplus L_*^\perp$ vanishes. Thus no leading $d$-dimensional eigenspace of $\bSigma_*$ equals $L_*$. Moreover, if the empirical leading projectors converged to $\bP_{L_*}$, their commutation with the empirical TME matrices and matrix convergence would imply that $\bP_{L_*}$ commutes with $\bSigma_*$. This contradicts the nonzero cross block. Hence the empirical TME subspace cannot converge to $L_*$, even if the limiting matrix has an eigenvalue tie.

\textbf{TME+STE recovers $L_*$ when $\delta_0 \leq \dssnr< 1$.}

Applying Lemma~\ref{lemma:special}(b) to the $\mathbf{J}$-transformed solution, $\tilde{\bSigma}_*$, results in
\begin{equation}
{\sigma_d([\tilde{\bSigma}_*]_{L_*,L_*})}/{\sigma_1([\tilde{\bSigma}_*]_{L_*^\perp,L_*^\perp})}\geq \frac{c_{\dssnr}}{1-\dssnr},\quad \text{for } \dssnr\in[c_{\dssnr,0},1)\cap(0,1),
\label{eq:bound_cond_dssnr}
\end{equation}
In orthonormal coordinates adapted to $L_*\oplus L_*^\perp$, the congruence ${\bSigma}_*=\mathbf J^{-1}\tilde{\bSigma}_*\mathbf J^{-\top}$ and the block-diagonal structure of $\tilde{\bSigma}_*$ give
\[
[g_1(\bSigma_*)]_{L_*,L_*}=[\tilde{\bSigma}_*]_{L_*,L_*}+\mathbf J_0[\tilde{\bSigma}_*]_{L_*^\perp,L_*^\perp}\mathbf J_0^\top-(\mathbf J_0[\tilde{\bSigma}_*]_{L_*^\perp,L_*^\perp})[\tilde{\bSigma}_*]_{L_*^\perp,L_*^\perp}^{-1}([\tilde{\bSigma}_*]_{L_*^\perp,L_*^\perp}\mathbf J_0^\top)=[\tilde{\bSigma}_*]_{L_*,L_*}.
\]
Combining this formula and the intermediate result from above that $[\tilde{\bSigma}_*]_{L_*^\perp,L_*^\perp}=[{\bSigma}_*]_{L_*^\perp,L_*^\perp}$  results in
\begin{equation*}
\kappa_1=\frac{\sigma_d(g_1(\bSigma_*))}{\sigma_1([\bSigma_*]_{L_*^\perp,L_*^\perp})}
=\frac{\sigma_d([\tilde{\bSigma}_*]_{L_*,L_*})}{\sigma_1([\tilde{\bSigma}_*]_{L_*^\perp,L_*^\perp})}\geq \frac{c_{\dssnr}}{1-\dssnr}.
\end{equation*}

Similarly, the congruence bound $\sigma_D({\bSigma}_*)\geq \sigma_D(\tilde{\bSigma}_*)/\|\mathbf{J}\|^2$ and the first part of Lemma~\ref{lemma:special}(b) (applied to the $\mathbf{J}$-transformed solution  $\tilde{\bSigma}_*$) imply
\begin{align}\label{eq:kappa22}
\kappa_2=&\frac{\sigma_1([\bSigma_*]_{L_*^\perp,L_*^\perp})}{\sigma_D(\bSigma_*)}\leq \|\mathbf{J}\|^2\frac{\sigma_1([\tilde{\bSigma}_*]_{L_*^\perp,L_*^\perp})}{\sigma_D(\tilde{\bSigma}_*)}=\|\mathbf{J}\|^2\frac{\sigma_1([\tilde{\bSigma}_*]_{L_*^\perp,L_*^\perp})}{\min(\sigma_d([\tilde{\bSigma}_*]_{L_*,L_*}),\sigma_{D-d}([\tilde{\bSigma}_*]_{L_*^\perp,L_*^\perp}))}\\
\leq &\|\mathbf{J}\|^2\max\left(\frac{\sigma_1(\tilde{\bSigma}^{(\mathrm{out})}_{L_*^\perp,L_*^\perp})}{\sigma_{D-d}(\tilde{\bSigma}^{(\mathrm{out})}_{L_*^\perp,L_*^\perp})}, \frac{1-\dssnr}{c_{\dssnr}}\right).\nonumber
\end{align}
We note that the above formulas, the upper bounds of $\|\mathbf{J}\|$ and $\|\mathbf{J}^{-1}\|$ by
\[
\max(\|\mathbf{J}\|,\|\mathbf{J}^{-1}\|)\leq \|\bI_{D\times D}\|+\left\|\bSigma^{(\mathrm{out})}_{L_*,L_*^\perp}\left(\bSigma^{(\mathrm{out})}_{L_*^\perp,L_*^\perp}\right)^{-1}\right\|\leq 1+\kappa_{\mathrm{out}},
\]  and
\eqref{eq:bound_cond_dssnr} imply that $\kappa_1$ converges to $\infty$ as $\dssnr\rightarrow 1$, whereas $\kappa_2$
is bounded.

It remains to investigate the RHS of \eqref{eq:kappa1}, that is,  \[\frac{C_{\mathrm b}d\kappa_{\mathrm{in},*}\tilde{\calA}}{n_1}\!\left(\!\frac{\tilde{\calA}}{\frac{n_1}{d}\!-\!\gamma\frac{n_0}{D-d}}\!+\!\frac{\kappa_2\tilde{\calA}}{\gamma \calS}(1+\kappa_{\mathrm{in},*})\!\right)
\!=\!C_{\mathrm b}\kappa_{\mathrm{in},*}\frac{\tilde{\calA}d}{n_1}\!\left(\!\frac{\frac{\tilde{\calA}d}{n_1}}{1\!-\!\gamma/\dssnr}+\frac{\tilde{\calA}}{\calS}\frac{\kappa_2}{\gamma}(1+\kappa_{\mathrm{in},*})\!\right).\]
By the analysis of the generalized haystack model in Section~\ref{sec:models},  ${\tilde{\calA}}/{\calS}$ and $\tilde{\calA}d/{n_1}$ have upper bounds that only depend on  $\kappa_{\mathrm{in},*}$,
 $\kappa_{\mathrm{out}}$, $\lambda$ and $\dssnr$ as follows:
\[
\frac{\tilde{\calA}}{\calS}\leq \kappa_{\mathrm{out}}^2\frac{D}{D-d-2} \leq \frac{\kappa_{\mathrm{out}}^2}{\lambda},
\]
where  the first inequality applies $\tilde{\calA}=\frac{n_0}{D-d}{\calA}$, \eqref{eq:calA3_2} and \eqref{eq:calS1}, and the second inequality applies the assumption $d\leq (1-\lambda)D-2$, which implies $D-d-2\geq \lambda D$. Also, applying $\tilde{\calA}d/n_1=\calA/\dssnr$ (which follows from the definitions of $\tilde{\calA}$ and $\dssnr$) and \eqref{eq:calA3_2} we have
\[
\frac{\tilde{\calA}d}{n_1}=\frac{\calA}{\dssnr}\leq \frac{\kappa_{\mathrm{out}}}{\dssnr}\cdot\frac{D-d}{D-d-2}\leq \frac{3\kappa_{\mathrm{out}}}{\dssnr}\leq \frac{3\kappa_{\mathrm{out}}}{\delta_0},
\]
where the second inequality follows from the assumption $D-d>2$, so $D-d\geq 3$ and $\frac{D-d}{D-d-2}\leq 3$, and the last one from $\dssnr\geq\delta_0$.

In addition, $\kappa_2$ is also bounded above using \eqref{eq:kappa22}. Since $\gamma<\delta_0-c_0\leq\dssnr-c_0$, we have $1-\gamma/\dssnr \geq c_0/\dssnr\geq c_0$, so the RHS of \eqref{eq:kappa1} has an upper bound that depends only on $\kappa_{\mathrm{in},*}, \kappa_{\mathrm{out}}, \lambda$ and $c_0$ (through $\delta_0$) and remains bounded above as $\dssnr\rightarrow 1$.

The cutoff can be chosen independently of $d,D$ under the stated $\lambda$ restriction. The transformed outlier covariance is block diagonal, with blocks equal to the Schur complement of the original tail block and the tail block itself. Both blocks have eigenvalues between the smallest and largest eigenvalues of the original covariance: for the Schur complement, this follows by compressing the inverse covariance and inverting. Hence the transformed outlier condition number is at most $\kappa_{\mathrm{out}}$. With
\[
M_h:=\kappa_{\mathrm{in}}^2\kappa_{\mathrm{out}}^3,
\]
the whitening calculation in Lemma~\ref{lemma:special}(b), including its return to the original coordinates, gives
\[
\kappa_1\geq\frac{1}{8M_h(1-\dssnr)}\qquad(\dssnr\geq3/4).
\]
Here $1-D/(4d)\leq3/4$ and $D/(D-d)\geq1$ remove the dimension dependence from the lemma's constants. If also $\dssnr\geq1-1/(8M_h)$, then $\kappa_1\geq1$ and \eqref{eq:kappa22} gives $\kappa_2\leq(1+\kappa_{\mathrm{out}})^2\kappa_{\mathrm{out}}$. Moreover, \eqref{eq:calA_R} implies $\calA\leq3\kappa_{\mathrm{out}}$ and $\calR\leq\kappa_{\mathrm{out}}^2/\lambda$. For $\gamma\geq c_0$ and $\dssnr-\gamma\geq c_0$, the RHS of \eqref{eq:kappa1} is therefore at most
\[
\begin{aligned}
H_h:=&\;6\kappa_{\mathrm{in}}\kappa_{\mathrm{out}}(28/c_0+2)\\
&\quad\times\left(\frac{3\kappa_{\mathrm{out}}}{c_0}
+\frac{(1+\kappa_{\mathrm{out}})^2\kappa_{\mathrm{out}}^3(1+\kappa_{\mathrm{in}})}{\lambda c_0}\right).
\end{aligned}
\]
Choosing $\delta_0$ strictly above
$\max\{3/4,2c_0,1-1/(8M_h),1-1/(16M_hH_h)\}$
makes the population LHS of \eqref{eq:kappa1} larger than twice its RHS, and $\delta_0$ depends only on $\kappa_{\mathrm{in}}$, $\kappa_{\mathrm{out}}$, $c_0$ and $\lambda$.

For the population quantities, $\kappa_{\mathrm{in},*}=\kappa_{\mathrm{in}}$, and the LHS of \eqref{eq:kappa1} diverges as $\dssnr\uparrow1$, whereas its RHS remains bounded for $\gamma$ bounded away from zero and from $\dssnr$. Choose $\delta_0\in(2c_0,1)$ large enough that the population LHS exceeds twice the population RHS for all $\delta_0\leq\dssnr<1$ and $c_0\leq\gamma\leq\delta_0-c_0$. We may also increase $\delta_0$ so that Lemma~\ref{lemma:special}(b) applies. The almost-sure convergence established at the start of the proof, the limit $\dssnr_N\to\dssnr$, and continuity of the quantities entering the basin condition at positive definite TME solutions then imply \eqref{eq:kappa1} for all sufficiently large $N$. The convergence is uniform in $\gamma\in[c_0,\delta_0-c_0]$ because the denominators stay bounded away from zero. Theorem~\ref{thm:main} now proves the claimed convergence of TME-initialized STE.

\section{Conclusion}
\label{sec:conclusion}
We proved that STE converges linearly from a sufficiently accurate initialization under deterministic geometric conditions, including when $\gamma<\dssnr<1$. For $\bSigma^{(0)}=\bP_{\hat L}+\alpha\bI$, the basin condition can be stated directly through the initial principal angle and $\alpha$. The proof tracks a Schur-complement SNR and gives an asymptotic convergence-factor bound of $\gamma/\dssnr$.

We also proved an $O(\epsilon)$ noisy-inlier error at an explicitly bounded iterate and showed that TME can initialize STE successfully under the asymptotic generalized haystack model when TME alone does not recover $L_*$.

The basin condition depends on the unknown target subspace and inlier--outlier partition. Its constants have not been optimized, and we do not know whether the order $\dssnr$ of the admissible initial angle under the generalized haystack model is optimal. Larger certified basins based on cross-alignment, control of noisy iterates beyond the stated stopping index, inexpensive initializers under weaker assumptions, and selection of the successful list candidate remain open questions.

\section*{Acknowledgments}
GL was supported by NSF awards DMS-2152766 and DMS-2427955, and TZ was supported by NSF award DMS-2318926.

\paragraph*{Declaration of generative AI and AI-assisted technologies in the manuscript preparation process.}
The results, proofs, and manuscript were developed by the authors before the use of AI tools. At a late stage of manuscript preparation, the authors used Claude and ChatGPT for supplementary checks of selected arguments and for suggestions on minor mathematical and expository refinements. The authors reviewed and independently verified all incorporated suggestions and take full responsibility for the content of the publication.

\appendix
\bigskip
\begin{center}
   \LARGE\bfseries {Complete Proofs and Other Technical Details}
\end{center}

These appendices provide the technical details deferred from the main text: Appendix~\ref{sec:calculate_kappa1} discusses the auxiliary parameters and quantities, Appendix~\ref{sec:proof_thm_noisy} completes the proof of Theorem~\ref{thm:noisy}, Appendix~\ref{sec:proof_lemmas} proves the auxiliary lemmas, and Appendix~\ref{sec:initialization} presents the list-decodable initialization application.

\section{Discussion: Auxiliary Parameters and Quantities}
\label{sec:calculate_kappa1}
\subsection{Verifying the Estimates for \texorpdfstring{$\kappa_1$, $\kappa_2$ and $\kappa_3$}{κ1, κ2 and κ3} under the Model in \eqref{eq:STE_init}}
We first verify the formula for $\kappa_1$ stated in \eqref{eq:kappa1_angle}. We use the principal angles between $\hat{L}$ and $L_*$ (see Sect.~3.2.1 of \cite{lp_recovery_part1_11}), where we recall that if $\mathbf{W}_* \in O(D,d)$ spans $L_*$ and $\hat{\mathbf{W}} \in O(D,d)$ spans $\hat{L}$, then the principal angles $\{\theta_i\}_{i=1}^d$ are the arccosines of the singular values of $\mathbf{W}_*^\top \hat{\mathbf{W}}$; we index the angles in decreasing order, so that $\theta_1$ is the largest principal angle. The corresponding principal vectors $\{\bu_i\}_{i=1}^d$ of $L_*$ and $\{\bv_i\}_{i=1}^d$ of $\hat L$ are obtained by multiplying the left and right singular vectors of $\mathbf W_*^\top\hat{\mathbf W}$ by $\mathbf W_*$ and $\hat{\mathbf W}$, respectively, and indexing them accordingly. Let $d'\leq d$ be the largest index such that $\theta_{d'}>0$ (equivalently, $d'=d-\dim(\hat{L}\cap L_*)$), and for $1\leq i\leq d'$, write $\bv_i=\cos\theta_i\bu_i+\sin\theta_i\bw_i$, where $\bw_i\in L_*^\perp$ (recall that $\bu_i\in L_*$). In addition, for indices $d'+1 \leq i \leq d$, we have $\bv_i = \bu_i$. We express $\bSigma^{(0)}$ in a basis partitioned into four orthogonal blocks: $\{\bu_1,\cdots,\bu_{d'}\}$, $\{\bu_{d'+1},\cdots,\bu_d\}=\{\bv_{d'+1},\cdots,\bv_d\}$, $\{\bw_1,\cdots,\bw_{d'}\}$, and $(L_*\oplus \hat{L})^\perp$:
\begin{align*}
&\bSigma^{(0)}=\sum_{i=1}^d\bv_i\bv_i^\top+\alpha\bI\\=&\begin{pmatrix}
\mathrm{Diag}(\{\cos^2\theta_i +\alpha\}_{i=1}^{d'})& 0 & \mathrm{Diag}(\{\cos\theta_i\sin\theta_i\}_{i=1}^{d'})& 0 \\
0& (1+\alpha) \bI_{d-d'}&0&0\\
\mathrm{Diag}(\{\cos\theta_i\sin\theta_i\}_{i=1}^{d'}) & 0&\mathrm{Diag}(\{\sin^2\theta_i +\alpha\}_{i=1}^{d'}) & 0\\
0&0&0&\alpha\bI_{D-d-d'}
\end{pmatrix}.
\end{align*}
Since $\{\bu_i\}_{i=1}^d$ is a basis of $L_*$ and is represented by the first two blocks in the expression above, whereas $\{\bw_1,\cdots,\bw_{d'}\}$ together with an orthonormal basis of $(L_*\oplus \hat{L})^\perp$ form a basis of $L_*^\perp$ (note that $d'+\dim((L_*\oplus \hat{L})^\perp)=D-d$),
$\bSigma^{(0)}_{L_*,L_*} = \mathrm{Diag}(\{\cos^2\theta_i +\alpha\}_{i=1}^d)$ (where $\theta_i=0$ for $d'+1 \leq i \leq d$),
\begin{align*}
\bSigma^{(0)}_{L_*^\perp,L_*} &= (\bSigma^{(0)}_{L_*,L_*^\perp})^{\top}
=
\begin{pmatrix}
\mathrm{Diag}(\{\cos\theta_i \sin \theta_i\}_{i=1}^{d'}) & \bm{0}_{d'\times (d-d')}\\
\bm{0}_{(D-d-d')\times d'} & \bm{0}_{(D-d-d')\times (d-d')}
\end{pmatrix},
\\
\bSigma^{(0)}_{L_*^\perp,L_*^\perp} &= \begin{pmatrix}
 \mathrm{Diag}(\{\sin^2\theta_i +\alpha\}_{i=1}^{d'}) & \bm{0}\\
\bm{0}&\alpha\bI_{D-d-d'}
\end{pmatrix}.
\end{align*}
Consequently, we can represent  $\bSigma^{(0)}_{L_*,L_*}-\bSigma^{(0)}_{L_*,L_*^\perp}\bSigma^{(0)\,-1}_{L_*^\perp,L_*^\perp}\bSigma^{(0)}_{L_*^\perp,L_*}$ in the basis  $\{\bu_i\}_{i=1}^d$  of $L_*$ (for $d'+1 \leq i \leq d$ the $i$-th diagonal entry of the subtracted term is zero, which is consistent with the formula below since $\theta_i=0$):
\begin{align*}
\bSigma^{(0)}_{L_*,L_*}-\bSigma^{(0)}_{L_*,L_*^\perp}\bSigma^{(0)\,-1}_{L_*^\perp,L_*^\perp}\bSigma^{(0)}_{L_*^\perp,L_*}=&\mathrm{Diag}\Big(\Big\{\cos^2\theta_i +\alpha - \frac{(\cos\theta_i\sin\theta_i )^2}{\sin^2\theta_i +\alpha}\Big\}_{i=1}^d\Big)\\=&\mathrm{Diag}\Big(\Big\{\frac{\alpha+\alpha^2}{\sin^2\theta_i +\alpha}\Big\}_{i=1}^d\Big),
\end{align*}
so
\begin{equation}
\label{eq:estimate_sigmad_difference}
\sigma_d\Big(\bSigma^{(0)}_{L_*,L_*}-\bSigma^{(0)}_{L_*,L_*^\perp}\bSigma^{(0)\,-1}_{L_*^\perp,L_*^\perp}\bSigma^{(0)}_{L_*^\perp,L_*}\Big)=\frac{\alpha+\alpha^2}{\sin^2\theta_1 +\alpha}.
\end{equation}
Combining it with $\sigma_1\Big(\bSigma^{(0)}_{L_*^\perp,L_*^\perp}\Big)=\sin^2\theta_1 +\alpha$, we conclude \eqref{eq:kappa1_angle}.

Since $\sigma_D(\bSigma^{(0)})=\alpha$ and the largest eigenvalue of its $L_*^\perp$ block is $\sin^2\theta_1+\alpha$, we have $\kappa_2=(\sin^2\theta_1+\alpha)/\alpha$. This formula implies \eqref{eq:example_k1_second}.

Lastly, we derive a formula for $\kappa_3$ and discuss its implications.
We use \eqref{eq:estimate_sigmad_difference} for the denominator, and note that 
$\sigma_1(\bSigma^{(0)}_{L_*,L_*}) = \cos^2\theta_d +\alpha$. Consequently, in this case $\kappa_3=(\cos^2\theta_d+\alpha)(\sin^2\theta_1+\alpha)/(\alpha+\alpha^2)\leq(1+\alpha)(\sin^2\theta_1+\alpha)/(\alpha+\alpha^2)=1+\sin^2\theta_1/\alpha$. In particular, if $\sin^2\theta_1\leq \alpha$, then  $\kappa_3\leq 2$.

\subsection{The reason for normalizing \texorpdfstring{$\calA$}{A} by the factor \texorpdfstring{$(D-d)/n_0$}{(D-d)/n0}}
\label{sec:normalize_A}
We divide by $n_0$ since we sum over all outliers. To see why we multiply by $D-d$, we first note that $\calA\geq 1$, since
\begin{equation}\label{eq:calA>1}
(D-d)\left\|\sum_{\bx\in\calX_{\mathrm{out}}}\frac{\bx\bx^\top}{\|\bU_{L_*^{\perp}}^\top\bx\|^2}\right\|\geq \tr\left(\sum_{\bx\in\calX_{\mathrm{out}}}\frac{\bU_{L_*^{\perp}}^\top\bx\bx^\top\bU_{L_*^{\perp}}}{\|\bU_{L_*^{\perp}}^\top\bx\|^2}\right)=n_0.
\end{equation}
For nonzero outliers sampled i.i.d.~from a spherically symmetric distribution on $L_*^\perp$, the strong law gives
\[
\calA\longrightarrow(D-d)\left\|\Expect\!\left[
\frac{\bx\bx^\top}{\|\bx\|^2}\right]\right\|
=(D-d)\left\|\frac{\bP_{L_*^\perp}}{D-d}\right\|=1
\qquad\text{almost surely as }n_0\to\infty.
\]
This population value explains the normalization.

\subsection{Establishing the bound \texorpdfstring{$\calR \geq 1$}{R ≥ 1}}
\label{sec:R_bound}
We derive the required bound as follows:
\begin{align*}
&\bar{\sigma}_{\mathrm{tail}}\left(\sum_{\bx\in\calX}\frac{\bx\bx^\top}{\|\bx\|^2}\right)=\frac{1}{D-d}\min_{\bV\in\R^{D\times (D-d)},\bV^\top\bV=\bI}\tr\left(\bV^\top\sum_{\bx\in\calX}\frac{\bx\bx^\top}{\|\bx\|^2}\bV\right)\leq
\\&\frac{1}{D-d}\tr\left(\bP_{L_*^\perp}\sum_{\bx\in\calX}\frac{\bx\bx^\top}{\|\bx\|^2}\bP_{L_*^\perp}\right)\leq \sigma_1\left(\bP_{L_*^\perp}\sum_{\bx\in\calX_{\mathrm{out}}}\frac{\bx\bx^\top}{\|\bx\|^2}\bP_{L_*^\perp}\right)\leq \sigma_1 \left(\sum_{\bx\in\calX_{\mathrm{out}}}\frac{\bx\bx^\top}{\|\bU_{L_*^{\perp}}^\top\bx\|^2}\right),
\end{align*}
where the first equality follows from the extremal partial trace formula~\cite[Proposition 1.3.4]{tao2012topics}, the first inequality from the definition of the minimum, and the last two inequalities from basic properties of $\sigma_1(\cdot) \equiv \|\cdot\|$ (since  $\bP_{L_*^\perp}\sum_{\bx\in\calX}{\bx\bx^\top}/{\|\bx\|^2}\bP_{L_*^\perp}$ is p.s.d.).

\subsection{Estimation of \texorpdfstring{$\calA$}{A} and \texorpdfstring{$\calR$}{R} under the asymptotic generalized haystack model}
\label{sec:proof_example_haystack}

Throughout this section, $\bx$ denotes a generic Gaussian outlier. The alignment summands are invariant to radial scaling, so we omit the covariance factor $1/D$ in these computations. This does not affect $\kappa_{\mathrm{out}}$. We thus write $\bx=(\bSigma^{(\mathrm{out})})^{1/2}\mathbf{g}$, where $\mathbf{g}\sim \mathcal{N}(\bm{0},\bI_D)$.

\textbf{Estimation of $\calA$:} By the definition of $\calA$ and the law of large numbers,
\begin{equation}\label{eq:calA1}\lim_{n_0\rightarrow\infty}{\calA} =  (D-d)\Big\|\Expect_{\bx\sim \mathcal{N}(\bm{0},\bSigma^{(\mathrm{out})})}\frac{\bx\bx^\top}{\|\bU_{L_*^{\perp}}^\top\bx\|^2}\Big\|.\end{equation}
We retain the covariance factors in the numerator, $\bx\bx^\top=(\bSigma^{(\mathrm{out})})^{1/2}\mathbf{g}\mathbf{g}^\top(\bSigma^{(\mathrm{out})})^{1/2}$, and bound the denominator. Let
\[
\mathbf H=(\bSigma^{(\mathrm{out})})^{1/2}\bP_{L_*^\perp}(\bSigma^{(\mathrm{out})})^{1/2}=\big((\bSigma^{(\mathrm{out})})^{1/2}\bU_{L_*^\perp}\big)\big((\bSigma^{(\mathrm{out})})^{1/2}\bU_{L_*^\perp}\big)^\top,
\]
so that $\|\bU_{L_*^\perp}^\top\bx\|^2=\mathbf g^\top\mathbf H\mathbf g$.
The matrix $\mathbf{H}$ has rank $D-d$ and its nonzero eigenvalues coincide with those of $\bU_{L_*^\perp}^\top\bSigma^{(\mathrm{out})}\bU_{L_*^\perp}$, which are at least $\sigma_D(\bSigma^{(\mathrm{out})})$ by Lemma~\ref{lemma:matrixperturbation}(a). Denoting by $\bP_{\mathbf{H}}$ the orthogonal projector onto $\im(\mathbf{H})$, which is a $(D-d)$-subspace, we thus have $\mathbf{H}\psdgeq \sigma_D(\bSigma^{(\mathrm{out})})\bP_{\mathbf{H}}$ and consequently $\|\bU_{L_*^\perp}^\top\bx\|^2\geq \sigma_D(\bSigma^{(\mathrm{out})})\|\bP_{\mathbf{H}}\mathbf{g}\|^2$. Using this denominator bound yields
\begin{equation}\label{eq:calA2}
\Expect_{\bx\sim \mathcal{N}(\bm{0},\bSigma^{(\mathrm{out})})}\frac{\bx\bx^\top}{\|\bU_{L_*^{\perp}}^\top\bx\|^2}
\psdleq \frac{(\bSigma^{(\mathrm{out})})^{1/2}}{\sigma_D(\bSigma^{(\mathrm{out})})}
\left(\Expect_{\mathbf{g}\sim\mathcal{N}(\bm{0},\bI)}
\frac{\mathbf{g}\mathbf{g}^\top}{\|\bP_{\mathbf H}\mathbf g\|^2}\right)
(\bSigma^{(\mathrm{out})})^{1/2}.
\end{equation}
Since $\bP_{\mathbf{H}}\mathbf{g}$ and $(\bI-\bP_{\mathbf{H}})\mathbf{g}$ are independent, $\Expect[(\bI-\bP_{\mathbf{H}})\mathbf{g}]=\bm{0}$, and $\|\bP_{\mathbf{H}}\mathbf{g}\|^2\sim\chi^2_{D-d}$, the cross terms in the inner expectation on the RHS of \eqref{eq:calA2} vanish and
\begin{equation}\label{eq:calA21}
\Expect_{\mathbf{g}\sim \mathcal{N}(\bm{0},\bI)}\frac{\mathbf{g}\mathbf{g}^\top}{\|\bP_{\mathbf{H}}\mathbf{g}\|^2}=\Expect\frac{\bP_{\mathbf{H}}\mathbf{g}\mathbf{g}^\top\bP_{\mathbf{H}}}{\|\bP_{\mathbf{H}}\mathbf{g}\|^2}+\Expect\Big[\frac{1}{\|\bP_{\mathbf{H}}\mathbf{g}\|^2}\Big]\,\Expect\big[(\bI-\bP_{\mathbf{H}})\mathbf{g}\mathbf{g}^\top(\bI-\bP_{\mathbf{H}})\big]=\frac{\bP_{\mathbf{H}}}{D-d}+\frac{\bI-\bP_{\mathbf{H}}}{D-d-2}.
\end{equation}
The first term follows from the fact that $\bP_{\mathbf{H}}\mathbf{g}/\|\bP_{\mathbf{H}}\mathbf{g}\|$ is uniformly distributed on the unit sphere of $\im(\mathbf{H})$, so that by symmetry the corresponding expectation is a scalar multiple of $\bP_{\mathbf{H}}$ with unit trace. The second term follows from the fact that ${1}/{\|\bP_{\mathbf{H}}\mathbf{g}\|^2}$ is inverse-chi-squared with $D-d$ degrees of freedom, whose expectation is $1/(D-d-2)$ when $D-d>2$, and from $\Expect[(\bI-\bP_{\mathbf{H}})\mathbf{g}\mathbf{g}^\top(\bI-\bP_{\mathbf{H}})]=\bI-\bP_{\mathbf{H}}$. Since $D-d-2<D-d$, the spectral norm of the RHS of \eqref{eq:calA21} is $1/(D-d-2)$, and spectral-norm submultiplicativity in \eqref{eq:calA2}, together with \eqref{eq:calA1}, yields
\begin{equation}\label{eq:calA3_2}
\lim_{n_0\rightarrow\infty}{\calA}\leq
\kappa_{\mathrm{out}} \cdot \frac{D-d}{D-d-2},
\end{equation}
which is the first inequality in \eqref{eq:calA_R}. We remark that if $\bSigma^{(\mathrm{out})}=\bI$, then \eqref{eq:calA2} holds with equality and $\lim_{n_0\rightarrow\infty}{\calA}=(D-d)/(D-d-2)$.

\textbf{Estimation of $\calR$:} We note that
\begin{equation}
\label{eq:calS1_prelim}
\calR=\frac{n_0\calA}{(D-d)\calS},
\end{equation}
where $\calS=\bar{\sigma}_{\mathrm{tail}}\left(\sum_{\bx\in\calX}\frac{\bx\bx^\top}{\|\bx\|^2}\right)$ is asymptotically estimated as follows:
\begin{align}\label{eq:calS1}&\lim_{n_0\rightarrow\infty}\frac{\calS}{n_0} \geq \bar{\sigma}_{\mathrm{tail}}\Big(\Expect_{\bx\sim \mathcal{N}(\bm{0},\bSigma^{(\mathrm{out})})}\frac{\bx\bx^\top}{\|\bx\|^2}\Big) \geq
\sigma_{D}\Big(\Expect_{\bx\sim \mathcal{N}(\bm{0},\bSigma^{(\mathrm{out})})}\frac{\bx\bx^\top}{\|\bx\|^2}\Big)\geq \frac{1}{D\,\kappa_{\mathrm{out}}}.\end{align}
The first inequality follows from the law of large numbers and the monotonicity of the sum of the $D-d$ smallest eigenvalues with respect to the Loewner order, noting that $\calS$ is defined for all points, whereas the RHS accounts for the outliers only. The second inequality follows from the definition of $\bar{\sigma}_{\mathrm{tail}}$. To verify the third one, we note that $\|\bx\|^2=\mathbf{g}^\top\bSigma^{(\mathrm{out})}\mathbf{g}\leq \sigma_1(\bSigma^{(\mathrm{out})})\|\mathbf{g}\|^2$, and thus
\[
\Expect_{\bx\sim \mathcal{N}(\bm{0},\bSigma^{(\mathrm{out})})}\frac{\bx\bx^\top}{\|\bx\|^2}\psdgeq \frac{(\bSigma^{(\mathrm{out})})^{1/2}\,\Expect_{\mathbf{g}\sim \mathcal{N}(\bm{0},\bI)}\big[{\mathbf{g}\mathbf{g}^\top}/{\|\mathbf{g}\|^2}\big]\,(\bSigma^{(\mathrm{out})})^{1/2}}{\sigma_1(\bSigma^{(\mathrm{out})})}=\frac{\bSigma^{(\mathrm{out})}}{D\,\sigma_1(\bSigma^{(\mathrm{out})})}\psdgeq \frac{\bI}{D\,\kappa_{\mathrm{out}}},
\]
where the equality uses the fact that by symmetry $\Expect_{\mathbf{g}\sim \mathcal{N}(\bm{0},\bI)}{\mathbf{g}\mathbf{g}^\top}/{\|\mathbf{g}\|^2}$ is a scalar matrix with unit trace and thus equals $\bI/D$.
Combining \eqref{eq:calS1_prelim} with \eqref{eq:calA3_2} and \eqref{eq:calS1}, we conclude the second inequality of \eqref{eq:calA_R}:
\[
\lim_{n_0\rightarrow\infty}{\calR}\leq \kappa_{\mathrm{out}}\cdot\frac{D-d}{D-d-2}\cdot\frac{D\,\kappa_{\mathrm{out}}}{D-d}=\kappa_{\mathrm{out}}^2\cdot \frac{D}{D-d-2}.
\]

\subsection{Derivation of the block-isotropic sufficient condition}
\label{sec:block_isotropic_details}
We derive \eqref{eq:explicit_threshold} under the assumptions of \eqref{eq:block_isotropic}. Define
\[
\mathbf{J}_0:=\bSigma^{(\mathrm{out})}_{L_*,L_*^\perp}\big(\bSigma^{(\mathrm{out})}_{L_*^\perp,L_*^\perp}\big)^{-1},
\qquad
\mathbf{J}:=\begin{pmatrix}
\bI_d & -\mathbf{J}_0\\
\bm{0} & \bI_{D-d}
\end{pmatrix}
\]
with respect to the decomposition $L_*\oplus L_*^\perp$. This map leaves $L_*$ pointwise fixed and makes the transformed outlier covariance block diagonal. In view of \eqref{eq:block_isotropic}, its two diagonal blocks are $a\bI_d$ and $b\bI_{D-d}$. The transformed population model is therefore invariant under $O(d)\times O(D-d)$ acting on $L_*\oplus L_*^\perp$, and its population TME solution has the form $\mathrm{Diag}(a_*\bI_d,b_*\bI_{D-d})$ by the argument of Lemma~\ref{lemma:special}. By the affine equivariance of TME, after transforming back, the $L_*^\perp$ block of the TME solution is $b_*\bI_{D-d}$ and the corresponding Schur complement is $a_*\bI_d$. Consequently, $\kappa_1=a_*/b_*$.

Writing $Y_1\sim\chi^2_d$ and $Y_2\sim\chi^2_{D-d}$ for independent variables and $\omega:=ba_*/(ab_*)$, the population fixed-point equation \eqref{eq:TME_definition_asymptotic} for the $L_*^\perp$ block reduces to
\begin{equation}\label{eq:kappa1_TME_explicit}
\Expect\Big[\frac{\omega Y_2}{Y_1+\omega Y_2}\Big]=1-(1-\dssnr)\frac{d}{D}.
\end{equation}
Indeed, the $L_*^\perp$ block of \eqref{eq:TME_definition_asymptotic} reads
\[
1=\frac{D}{D-d+d\dssnr}\Expect\!\left[\frac{\omega Y_2}{Y_1+\omega Y_2}\right],
\]
and $(D-d+d\dssnr)/D=1-(1-\dssnr)d/D$. The left-hand side of \eqref{eq:kappa1_TME_explicit} increases from $0$ to $1$ with $\omega$, so \eqref{eq:kappa1_TME_explicit} has a unique solution $\omega=\omega(\dssnr)>0$ for $\dssnr\in(0,1)$, which is increasing in $\dssnr$, and tends to infinity as $\dssnr\to1$. Thus $\kappa_1=(a/b)\omega(\dssnr)$. Write $q:=D-d>2$. The Cauchy--Schwarz argument following \eqref{eq:eta0_upperbound} in the proof of Lemma~\ref{lemma:special}, applied with $\kappa_{\mathrm{out}}=1$ and $\varrho=1/\omega$, gives the direct bound
\begin{equation}\label{eq:s_blockiso_lower}
\omega(\dssnr)\geq\frac{D-(1-\dssnr)(d+2)}{(1-\dssnr)(D-d)},
\qquad 0<\dssnr<1.
\end{equation}
Its numerator is positive because $D-(1-\dssnr)(d+2)>D-d-2>0$. The inequality in \eqref{eq:s_blockiso_lower} is strict: equality in the Cauchy--Schwarz step would require $Y_2+Y_1/\omega$ to be almost surely constant, whereas $Y_1$ and $Y_2$ are independent nondegenerate chi-squared variables. Therefore,
\[
\kappa_1\geq \frac{a}{b}\cdot\frac{D-(1-\dssnr)(d+2)}{(1-\dssnr)(D-d)}.
\]
The exact scalar equation also determines the leading constant as $\dssnr\uparrow1$:
\begin{equation}\label{eq:s_blockiso_asymptotic}
(1-\dssnr)\omega(\dssnr)\longrightarrow\frac{D}{D-d-2}.
\end{equation}
Indeed, its complementary equation is
$\Expect[Y_1/(Y_1+\omega Y_2)]=(1-\dssnr)d/D$. As $\omega\to\infty$,
$\omega Y_1/(Y_1+\omega Y_2)\to Y_1/Y_2$, and this integrand is bounded by the integrable random variable $Y_1/Y_2$, whose expectation is $d/(q-2)$. Dominated convergence proves \eqref{eq:s_blockiso_asymptotic}.

Moreover, \eqref{eq:kappa22} and the bound
\[
\|\mathbf{J}\|\leq 1+\left\|\bSigma^{(\mathrm{out})}_{L_*,L_*^\perp}\big(\bSigma^{(\mathrm{out})}_{L_*^\perp,L_*^\perp}\big)^{-1}\right\|\leq 1+\kappa_{\mathrm{out}}
\]
give $\kappa_2\leq(1+\kappa_{\mathrm{out}})^2$ whenever $\kappa_1\geq1$. In addition, \eqref{eq:calA_R} gives, asymptotically,
\[
\calA\leq \bar{\calA}:=\frac{\kappa_{\mathrm{out}}(D-d)}{D-d-2},
\qquad
\calR\leq \bar{\calR}:=\frac{\kappa_{\mathrm{out}}^2D}{D-d-2},
\]
and $\kappa_{\mathrm{in},*}\to1$. The right-hand side of \eqref{eq:explicit_threshold} exceeds $26$, so that condition itself implies $\kappa_1>1$ and justifies the preceding bound on $\kappa_2$. Substituting these bounds and \eqref{eq:s_blockiso_lower} into the basin condition \eqref{eq:kappa1} gives the sufficient condition \eqref{eq:explicit_threshold}. Since \eqref{eq:s_blockiso_lower} is strict, equality in \eqref{eq:explicit_threshold} still leaves a strict population basin margin; almost-sure consistency then gives the basin condition for all sufficiently large finite samples. In particular, its left-hand side diverges like $(1-\dssnr)^{-1}$ as $\dssnr\to1$, whereas for fixed $\gamma>0$ and any $c_0>0$ its right-hand side remains bounded on $[\gamma+c_0,1)$.

A less conservative sufficient condition uses the exact root in \eqref{eq:kappa1_TME_explicit} and exact population values for two of the geometric quantities. Write $j:=\|\mathbf J_0\|$ and $\beta:=qa/[b(q-2)]$. Recall that the population TME has $\kappa_1=(a/b)\omega(\dssnr)$. The exact population values are
\[
\calA_\infty:=\frac{\beta+j^2+1+\sqrt{(\beta+j^2-1)^2+4j^2}}{2},\qquad\kappa_{2,\infty}:=\frac{\kappa_1+j^2+1+\sqrt{(\kappa_1+j^2-1)^2+4j^2}}{2\kappa_1}.
\]
Thus $\calA\to\calA_\infty$ and $\kappa_2\to\kappa_{2,\infty}$ in the population limit. To verify these formulas, write a generic outlier as
$\bx=(\sqrt a\,\mathbf{g}+\sqrt b\,\mathbf J_0\mathbf{h},\sqrt b\,\mathbf{h})$ with independent standard Gaussian $\mathbf{g}\in\R^d$ and $\mathbf{h}\in\R^q$.
The identities $\Expect[1/\|\mathbf{h}\|^2]=1/(q-2)$ and $\Expect[\mathbf{h}\mathbf{h}^\top/\|\mathbf{h}\|^2]=\bI_q/q$ give
\[
\calA_\infty=\left\|\begin{pmatrix}
\beta\bI_d+\mathbf J_0\mathbf J_0^\top&\mathbf J_0\\
\mathbf J_0^\top&\bI_q
\end{pmatrix}\right\|.
\]
Similarly, the population TME matrix divided by $b_*$ is
$\left(\begin{smallmatrix}\kappa_1\bI_d+\mathbf J_0\mathbf J_0^\top&\mathbf J_0\\\mathbf J_0^\top&\bI_q\end{smallmatrix}\right)$.
A singular-value decomposition of $\mathbf J_0$ reduces both matrices to two-dimensional blocks. Their largest and smallest relevant eigenvalues occur at its largest singular value $j$; the displayed expressions follow from the quadratic formula (each TME block has determinant $\kappa_1$).
Consequently the following inequality is sufficient in the population model:
\begin{equation}\label{eq:blockiso_exact_root_certificate}
\kappa_1\geq C_{\mathrm b}\frac{\calA_\infty}{\dssnr}
\left(\frac{\calA_\infty}{\dssnr-\gamma}
      +\frac{2\kappa_{2,\infty}\bar\calR}{\gamma}\right),\qquad 0<\gamma<\dssnr<1.
\end{equation}
This substitutes exact values for $\kappa_1$, $\kappa_2$, and $\calA$ into \eqref{eq:kappa1}, while retaining the upper bound $\calR\leq\bar\calR$. A strict inequality ensures that the same sufficient condition holds almost surely for all sufficiently large finite samples.

\section{Completion of the Proof of Theorem~\ref{thm:noisy}}
\label{sec:proof_thm_noisy}
We first specify a choice for the different constants, for which our proof works:
\begin{equation}\label{eq:Cvalue_noisy_final}
C_{\mathrm b} = \frac{3340}{C_E}+\left( 76\,\frac{\dssnr+\gamma}{\gamma}+16\,\frac{(\dssnr+2\gamma)^2}{(3\gamma)^2} \right)\Big/
\left(1-\frac{2(\dssnr+2\gamma)}{3(\dssnr+\gamma)}\right),
\end{equation}
\begin{equation}
\label{eq:C_kappa_2}
C_{\kappa_2}=7,
\end{equation}
\begin{equation}C_{\kappa_3} = {26\kappa_{\mathrm{in},*}+1}
,\label{eq:c_kappa3}\end{equation}
\begin{equation}
\label{eq:C_kappa_1}
C_{\kappa_1}=\frac{1}{100 \, \epsilon \, C_{\kappa_3}}\min
\left(
\frac{1}{\epsilon \, C_{\kappa_2} } \left(\min
\left(\frac{\dssnr}{\gamma} - 1,{1,\frac{C_E}{2}}
\right)
\right)^2,\frac{n_0}{n_1(D-d)} \min\left(\frac{\dssnr}{\gamma} - 1,1
\right)\right).
\end{equation}
\textbf{A lower bound for relative alignment in the noisy case.}
The noiseless inequality $\calR\geq1$ need not hold for noisy inliers.
The following replacement suffices for the proof. For any matrix
satisfying the hypotheses of \eqref{eq:overiterations_noisy}, the
Schur complement has smallest eigenvalue at least $\sigma_D(\bSigma)$.
The lower eigenvalue bound in \eqref{eq:phi_relation} therefore gives
\[
\hat\kappa_1(\bSigma)\geq
\frac{1}{\sigma_1(\bSigma_{\mathrm{in},*})\hat\kappa_2(\bSigma)}
\geq\frac{1}{7\sigma_1(\bSigma_{\mathrm{in},*})}.
\]
The upper threshold in \eqref{eq:overiterations_noisy} implies
$C_{\kappa_1}\geq1/7$. On the other hand, \eqref{eq:C_kappa_1} gives
\[
C_{\kappa_1}\leq\frac{n_0}{100\epsilon C_{\kappa_3}n_1(D-d)},
\]
so that
\[
\frac{n_1}{n_0}\epsilon^2
\leq \frac{7\epsilon}{100C_{\kappa_3}(D-d)}\leq\frac1{\kappa_{\mathrm{in},*}},
\]
using $\epsilon\leq1$, $D-d\geq1$, and $C_{\kappa_3}=26\kappa_{\mathrm{in},*}+1$.
The extremal partial-trace formula and the cone assumption give
$\calS\leq(n_0+n_1\epsilon^2)/(D-d)$, whereas
$\tilde{\calA}=n_0\calA/(D-d)\geq n_0/(D-d)$. Consequently,
\begin{equation}\label{eq:R_lower_noisy}
\calR\geq\frac{1}{1+(n_1/n_0)\epsilon^2}
\geq\frac{\kappa_{\mathrm{in},*}}{1+\kappa_{\mathrm{in},*}},
\qquad \calR(1+\kappa_{\mathrm{in},*})\geq\kappa_{\mathrm{in},*}\geq1.
\end{equation}
In particular, the parenthesized expression in \eqref{eq:tildekappa1}
is at least $\kappa_{\mathrm{in},*}$ and at least $\calR/\gamma$. These are the lower bounds
needed when adapting the noiseless proof.

Since $C_{\kappa_2}\equiv 7$, we did not mention it in the statement of the theorem. The other constants depend on the data, or $\epsilon$ and $\gamma$, but do not depend on the algorithm output at any iteration.

\textbf{Organization of the rest of the proof:}
Appendix~\ref{sec:proof_lemma_10} provides some preliminary results.
Appendices~\ref{sec:proof_or_overiter_noisy} and~\ref{sec:proof_or_overiter_noisy_2} establish \eqref{eq:overiterations_noisy}, and Appendix~\ref{sec:growth_upper} provides the one-step upper bound on the growth of $\hat{\kappa}_1$ used in part (b) of the theorem; together with the reduction in Section~\ref{sec:reduce_noisy-thm}, this completes the proof.

\subsection{Preliminary Results}
\label{sec:proof_lemma_10}
We establish preliminary results required for proving \eqref{eq:overiterations_noisy}.

\textbf{New quantities and their bounds:}
We define the following spectral dominance factor for the noisy case:
\[\kappa_P(\bSigma,\epsilon) :={\epsilon}\sqrt{\sigma_1\left(\bSigma_{L_*,L_*}\right)/\sigma_D(\bSigma)}.\]
We prove below \eqref{eq:noisyassumption20}, which implies that $\kappa_P(\bSigma,\epsilon) < 0.1$ and thus we can use $1-\kappa_P(\bSigma,\epsilon)$ as a positive expression throughout the proof.
We also define
\[C_P(\bSigma,\epsilon):=\frac{(2\epsilon+\epsilon^2)\,n_1}{(1-\kappa_P(\bSigma,\epsilon))^2} \, \sigma_1\left(\bSigma_{L_*,L_*}\right)\,\ \text{ and } \ c_4(\bSigma,\epsilon):=\frac{C_P(\bSigma,\epsilon)}{\sigma_1\left(\bSigma_{L_*^\perp,L_*^\perp}\right)}\Big/\frac{n_0}{D-d}.\]
The next lemma bounds the new quantities.
\begin{lemma}\label{lemma:constants_noisy}
    If $\bSigma$ satisfies the hypotheses of \eqref{eq:overiterations_noisy}, then the following bounds hold:
\begin{equation}\label{eq:noisyassumption20}\kappa_P(\bSigma,\epsilon)\leq {\min(1/10,C_E/20)}
,\end{equation}
\begin{equation}\label{eq:noisy_conclusion} \max(c_4(\bSigma,\epsilon),\kappa_P(\bSigma,\epsilon))\leq \frac{1}{10}\min(\dssnr/\gamma - 1,1),\end{equation}
\begin{equation}\label{eq:c_4_kappa_p}
(1+c_4(\bSigma,\epsilon))(1+\kappa_P(\bSigma,\epsilon))^2 \leq \frac{\dssnr+2\gamma}{3\gamma},\,\,(1+c_4(\bSigma,\epsilon))^2(1+\kappa_P(\bSigma,\epsilon))^2 \leq \frac{\dssnr+\gamma}{2\gamma}.
\end{equation}
\end{lemma}

\begin{proof}[Proof of Lemma~\ref{lemma:constants_noisy}]
First, we simultaneously establish the upper bounds of $\kappa_P(\bSigma,\epsilon)$ in  \eqref{eq:noisyassumption20} and  \eqref{eq:noisy_conclusion}.
We apply the definitions of $\hat{\kappa}_1(\bSigma)$, $\hat{\kappa}_2(\bSigma)$ and $\hat{\kappa}_3(\bSigma)$ (see \eqref{eq:def_kappa_1_2_new}, second equation in \eqref{eq:def_kappa_1_2}, and \eqref{eq:def_kappa_3_new}) and \eqref{eq:phi_relation} to bound  $\kappa_P(\bSigma,\epsilon)$:
\begin{equation}
\label{eq:bound_kappa_p}
\kappa_P(\bSigma,\epsilon)=\epsilon\sqrt{\sigma_1\left(\bSigma_{L_*,L_*}\right)/\sigma_D(\bSigma)}\leq \epsilon \sqrt{\sigma_1(\bSigma_{\mathrm{in},*})\hat{\kappa}_1(\bSigma)\hat{\kappa}_2(\bSigma)\hat{\kappa}_3(\bSigma)}.
\end{equation}
We thus reduce the proof of the upper bounds on $\kappa_P(\bSigma,\epsilon)$ in \eqref{eq:noisyassumption20} and \eqref{eq:noisy_conclusion} to  verifying the bound
\begin{equation*}
\epsilon \sqrt{{\sigma_1(\bSigma_{\mathrm{in},*})}\hat{\kappa}_1(\bSigma)}\leq \frac{1}{10\sqrt{\hat{\kappa}_2(\bSigma)\hat{\kappa}_3(\bSigma)}} \min(\dssnr/\gamma - 1,{1,\frac{C_E}{2}}
),
\end{equation*}
or equivalently,
\[
{\sigma_1(\bSigma_{\mathrm{in},*})}\hat{\kappa}_1(\bSigma) \leq \frac{1}{100}\frac{1}{\epsilon^2{\hat{\kappa}_2(\bSigma)\hat{\kappa}_3(\bSigma)}} \min\left(\frac{\dssnr}{\gamma} - 1,{1,\frac{C_E}{2}}
\right)^2.
\]
We note that the above equation  follows from the assumptions $\hat{\kappa}_2(\bSigma)\leq C_{\kappa_2}$, $\hat{\kappa}_3(\bSigma)\leq C_{\kappa_3}$, and $\hat{\kappa}_1(\bSigma)\leq C_{\kappa_1}/\sigma_1(\bSigma_{\mathrm{in},*})$ (see \eqref{eq:overiterations_noisy}), and the choice of $C_{\kappa_1}$, which depends on $C_{\kappa_2}$ and $C_{\kappa_3}$ (see \eqref{eq:C_kappa_1} and note that we use the first term of the minimum). This concludes \eqref{eq:noisyassumption20} and the bound on $\kappa_P(\bSigma,\epsilon)$ in \eqref{eq:noisy_conclusion}.

Next, we verify the bound of $c_4(\bSigma,\epsilon)$ in  \eqref{eq:noisy_conclusion}.
Applying $\kappa_P(\bSigma,\epsilon)\leq 1/10$, which follows from \eqref{eq:noisyassumption20} proved above, the assumption $\epsilon\leq 1$ (so $\epsilon^2\leq \epsilon$), and the definition of $c_4(\bSigma,\epsilon)$, the upper bound of $c_4(\bSigma,\epsilon)$ in \eqref{eq:noisy_conclusion} holds if \[\frac{3\epsilon}{0.9^2}\,n_1\frac{\sigma_1\left(\bSigma_{L_*,L_*}\right)}{{\sigma_1\left(\bSigma_{L_*^\perp,L_*^\perp}\right)}}\leq \frac{1}{10}\min(\dssnr/\gamma - 1,1)\frac{n_0}{D-d}.\]
Note that   ${\sigma_1\left(\bSigma_{L_*,L_*}\right)}/{{\sigma_1\left(\bSigma_{L_*^\perp,L_*^\perp}\right)}}\leq {\sigma_1(\bSigma_{\mathrm{in},*})}\hat{\kappa}_1(\bSigma)\hat{\kappa}_3(\bSigma)$, where the proof is similar to that of \eqref{eq:bound_kappa_p}, and it is thus sufficient to verify
\begin{equation*}
\epsilon {\sigma_1(\bSigma_{\mathrm{in},*})}\hat{\kappa}_1(\bSigma)\leq \frac{n_0}{n_1(D-d)\hat{\kappa}_3(\bSigma)} \cdot 0.027 \cdot \min(\dssnr/\gamma - 1,1).
\end{equation*}

This equation follows from the assumptions $\hat{\kappa}_3(\bSigma)\leq C_{\kappa_3}$ and $\hat{\kappa}_1(\bSigma)\leq C_{\kappa_1}/\sigma_1(\bSigma_{\mathrm{in},*})$ (see \eqref{eq:overiterations_noisy}), and the choice of $C_{\kappa_1}$ (see \eqref{eq:C_kappa_1} and note that we only use the second term of the minimum).

Lastly, we prove the first equation of \eqref{eq:c_4_kappa_p} by applying \eqref{eq:noisy_conclusion} and the elementary inequality
$(1+\min(x,1)/10)^3\leq 1+x/3$ for $x>0$. For $0<x\leq1$, expansion and $x^2,x^3\leq x$ give $(1+x/10)^3\leq1+331x/1000\leq1+x/3$. For $x\geq1$, $(11/10)^3=1331/1000<4/3\leq1+x/3$.
Setting $x=\dssnr/\gamma - 1$, we observe from \eqref{eq:noisy_conclusion} that both $c_4(\bSigma,\epsilon)$ and $\kappa_P(\bSigma,\epsilon)$ are bounded above by $0.1\min(x,1)$.
Substituting these bounds into the latter inequality yields the first equation in \eqref{eq:c_4_kappa_p}. The proof of the second equation of \eqref{eq:c_4_kappa_p} is similar.
\end{proof}

\textbf{Auxiliary propositions:}
Let $T_{1,P}$ denote the TME operator, $T_1$, restricted to the set of projected inliers $\{\bU_{L_*}^\top\bx: \bx\in\calX_{\mathrm{in}}\}$, so that $T_{1,P}=\frac{n_1}{d}T_{\mathrm{in}}$ with $T_{\mathrm{in}}$ defined in Section~\ref{sec:theory_noise}. That is,
\begin{equation}\label{eq:tp_inliers}
T_{1,P}(\bSigma)=\sum_{\bx\in \calX_{\mathrm{in}}}\frac{\bU_{L_*}^\top\bx\bx^\top \bU_{L_*}}{(\bU_{L_*}^\top\bx)^\top \bSigma^{-1}\bU_{L_*}^\top\bx}.
\end{equation}

The following lemma generalizes parts (a) and (b) of Lemma~\ref{lemma:xplus} to the noisy setting.
\begin{lemma}\label{lemma:xplus_noisy}
The following properties hold, assuming that $\kappa_P(\bSigma,\epsilon)<1$:
\begin{equation}\label{eq:xin_noisyy_upperbound}
\sigma_1(\bSigma_{\mathrm{in},*}^{-0.5}[\bSigma_{+,\mathrm{in}}]_{L_*,L_*}\,\bSigma_{\mathrm{in},*}^{-0.5})\leq \frac{n_1}{d(1-\kappa_P(\bSigma,\epsilon))^2}\sigma_1(\bSigma_{\mathrm{in},*}^{-0.5}\bSigma_{L_*,L_*}\bSigma_{\mathrm{in},*}^{-0.5}),
\end{equation}
\begin{equation}\label{eq:xin_noisyy}
h(\bSigma_{\diamond,\mathrm{in}})\geq \frac{n_1}{d(1+\kappa_P(\bSigma,\epsilon))^2}h(\bSigma),
\end{equation}
\begin{equation}
\label{eq:xin_noisy2}
\|\bSigma_{+,\mathrm{in}}-\bSigma_{\diamond,\mathrm{in}}\|\leq C_P(\bSigma,\epsilon),
\end{equation}
and
\begin{align}\label{eq:xout1_noisy}
\|\bSigma_{\diamond,\mathrm{out}}\|&
\leq C_P(\bSigma,\epsilon)+\|\bSigma_{+,\mathrm{out}}\|.
\end{align}
An immediate consequence of \eqref{eq:xout1_noisy} is
\begin{align}\label{eq:xout2_noisy}
\frac{\|\bSigma_{\diamond,\mathrm{out}}\|}{\sigma_1(\bSigma_{L_*^{\perp},L_*^{\perp}})}\leq (1+c_4(\bSigma,\epsilon))\tilde{\calA} \ \text{ and } \ \frac{\tr([\bSigma_{\diamond,\mathrm{out}}]_{L_*^{\perp},L_*^{\perp}})}{\sigma_1(\bSigma_{L_*^{\perp},L_*^{\perp}})}\leq  (1+c_4(\bSigma,\epsilon)){n_0}. \end{align}
\end{lemma}
\begin{proof}[Proof of Lemma~\ref{lemma:xplus_noisy}]

(a) We claim that for any
$\bSigma \in S_{++}(D)$:
\begin{equation}\label{eq:xin_noisy1}
\frac{1}{(1+\kappa_P(\bSigma,\epsilon))^2}T_{1,P}([g_1(\bSigma)]_{L_*,L_*})\psdleq  [\bSigma_{+,\mathrm{in}}]_{L_*,L_*}\psdleq \frac{1}{(1-\kappa_P(\bSigma,\epsilon))^2}T_{1,P}(\bSigma_{L_*,L_*})
\end{equation}
and
\begin{align}
\nonumber
&\frac{n_1}{d}{\left(1+C_E\left(1-\sigma_d(\bSigma_{\mathrm{in},*}^{-0.5}\bSigma\,\bSigma_{\mathrm{in},*}^{-0.5})/\sigma_1(\bSigma_{\mathrm{in},*}^{-0.5}\bSigma\,\bSigma_{\mathrm{in},*}^{-0.5})\right)\right)}\sigma_d(\bSigma_{\mathrm{in},*}^{-0.5}\bSigma\,\bSigma_{\mathrm{in},*}^{-0.5}) \\
&\leq \sigma_d(\bSigma_{\mathrm{in},*}^{-0.5}T_{1,P}(\bSigma)\,\bSigma_{\mathrm{in},*}^{-0.5})
\leq \sigma_1(\bSigma_{\mathrm{in},*}^{-0.5}T_{1,P}(\bSigma)\,\bSigma_{\mathrm{in},*}^{-0.5})\leq \frac{n_1}{d}\sigma_1(\bSigma_{\mathrm{in},*}^{-0.5}\bSigma\,\bSigma_{\mathrm{in},*}^{-0.5}).\label{eq:xin_noisy3}
\end{align}
Before proving these claims, we show they are sufficient for concluding this part. Indeed,  \eqref{eq:xin_noisyy_upperbound} follows from the second Loewner order  in \eqref{eq:xin_noisy1} and the third inequality in \eqref{eq:xin_noisy3}, while \eqref{eq:xin_noisyy} follows from the first Loewner order  in \eqref{eq:xin_noisy1}, the first inequality in \eqref{eq:xin_noisy3}, and  $C_E>0$.

We first conclude \eqref{eq:xin_noisy3} by  noting that its first inequality follows from the proofs of   \eqref{eq:lemma2sigmad}, and its third inequality follows from the proof of Lemma 2.1 of  \cite{tyler1987distribution}.

To  prove \eqref{eq:xin_noisy1}, we first note that
\begin{equation}
\label{eq:Simga+L*}
[\bSigma_{+,\mathrm{in}}]_{L_*,L_*}=\sum_{\bx\in\calX_{\mathrm{in}}}\frac{\bU_{L_*}^\top\bx\bx^\top \bU_{L_*}}{\bx^\top\bSigma^{-1}\bx}.
\end{equation}
We thus need to compare the RHS of \eqref{eq:tp_inliers}, applied to either  $[g_1(\bSigma)]_{L_*,L_*}$ or $\bSigma_{L_*,L_*}$, with that of \eqref{eq:Simga+L*}. Note that they have the same numerators, so it is sufficient to compare their denominators,   $(\bU_{L_*}^\top\bx)^\top\bSigma^{-1}(\bU_{L_*}^\top\bx)$, applied to the same matrices, and $\bx^\top\bSigma^{-1}\bx$,  where $\bx\in\calX_{\mathrm{in}}$.

For any $\bx\in\calX_{\mathrm{in}}$, using basic matrix properties and inequalities (in particular, $g_1(\bSigma) \psdleq \bSigma$, $\bA_{L_*,L_*} \psdleq \bB_{L_*,L_*}$ whenever $\bA \psdleq \bB$, and $\sigma_{1}([\bSigma^{-1}]_{L_*^\perp,L_*^\perp})\leq \sigma_{1}([\bSigma^{-1}])=1/\sigma_D(\bSigma)$),
\begin{align}
&\sqrt{\bx^\top\bSigma^{-1}\bx}=\|\bSigma^{-0.5}\bx\|\leq \|\bSigma^{-0.5}\bP_{L_*}\bx\|+\|\bSigma^{-0.5}\bP_{L_*^\perp}\bx\|
= \sqrt{\bx^\top\bU_{L_*}[\bSigma^{-1}]_{L_*,L_*}\bU_{L_*}^\top\bx}\nonumber\\
&
+\sqrt{\bx^\top\bU_{L_*^\perp}[\bSigma^{-1}]_{L_*^\perp,L_*^\perp}\bU_{L_*^\perp}^\top\bx}
\leq \sqrt{\bx^\top\bU_{L_*}([g_1(\bSigma)]_{L_*,L_*})^{-1}\bU_{L_*}^\top\bx}+\frac{\|\bU_{L_*^\perp}^\top\bx\|}{\sqrt{\sigma_D(\bSigma)}}.\label{eq:lemma11a_1}
\end{align}
Similarly,
\begin{align}
\label{eq:lemma11a_2}
&\sqrt{\bx^\top\bSigma^{-1}\bx}
\geq \|\bSigma^{-0.5}\bP_{L_*}^\top\bx\|-\|\bSigma^{-0.5}\bP_{L_*^\perp}^\top\bx\|\geq\sqrt{\bx^\top\bU_{L_*}
[\bSigma^{-1}]_{L_*,L_*}\bU_{L_*}^\top\bx}
\\ &-\frac{\|\bU_{L_*^\perp}^\top\bx\|}{\sqrt{\sigma_D(\bSigma)}}\nonumber\geq\sqrt{\bx^\top\bU_{L_*}
(\bSigma_{L_*,L_*})^{-1}\bU_{L_*}^\top\bx}
-\frac{\|\bU_{L_*^\perp}^\top\bx\|}{\sqrt{\sigma_D(\bSigma)}},\end{align}
where the last inequality uses  $[\bSigma^{-1}]_{L_*,L_*}\psdgeq (\bSigma_{L_*,L_*})^{-1}$, which we clarify below. Applying the formula for the inversion of a $2\times 2$ block matrix~\cite[bottom of page~86]{bhatia2009positive}:
\[
\bSigma^{-1}=\begin{pmatrix}
\bSigma_{L_*,L_*} & \bSigma_{L_*,L_*^\perp} \\
\bSigma_{L_*^\perp,L_*} & \bSigma_{L_*^\perp,L_*^\perp}
\end{pmatrix}^{-1}=\begin{pmatrix}
(\bSigma_{L_*,L_*}-\bSigma_{L_*,L_*^\perp}\bSigma_{L_*^\perp,L_*^\perp}^{-1}\bSigma_{L_*^\perp,L_*})^{-1} & * \\
* & *
\end{pmatrix}.
\]
This implies  $[\bSigma^{-1}]_{L_*,L_*}=([g_1(\bSigma)]_{L_*,L_*})^{-1}$. Applying $g_1(\bSigma)  \psdleq \bSigma$ (see Lemma~\ref{lemma:g1}(a)), we conclude that $[g_1(\bSigma)]_{L_*,L_*} \psdleq \bSigma_{L_*,L_*}$ and consequently the inequality
$[\bSigma^{-1}]_{L_*,L_*}\psdgeq (\bSigma_{L_*,L_*})^{-1}$.

Applying again $\bSigma_{L_*,L_*}\psdgeq [g_1(\bSigma)]_{L_*,L_*}$ yields
\[
\sqrt{\bx^\top\bU_{L_*}([g_1(\bSigma)]_{L_*,L_*})^{-1}\bU_{L_*}^\top\bx}\geq \sqrt{\bx^\top\bU_{L_*}\left(\bSigma_{L_*,L_*}\right)^{-1}\bU_{L_*}^\top\bx} \geq \frac{\|\bU_{L_*}^\top\bx\|}{\sqrt{\sigma_1\left(\bSigma_{L_*,L_*}\right)}}.
\]
Applying the latter equation, the model assumption $\|\bP_{L_*^\perp}\bx\|\leq \epsilon \|\bP_{L_*}\bx\|$ and the definition of $\kappa_P(\bSigma,\epsilon)$, results in the following bound, which is relevant to both \eqref{eq:lemma11a_1} and \eqref{eq:lemma11a_2}:
\begin{equation}
\frac{{\|\bU_{L_*^\perp}^\top\bx\|}/{\sqrt{\sigma_D(\bSigma)}}}{\sqrt{\bx^\top\bU_{L_*}([g_1(\bSigma)]_{L_*,L_*})^{-1}\bU_{L_*}^\top\bx}}\leq
\frac{{\|\bU_{L_*^\perp}^\top\bx\|}/{\sqrt{\sigma_D(\bSigma)}}}{\sqrt{\bx^\top\bU_{L_*}\left(\bSigma_{L_*,L_*}\right)^{-1}\bU_{L_*}^\top\bx}}\leq \kappa_P(\bSigma,\epsilon).
\label{eq:bound_last_lemma_10a}
\end{equation}

Combining \eqref{eq:bound_last_lemma_10a} with \eqref{eq:lemma11a_1} and \eqref{eq:lemma11a_2} results in
\begin{equation}\label{eq:weight_ratio}
\frac{(1+\kappa_P(\bSigma,\epsilon))^{-2}}{\bx^\top\bU_{L_*}([g_1(\bSigma)]_{L_*,L_*})^{-1}\bU_{L_*}^\top\bx}\leq \frac{1}{\bx^\top\bSigma^{-1}\bx}\leq
\frac{{(1-\kappa_P(\bSigma,\epsilon))^{-2}}}{\bx^\top\bU_{L_*}(\bSigma_{L_*,L_*})^{-1}\bU_{L_*}^\top\bx}.
\end{equation}
In view of the proof strategy explained above (that is, comparing the denominators of \eqref{eq:Simga+L*} and \eqref{eq:tp_inliers} applied to either  $[g_1(\bSigma)]_{L_*,L_*}$ or $\bSigma_{L_*,L_*}$), \eqref{eq:weight_ratio} concludes the proof of \eqref{eq:xin_noisy1}.

(b)
Clearly, the $(L_*,L_*)$ block of $\bSigma_{\diamond,\mathrm{in}}$ coincides with that of $\bSigma_{+,\mathrm{in}}$ and the other blocks are zero. Consequently,
\begin{equation}
\|\bSigma_{+,\mathrm{in}}-\bSigma_{\diamond,\mathrm{in}}\|\leq 2\|[\bSigma_{+,\mathrm{in}}]_{L_*,L_*^\perp}\|+\|[\bSigma_{+,\mathrm{in}}]_{L_*^\perp,L_*^\perp}\|.\label{eq:assist_10_b_1}
\end{equation}

To bound both terms on the RHS of the above equation, we use the following inequality obtained by first applying the second inequality of \eqref{eq:weight_ratio} and then the Loewner order $(\bSigma_{L_*,L_*})^{-1} \psdgeq \sigma_1(\bSigma_{L_*,L_*})^{-1}\bI_{d}$:
\[\frac{1}{\bx^\top\bSigma^{-1}\bx}\leq \frac{1}{(1-\kappa_P(\bSigma,\epsilon))^2}\frac{1}{\bx^\top \bU_{L_*}(\bSigma_{L_*,L_*})^{-1}\bU_{L_*}^\top \bx}\leq \frac{1}{(1-\kappa_P(\bSigma,\epsilon))^2}\frac{\sigma_1\left(\bSigma_{L_*,L_*}\right)}{ \|\bU_{L_*}^\top\bx\|^2}.\]
Applying the definition of $\bSigma_{+,\mathrm{in}}$, then the noise assumption $\|\bU^\top_{L_*^\perp}\bx\|\leq \epsilon \|\bU^\top_{L_*}\bx\|$, and at last the above inequality while noting that the resulting bound is the same for any of the $n_1$ inliers, yields
\begin{align*}\|[\bSigma_{+,\mathrm{in}}]_{L_*,L_*^\perp}\|=&\left\|\sum_{\bx\in\calX_{\mathrm{in}}}\frac{\bU_{L_*}^\top\bx\bx^\top \bU_{L_*^\perp}}{\bx^\top\bSigma^{-1}\bx}\right\|\leq \sum_{\bx\in\calX_{\mathrm{in}}}\frac{\epsilon\|\bU_{L_*}^\top\bx\|^2}{\bx^\top\bSigma^{-1}\bx}\leq 
n_1\frac{\epsilon\sigma_1\left(\bSigma_{L_*,L_*}\right)}{(1-\kappa_P(\bSigma,\epsilon))^2},\\ \|[\bSigma_{+,\mathrm{in}}]_{L_*^\perp,L_*^\perp}\|=&\left\|\sum_{\bx\in\calX_{\mathrm{in}}}\frac{\bU_{L_*^\perp}^\top\bx\bx^\top \bU_{L_*^\perp}}{\bx^\top\bSigma^{-1}\bx}\right\|\leq \sum_{\bx\in\calX_{\mathrm{in}}}\frac{\epsilon^2\|\bU_{L_*}^\top\bx\|^2}{\bx^\top\bSigma^{-1}\bx}\leq 
n_1\frac{\epsilon^2\sigma_1\left(\bSigma_{L_*,L_*}\right)}{(1-\kappa_P(\bSigma,\epsilon))^2}.
\end{align*}
Lastly, combining \eqref{eq:assist_10_b_1} with the above two equations, we conclude  \eqref{eq:xin_noisy2}:
\begin{align*}\|\bSigma_{+,\mathrm{in}}-\bSigma_{\diamond,\mathrm{in}}\|\leq n_1\frac{2\epsilon+\epsilon^2}{(1-\kappa_P(\bSigma,\epsilon))^2}\sigma_1\left(\bSigma_{L_*,L_*}\right)=C_P(\bSigma,\epsilon).
\end{align*}

Next, we note that the definitions of $\bSigma_{\diamond,\mathrm{out}}$ and $\bSigma_{+,\mathrm{out}}$ imply  $\bSigma_{\diamond,\mathrm{out}}-\bSigma_{+,\mathrm{out}}=-(\bSigma_{\diamond,\mathrm{in}}-\bSigma_{+,\mathrm{in}})$.
 Consequently, $\|\bSigma_{\diamond,\mathrm{out}}\|\leq \|\bSigma_{+,\mathrm{out}}\| +\|\bSigma_{+,\mathrm{in}}-\bSigma_{\diamond,\mathrm{in}}\|$ and  \eqref{eq:xout1_noisy} follows from the latter inequality and \eqref{eq:xin_noisy2}. The proof of \eqref{eq:xout2_noisy} follows from \eqref{eq:xout2}, \eqref{eq:xout1_noisy}, the definition of $c_4(\bSigma,\epsilon)$, and the fact that $\calA\geq 1$ (see  \eqref{eq:calA>1}).
\end{proof}

The following lemma establishes that $\sigma_1\left([g_1(\bSigma)]_{L_*,L_*}\right)$ remains comparable in magnitude to $\sigma_1\left(\bSigma_{L_*,L_*}\right)$. This result will be used in Part II of the proof.
\begin{lemma}\label{lemma:compare_g1}
If $\bSigma\in S_{++}(D)$, then
\[
[g_1(\bSigma)]_{L_*,L_*}
\psdgeq \frac{\bSigma_{L_*,L_*}}{\hat\kappa_2(\bSigma)}.
\]
In particular, if $\hat\kappa_2(\bSigma)\leq7$, then
$\sigma_1([g_1(\bSigma)]_{L_*,L_*})\geq\sigma_1(\bSigma_{L_*,L_*})/7$.
\end{lemma}
\begin{proof}[Proof of Lemma~\ref{lemma:compare_g1}]
Write $A:=\bSigma_{L_*,L_*}$, $B:=\bSigma_{L_*,L_*^\perp}$,
$C:=\bSigma_{L_*^\perp,L_*^\perp}$, $\mu:=\sigma_D(\bSigma)$,
and $c:=\sigma_1(C)$. For $0<s<\mu$, the Schur complement of
$C-sI$ in $\bSigma-sI\psdgeq0$ gives
\[
A-sI\psdgeq B(C-sI)^{-1}B^\top.
\]
Since every eigenvalue of $C$ is at most $c$,
$C^{-1}\psdleq(1-s/c)(C-sI)^{-1}$.
Consequently,
\[
A-BC^{-1}B^\top
\psdgeq A-(1-s/c)(A-sI)
=\frac{s}{c}A+s(1-s/c)I.
\]
Letting $s\uparrow\mu$ and using
$\hat\kappa_2(\bSigma)=c/\mu$ proves the stated Loewner bound.
\end{proof}

\subsection{Proof of \eqref{eq:overiterations_noisy}-Part I}\label{sec:proof_or_overiter_noisy}
We derive the bounds of  $\hat{\kappa}_1(T(\bSigma))$ and  $\hat{\kappa}_2(T(\bSigma))$, given the assumptions on $\hat{\kappa}_1(\bSigma)$ and  $\hat{\kappa}_2(\bSigma)$ stated in \eqref{eq:overiterations_noisy}.

\textbf{Verifying $\boldsymbol{{\hat{\kappa}_1(T(\bSigma))\geq C_0 \hat{\kappa}_1(\bSigma)}}$}:
We follow the proof of  Theorem~\ref{thm:main}, in particular, \eqref{eq:kappa1_proof}-\eqref{eq:conditions_C}. However, we replace $(\bSigma_{+,\mathrm{in}}, \bSigma_{+,\mathrm{out}})$ with $(\bSigma_{\diamond,\mathrm{in}}, \bSigma_{\diamond,\mathrm{out}})$ and redefine   $\nu:=\sigma_d(\bSigma_{\diamond,\mathrm{in}})/\|\bSigma_{\diamond,\mathrm{out}}\|$. There are two technical modifications. First, since \eqref{eq:xout2_noisy} serves as a replacement for \eqref{eq:xout2},
all estimates of $\sigma_{1}(\bSigma_{L_*^{\perp},L_*^{\perp}})$ must include an
additional factor of $1 + c_{4}(\bSigma,\epsilon)$. Second, because we use \eqref{eq:xin_noisyy} instead of \eqref{eq:xin},
all estimates of $h\big(\bSigma_{\diamond,\mathrm{in}}\big)$ must include the additional factor
$\big(1 + \kappa_{P}(\bSigma,\epsilon)\big)^{2}$.

Applying these required modifications,  \eqref{eq:est_nu} slightly changes to include an additional factor:
\begin{align}\label{eq:est_nu_noisy}
\nu \geq  \frac{1}{(1+\kappa_P(\bSigma,\epsilon))^2(1+c_4(\bSigma,\epsilon))}\hat{\kappa}_1(\bSigma)\frac{n_1\sigma_d(\bSigma_{\mathrm{in},*})}{\tilde{\calA}d}.
\end{align}
Similarly, the conclusion $\nu\geq C_{\mathrm b}$ that followed from \eqref{eq:est_nv_asumption} is replaced with
\[
\nu\geq C':=\frac{C_{\mathrm b}}{(1+\kappa_P(\bSigma,\epsilon))^2(1+c_4(\bSigma,\epsilon))},
\]
where $C'\geq 76\geq 58$, since $C_{\mathrm b}\geq 76(\dssnr+\gamma)/\gamma\geq 76(\dssnr+2\gamma)/(3\gamma)$ by \eqref{eq:Cvalue_noisy_final} and $(1+\kappa_P(\bSigma,\epsilon))^2(1+c_4(\bSigma,\epsilon))\leq (\dssnr+2\gamma)/(3\gamma)$ by \eqref{eq:c_4_kappa_p}.
In view of~\eqref{eq:est_nu_noisy}, the bound on $\nu$ in~\eqref{eq:kappa1nu} includes the same  additional factor:
\begin{equation}\label{eq:kappa1nu_noisy}
\nu\geq \frac{C_{\mathrm b}}{(1+c_4(\bSigma,\epsilon))(1+\kappa_P(\bSigma,\epsilon))^2}\left(\frac{\tilde{\calA}}{\frac{n_1}{d}-\gamma\frac{n_0}{D-d}}+\frac{\kappa_2\tilde{\calA}}{\gamma \calS}(1+\kappa_{\mathrm{in},*})\right).
\end{equation}
Furthermore, the estimate of the correction term in the numerator of $\hat{\kappa}_1(T(\bSigma))$ that leads to \eqref{eq:bound_nuemrator_final} uses $\|\bSigma_{+,\mathrm{out}}\|/\sigma_D(\bSigma)\leq\tilde{\calA}\hat{\kappa}_2(\bSigma)$, whose noisy counterpart carries the factor $1+c_4(\bSigma,\epsilon)$ by \eqref{eq:xout2_noisy}. Consequently, the final inequality in \eqref{eq:est_kappa1} is modified as follows:
\begin{align*}
\hat{\kappa}_1(T(\bSigma))\geq \frac{1}{(1+c_4(\bSigma,\epsilon))(1+\kappa_P(\bSigma,\epsilon))^2}\frac{1-\Big(2+\frac{4(1+c_4(\bSigma,\epsilon))\tilde{\calA}\hat{\kappa}_2(\bSigma)}{\gamma\calS}\Big) \frac{\kappa_{\mathrm{in},*}}{\nu}}{\frac{4}{\nu-2} \, \tilde{\calA} +\gamma \frac{n_0}{D-d}}\frac{n_1}{d}\hat{\kappa}_1(\bSigma).
\end{align*}
Here the noisy induction retains $\hat\kappa_2(\bSigma)\leq7\leq7\kappa_2$, rather than the bound $\max(\kappa_2,3)$ used in the noiseless proof; its corresponding correction term is therefore $28/C'$.
We define   $c_4'(\bSigma,\epsilon)=(1+c_4(\bSigma,\epsilon))(1+\kappa_P(\bSigma,\epsilon))^2$ and note that the first two estimates in \eqref{eq:aux_eq_to_figure_const} hold with $C_{\mathrm b}$ replaced by $C'=C_{\mathrm b}/c_4'(\bSigma,\epsilon)$, whereas the last one becomes $\frac{4(1+c_4(\bSigma,\epsilon))\tilde{\calA}\hat{\kappa}_2(\bSigma)}{\gamma\calS}\cdot\frac{\kappa_{\mathrm{in},*}}{\nu}\leq \frac{28(1+c_4(\bSigma,\epsilon))}{C'}\leq \frac{30.8}{C'}$, where we used $c_4(\bSigma,\epsilon)\leq 0.1$ (see \eqref{eq:noisy_conclusion}). Observing the last inequality of \eqref{eq:aux_eq_to_figure_const}, one may note that for the current theorem  the main condition in \eqref{eq:kappa1} can omit the dependence on $\kappa_2$, but for consistency with the previous proof we do not change \eqref{eq:kappa1}.
In view of these replacements,
condition~\eqref{eq:conditions_C_aux} in the proof of Theorem~\ref{thm:main} is replaced with
\begin{equation*}
    \frac{1-\frac{32.8}{C_{\mathrm b}}\,c_4'(\bSigma,\epsilon)}{\frac{8}{C_{\mathrm b}}\,c_4'(\bSigma,\epsilon)\left(\frac{n_1}{d}-\gamma\frac{n_0}{D-d}\right)+\gamma\frac{n_0}{D-d}}\cdot\frac{1}{c_4'(\bSigma,\epsilon)}\cdot\frac{n_1}{d}
\geq C_0.
\end{equation*}
Applying the arguments used to derive \eqref{eq:conditions_C} from \eqref{eq:conditions_C_aux} gives:
\begin{equation}
\frac{1-\frac{32.8}{C_{\mathrm b}}c_4'(\bSigma,\epsilon)}{\frac{8}{C_{\mathrm b}}c_4'(\bSigma,\epsilon)\Big(\dssnr-\gamma\Big)+\gamma}\cdot\frac{\dssnr}{c_4'(\bSigma,\epsilon)}
\geq \frac{2\dssnr}{\dssnr+\gamma}.\label{eq:before_C_noisy}
\end{equation}
Note that \eqref{eq:c_4_kappa_p} implies $c_4'(\bSigma,\epsilon)\leq ({\dssnr+2\gamma})/({3\gamma})$. Combining this latter bound with \eqref{eq:before_C_noisy} and basic algebraic manipulation, results in the following requirement on $C_{\mathrm b}$:
\begin{equation*}
C_{\mathrm b}
\geq \frac{32.8\,\frac{\dssnr+2\gamma}{3\gamma}+{2\frac{\dssnr+2\gamma}{3\gamma}}\cdot 8\frac{\dssnr+2\gamma}{3\gamma}\cdot\frac{\dssnr-\gamma}{{\dssnr+\gamma}}}{1-\frac{2(\dssnr+2\gamma)}{3(\dssnr+\gamma)}}.
\end{equation*}
This requirement holds for $C_{\mathrm b}$ set in \eqref{eq:Cvalue_noisy_final}, since $76(\dssnr+\gamma)/\gamma\geq 32.8(\dssnr+2\gamma)/(3\gamma)$ and $(\dssnr-\gamma)/(\dssnr+\gamma)\leq 1$, and we thus conclude that  $\hat{\kappa}_1(T(\bSigma))\geq C_0 \hat{\kappa}_1(\bSigma)$.

\textbf{Verifying $\boldsymbol{\hat{\kappa}_2(T(\bSigma))\leq 7}$:}
We follow the proof of Theorem~\ref{thm:main}, specifically \eqref{eq:bound_of_T_for_lemma6}-\eqref{eq:est_kappa2}, applying the above replacements. We recall that $C_{\mathrm b}$ is replaced with $C'$ specified above.  Combining  \eqref{eq:Cvalue_noisy_final} and  \eqref{eq:c_4_kappa_p} yields   $C_{\mathrm b} \geq 70(1+c_4(\bSigma,\epsilon))^2(1+\kappa_P(\bSigma,\epsilon))^2$. This results in two inequalities we will use in the proof: $C' \geq 70$ and  $C'\geq 58(1+c_4(\bSigma,\epsilon))\geq 2+56(1+c_4(\bSigma,\epsilon))$.
Following the proof  of \eqref{eq:kappa2nu}, but applying the latter inequality when verifying the last inequality of \eqref{eq:kappa2nu} in the current case, yields the following adaptation for the noisy case:
\begin{equation}\label{eq:kappa2nu_noisy}
(\nu-2)\gamma\calS \geq 16(1+c_4(\bSigma,\epsilon))\tilde{\calA}\hat{\kappa}_2(\bSigma).
\end{equation}

We claim that  \eqref{eq:kappa2nu2} still holds for the noisy case. Indeed, \eqref{eq:xout2} is replaced with \eqref{eq:xout2_noisy}, which has an additional factor of $(1+c_4(\bSigma,\epsilon))$, but this factor cancels with the same factor in \eqref{eq:kappa2nu_noisy}. One can note that no additional change is needed until \eqref{eq:kappa2_bound} and \eqref{eq:check_a_b_ratio}, where $C_{\mathrm b}$ needs to be replaced by $C'$ and, since the upper bound of $\|\bSigma_{\diamond,\mathrm{out}}\|$ in \eqref{eq:xout2_noisy} carries the factor $1+c_4(\bSigma,\epsilon)$, the bound on $a/b$ becomes $28(1+c_4(\bSigma,\epsilon))/C'$.
These replacements  result in the following revised version of  \eqref{eq:est_kappa2}:
\begin{align}\label{eq:kappa2_bound2_noisy}
\hat{\kappa}_2(T(\bSigma))
\leq 2\max\left(1+\frac{28(1+c_4(\bSigma,\epsilon))}{C'},\frac{2}{C'-2}\right).
\end{align}
Combining $C' \geq 70$ and $c_4(\bSigma,\epsilon)\leq 0.1$ (see \eqref{eq:noisy_conclusion}) with \eqref{eq:kappa2_bound2_noisy} yields the desired bound:
$\hat{\kappa}_2(T(\bSigma))\leq 2\cdot(1+30.8/70)<3\leq 7$.

\subsection{Proof of \eqref{eq:overiterations_noisy}-Part II}\label{sec:proof_or_overiter_noisy_2}
We prove   $\hat{\kappa}_3(T(\bSigma))\leq C_{\kappa_3}$ if $\hat{\kappa}_3(\bSigma)\leq C_{\kappa_3}$. To do this, we will show that both $[g_1(T(\bSigma))]_{L_*,L_*}$ and $[T(\bSigma)]_{L_*,L_*}$, which appear in the  definition of $\hat{\kappa}_3(T(\bSigma))$ (see \eqref{eq:def_kappa_3_new}) can be approximated by  $[\bSigma_{\diamond,\mathrm{in}}]_{L_*,L_*}$, which in turn can be approximated by $T_{1,P}([g_1(\bSigma)]_{L_*,L_*})$ and $T_{1,P}(\bSigma_{L_*,L_*})$ respectively. Combining it with  Lemma~\ref{lemma:TME_contraction}, we bound $\hat{\kappa}_3(T(\bSigma))$ by $C_{\kappa_3}$. The proof uses Lemma~\ref{lemma:bound_sigma_out}, where we recall that in the current setting it holds when $\bSigma_{+,\mathrm{in}}$ and $\bSigma_{+,\mathrm{out}}$ are replaced with their noisy counterparts $\bSigma_{\diamond,\mathrm{in}}$ and $\bSigma_{\diamond,\mathrm{out}}$. We partition the proof into the following four steps.

\textbf{First step: analysis of the numerator and denominator of $\boldsymbol{\hat{\kappa}_3(T(\bSigma))}$.}
For the numerator of $\hat{\kappa}_3(T(\bSigma))$, applying $\sigma_1(\bX)\leq \sigma_1(\bY)+\|\bX-\bY\|$, $\phi(\bX)-\phi(\bY)=\phi(\bX-\bY)$ (recall that $\phi$ was defined in \eqref{eq:def_phi}) and $\|\phi(\bX-\bY)\|\leq \|\bX-\bY\|/\sigma_d(\bSigma_{\mathrm{in},*})$, which follows from submultiplicativity of the spectral norm, then applying $[\bSigma_{\diamond,\mathrm{in}}]_{L_*,L_*}=[\bSigma_{+,\mathrm{in}}]_{L_*,L_*}$, and at last, the second Loewner order in  \eqref{eq:xin_noisy1}, the definition of $\phi$, and the last inequality in \eqref{eq:xin_noisy3}, results in
\begin{align}\nonumber
\sigma_1(\phi([T(\bSigma)]_{L_*,L_*})) \leq & \sigma_1(\phi([\bSigma_{\diamond,\mathrm{in}}]_{L_*,L_*})) +\frac{\|[T(\bSigma)]_{L_*,L_*}-[\bSigma_{\diamond,\mathrm{in}}]_{L_*,L_*}\|}{\sigma_d(\bSigma_{\mathrm{in},*})}\\= & \sigma_1(\phi([\bSigma_{+,\mathrm{in}}]_{L_*,L_*})) +\frac{\|[T(\bSigma)]_{L_*,L_*}-[\bSigma_{\diamond,\mathrm{in}}]_{L_*,L_*}\|}{\sigma_d(\bSigma_{\mathrm{in},*})}\label{eq:bound_kappa_3T_new_2} \\\nonumber
 \leq& \frac{n_1}{d}\frac{1}{(1-\kappa_P(\bSigma,\epsilon))^2}\sigma_1(\phi(\bSigma_{L_*,L_*}))+\frac{\|[T(\bSigma)]_{L_*,L_*}-[\bSigma_{\diamond,\mathrm{in}}]_{L_*,L_*}\|}{\sigma_d(\bSigma_{\mathrm{in},*})}.
\end{align}

Following the same proof as above, but using instead the inequality $\sigma_d(\bX)\geq \sigma_d(\bY)-\|\bX-\bY\|$,
the first Loewner order in  \eqref{eq:xin_noisy1}, and the first inequality in \eqref{eq:xin_noisy3}, yields
\begin{multline}
\label{eq:bound_kappa_3T_new_3}
 \sigma_d(\phi([g_1(T(\bSigma))]_{L_*,L_*})) \geq \sigma_d(\phi([\bSigma_{\diamond,\mathrm{in}}]_{L_*,L_*})) -\frac{\|[g_1(T(\bSigma))]_{L_*,L_*}-[\bSigma_{\diamond,\mathrm{in}}]_{L_*,L_*}\|}{\sigma_d(\bSigma_{\mathrm{in},*})} \geq \\
 \frac{n_1}{d}\frac{{\left(1+C_E\left(1-\frac{\sigma_d(\phi([g_1(\bSigma)]_{L_*,L_*}))}{\sigma_1(\phi([g_1(\bSigma)]_{L_*,L_*}))}\right)\right)}}{(1+\kappa_P(\bSigma,\epsilon))^2}\sigma_d(\phi([g_1(\bSigma)]_{L_*,L_*}))\!-\!\frac{\|[g_1(T(\bSigma))-\bSigma_{\diamond,\mathrm{in}}]_{L_*,L_*}\|}{\sigma_d(\bSigma_{\mathrm{in},*})}.
\end{multline}

\textbf{Second step: a lower bound on a term in the RHS of \eqref{eq:bound_kappa_3T_new_3}}.
We prove
\begin{equation}\label{eq:kappa3_bar}
\frac{\sigma_1(\phi([g_1(\bSigma)]_{L_*,L_*}))}{\sigma_d(\phi([g_1(\bSigma)]_{L_*,L_*}))}\geq  \bar{\kappa}_3(\bSigma):=\max\left(\frac{\hat{\kappa}_3(\bSigma)}{8\kappa_{\mathrm{in},*}},1\right).
\end{equation}
Lemma~\ref{lemma:compare_g1} and $\hat\kappa_2(\bSigma)\leq7$ give
\[
\sigma_1(\bSigma_{L_*,L_*})\leq8\sigma_1([g_1(\bSigma)]_{L_*,L_*}).
\]
Applying the eigenvalue bounds in \eqref{eq:phi_relation} and the definition
of
$\kappa_{\mathrm{in},*}$ (see \eqref{eq:defkinstar}) to the latter  inequality yields
\[
\sigma_1(\phi(\bSigma_{L_*,L_*}))
\leq8\kappa_{\mathrm{in},*}\sigma_1(\phi([g_1(\bSigma)]_{L_*,L_*})).
\]
Thus,
\begin{equation*}
\frac{\sigma_1(\phi([g_1(\bSigma)]_{L_*,L_*}))}{\sigma_d(\phi([g_1(\bSigma)]_{L_*,L_*}))}=\hat{\kappa}_3(\bSigma) \frac{\sigma_1(\phi([g_1(\bSigma)]_{L_*,L_*}))}{\sigma_1(\phi(\bSigma_{L_*,L_*}))}\geq  \frac{\hat{\kappa}_3(\bSigma)}{8\kappa_{\mathrm{in},*}}.\end{equation*}
This inequality and the fact that the eigenvalues, $\sigma_i(\cdot)$, are in decreasing order result in \eqref{eq:kappa3_bar}.

\textbf{Third step: Refined bounds for terms used in the first step.} We bound the quantity
\begin{equation}
    \label{eq:def_tau}
\tau:=\frac{\|[g_1(T(\bSigma))]_{L_*,L_*}-[\bSigma_{\diamond,\mathrm{in}}]_{L_*,L_*}\|+\|[T(\bSigma)]_{L_*,L_*}-[\bSigma_{\diamond,\mathrm{in}}]_{L_*,L_*}\|}
{\sigma_d(\bSigma_{\mathrm{in},*}) \cdot \frac{n_1}{d}\frac{1}{(1+\kappa_P(\bSigma,\epsilon))^2}\sigma_d(\phi([g_1(\bSigma)]_{L_*,L_*}))}.
\end{equation}
Consequently, we will later conclude that the second terms in the RHS's of \eqref{eq:bound_kappa_3T_new_2}  and \eqref{eq:bound_kappa_3T_new_3} are dominated by the first terms on these RHS's.

To establish an upper bound of $\tau$, we approximate  $[g_1(T(\bSigma))]_{L_*,L_*}$ by $[T(\bSigma)]_{L_*,L_*}$:
\begin{align}\label{eq:bound_kappa_3T_new_1}
\|[g_1(T(\bSigma))]_{L_*,L_*}-[T(\bSigma)]_{L_*,L_*}\|\leq \frac{\|[T(\bSigma)]_{L_*,L_*^\perp}\|^2}{\sigma_{D-d}([T(\bSigma)]_{L_*^\perp,L_*^\perp})} \leq \frac{4 \|\bSigma_{\diamond,\mathrm{out}}\|^2}{\gamma\sigma_D(\bSigma)\calS},
\end{align}
where the first inequality follows from \eqref{eq:widetildeT}, and the second inequality uses both \eqref{eq:sigmaD-d} and the bound $\|[T(\bSigma)]_{L_*,L_*^\perp}\|\leq 2 \|\bSigma_{\diamond,\mathrm{out}}\|$, which follows from \eqref{eq:widetildeT2}  in Lemma~\ref{lemma:bound_sigma_out}(a), after replacing the $+$ subscript with the $\diamond$ subscript.

The approximation of $[T(\bSigma)]_{L_*,L_*}$ by $[\bSigma_{\diamond,\mathrm{in}}]_{L_*,L_*}$ follows from \eqref{eq:pertubation} in Lemma~\ref{lemma:bound_sigma_out}(a), after replacing the $+$ subscript with the $\diamond$ subscript:
\begin{equation}\label{eq:bound_kappa_3T_new_1b}
\|[T(\bSigma)]_{L_*,L_*}-[\bSigma_{\diamond,\mathrm{in}}]_{L_*,L_*}\|\leq \|T(\bSigma)-\bSigma_{\diamond,\mathrm{in}}\|\leq 2 \|\bSigma_{\diamond,\mathrm{out}}\|.
\end{equation}

Combining the definition of $\tau$,   \eqref{eq:bound_kappa_3T_new_1},  \eqref{eq:bound_kappa_3T_new_1b}, and the triangle inequality, yields
\begin{equation}
\tau\leq \frac{4 \|\bSigma_{\diamond,\mathrm{out}}\| +\frac{4 \|\bSigma_{\diamond,\mathrm{out}}\|^2}{\gamma\sigma_D(\bSigma)\calS}}{ \sigma_d(\bSigma_{\mathrm{in},*})\cdot \frac{n_1}{d(1+\kappa_P(\bSigma,\epsilon))^2}\sigma_d(\phi([g_1(\bSigma)]_{L_*,L_*}))}.\label{eq:bound_kappa_3T_new_3_RHS2}
\end{equation}
Note that the first inequality in  \eqref{eq:xout2_noisy} and the bound $c_4(\bSigma,\epsilon)\leq 0.1$ (see  \eqref{eq:noisy_conclusion}) imply
\begin{equation*}
{\|\bSigma_{\diamond,\mathrm{out}}\|}\leq\frac{3}{2}\tilde{\calA}\sigma_1\left(\bSigma_{L_*^\perp,L_*^\perp}\right)= \frac{3}{2}\frac{n_0}{D-d}\calA\sigma_1\left(\bSigma_{L_*^\perp,L_*^\perp}\right).
\end{equation*}
This equation and the definitions of  $\hat{\kappa}_2$ (see \eqref{eq:def_kappa_1_2}) and $\hat{\kappa}_1$ (see \eqref{eq:def_kappa_1_2_new}) result in
\begin{equation}
\label{eq:bound_kappa_3T_new_3_RHS2_2}
\frac{\|\bSigma_{\diamond,\mathrm{out}}\|}{\sigma_D(\bSigma)}\leq \frac{3}{2}\frac{n_0}{D-d}\calA\hat{\kappa}_2(\bSigma)
    \ \text{ and } \
\frac{
\|\bSigma_{\diamond,\mathrm{out}}\|}{\sigma_d(\phi([g_1(\bSigma)]_{L_*,L_*}))}
\leq
\frac{3}{2} \frac{n_0}{D-d}\frac{
\calA}{\hat{\kappa}_1(\bSigma)}.
\end{equation}
Applying the first inequality in \eqref{eq:bound_kappa_3T_new_3_RHS2_2}
 and then \eqref{eq:calS1_prelim} yields
 \[
4\|\bSigma_{\diamond,\mathrm{out}}\| +\frac{4 \|\bSigma_{\diamond,\mathrm{out}}\|^2}{\gamma\sigma_D(\bSigma)\calS}\leq \left(4+6\frac{n_0}{D-d}\frac{\calA}{\calS\gamma}\hat{\kappa}_2(\bSigma)\right)\|\bSigma_{\diamond,\mathrm{out}}\|=\left(4+6\frac{\calR}{\gamma}\hat{\kappa}_2(\bSigma)\right)\|\bSigma_{\diamond,\mathrm{out}}\|.
\]
We further bound the RHS of \eqref{eq:bound_kappa_3T_new_3_RHS2}
times $\sigma_d(\bSigma_{\mathrm{in},*})$ (for brevity)
using the above equation, then applying the second equation in \eqref{eq:bound_kappa_3T_new_3_RHS2_2}, and at last the assumption  $\hat{\kappa}_1(\bSigma)\geq \tilde{\kappa}_1$ (see    \eqref{eq:overiterations_noisy}):
\begin{multline}
\frac{4 \|\bSigma_{\diamond,\mathrm{out}}\| +\frac{4 \|\bSigma_{\diamond,\mathrm{out}}\|^2}{\gamma\sigma_D(\bSigma)\calS}}{  \frac{n_1}{d(1+\kappa_P(\bSigma,\epsilon))^2}\sigma_d(\phi([g_1(\bSigma)]_{L_*,L_*}))}\leq
\frac{\|\bSigma_{\diamond,\mathrm{out}}\| (4+6\frac{\calR}{\gamma}\hat{\kappa}_2(\bSigma))}{  \frac{n_1}{d(1+\kappa_P(\bSigma,\epsilon))^2}\sigma_d(\phi([g_1(\bSigma)]_{L_*,L_*}))}
\\
\leq
\frac{ (4+6\frac{\calR}{\gamma}\hat{\kappa}_2(\bSigma))}{  \frac{n_1}{d(1+\kappa_P(\bSigma,\epsilon))^2}}\frac{3}{2}\frac{n_0}{D-d}\frac{\calA}{\hat{\kappa}_1(\bSigma)}\leq (1+\kappa_P(\bSigma,\epsilon))^2\frac{ 6+9\frac{\calR}{\gamma}C_{{\kappa}_2}}{ \dssnr}\frac{\calA}{\tilde{\kappa}_1}.\label{eq:kappa3_est_bound}
\end{multline}
Applying the definition of $\tilde{\kappa}_1$ in \eqref{eq:tildekappa1} and the facts $\kappa_2\geq 1$ (Lemma~\ref{lemma:matrixperturbation}(a)), $\calR(1+\kappa_{\mathrm{in},*})\geq\kappa_{\mathrm{in},*}\geq1$ from \eqref{eq:R_lower_noisy}, and $\gamma<1$, so that the expression in parentheses in \eqref{eq:tildekappa1} is at least $1$ and at least $\calR/\gamma$, yields the two inequalities:
\begin{equation}
\frac{\calA}{\sigma_d(\bSigma_{\mathrm{in},*})\cdot \dssnr\cdot\tilde{\kappa}_1}\leq \frac{1}{C_{\mathrm b}} \ \text{ and } \ \frac{\calR}{\gamma}\cdot \frac{\calA}{\sigma_d(\bSigma_{\mathrm{in},*})\cdot \dssnr\cdot\tilde{\kappa}_1}\leq \frac{1}{C_{\mathrm b}}.
\label{eq:two_bounds_by_one_over_C}
\end{equation}
Combining \eqref{eq:bound_kappa_3T_new_3_RHS2}, \eqref{eq:kappa3_est_bound} (divided by $\sigma_d(\bSigma_{\mathrm{in},*})$), \eqref{eq:two_bounds_by_one_over_C} and $\kappa_P(\bSigma,\epsilon)\leq 0.1$ (see  \eqref{eq:noisy_conclusion}), then applying
$C_{{\kappa}_2} = 7$, and at last
$C_{\mathrm b}\geq 3340/C_E$ (see   \eqref{eq:Cvalue_noisy_final}) result in a useful bound on $\tau$:
\begin{equation}
\label{eq:bound_kappa_3T_new_3_RHS}
\tau\leq  1.1^2\cdot \left(\frac{6}{C_{\mathrm b}}+\frac{9\cdot C_{\kappa_2}}{C_{\mathrm b}}\right)<\frac{83.5}{C_{\mathrm b}}\leq \frac{C_E}{40}.
\end{equation}

\textbf{Fourth step: Finalizing the proof.}
In view of \eqref{eq:def_kappa_3_new} and  \eqref{eq:bound_kappa_3T_new_2}-\eqref{eq:kappa3_bar}, we
write
\begin{align*}
x &= \frac{n_1}{d}\,\frac{1}{\bigl(1-\kappa_P(\bSigma,\epsilon)\bigr)^2}
     \,\sigma_1\left(\phi\left(\bSigma_{L_*,L_*}\right)\right),
\hspace{0.75in}\epsilon_1 =
 \frac{\bigl\|[T(\bSigma)-\bSigma_{\diamond,\mathrm{in}}]_{L_*,L_*}\bigr\|}
      {\sigma_d(\bSigma_{\mathrm{in},*})},\\[0.5em]
y &= \frac{n_1}{d}\,\frac{1+C_E\left(1-{1}/{\bar{\kappa}_3(\bSigma)}\right)}
              {\bigl(1+\kappa_P(\bSigma,\epsilon)\bigr)^2}
     \,\sigma_d\left(\phi\left([g_1(\bSigma)]_{L_*,L_*}\right)\right),\hspace{0.05in}
\epsilon_2 =
 \frac{\bigl\|[g_1(T(\bSigma))-\bSigma_{\diamond,\mathrm{in}}]_{L_*,L_*}\bigr\|}
      {\sigma_d(\bSigma_{\mathrm{in},*})}
\end{align*}
and reduce the proof of  $\hat{\kappa}_3(T(\bSigma))\leq C_{\kappa_3}$ to verifying
\begin{equation}\label{eq:xyepsilon_representation}
\frac{x+\epsilon_1}{y-\epsilon_2}\leq C_{\kappa_3}.
\end{equation}

We derive several initial estimates. Combining the definitions of $\tau$ (see \eqref{eq:def_tau}), $\epsilon_1$, $\epsilon_2$ and $y$, then applying $\bar{\kappa}_3(\bSigma) \geq 1$ (see \eqref{eq:kappa3_bar}), next    \eqref{eq:bound_kappa_3T_new_3_RHS}, and at last $C_E\leq 1$ (see Lemma~\ref{lemma:TME_contraction}) result in
\begin{equation}\label{eq:tau_bound2}\frac{\epsilon_1+\epsilon_2}{y}= \frac{\tau}{1+C_E(1-{1}/{\bar{\kappa}_3(\bSigma)})}\leq \tau\leq \frac{C_E}{40}\leq\frac{1}{40}.\end{equation}
Consequently, $\epsilon_2\leq y/40$ and thus $y-\epsilon_2\geq \frac{39}{40}y$. Applying the latter inequality yields
\begin{equation}\label{eq:xyepsilon_representation2}
\frac{x+\epsilon_1}{y-\epsilon_2}= \frac{x}{y-\epsilon_2}+\frac{\epsilon_1}{y-\epsilon_2}=\frac{x}{y}+\frac{x}{y}\,\frac{\epsilon_2}{y-\epsilon_2}+\frac{\epsilon_1}{y-\epsilon_2}\leq \frac{x}{y}\left(1+\frac{40}{39}\frac{\epsilon_2}{y}\right)+\frac{40}{39}\frac{\epsilon_1}{y}.
\end{equation}
Using
the definition
$\hat{\kappa}_3(\bSigma)
  = {\sigma_1\left(\phi\left(\bSigma_{L_*,L_*}\right)\right)}/         {\sigma_d\left(\phi\left([g_1(\bSigma)]_{L_*,L_*}\right)\right)}$ (see  \eqref{eq:def_kappa_3_new} yields
\[
\frac{x}{y} =  \frac{\bigl(1+\kappa_P(\bSigma,\epsilon)\bigr)^2}
     {\bigl(1-\kappa_P(\bSigma,\epsilon)\bigr)^2}
\cdot
\frac{\hat{\kappa}_3(\bSigma)}
     {1+C_E\left(1-1/{\bar{\kappa}_3(\bSigma)}\right)}.
\]
Combining the above equation with \eqref{eq:xyepsilon_representation2} and
$
\max(\epsilon_1/y,\epsilon_2/y)\leq {(\epsilon_1+\epsilon_2)}/{y}\leq \tau
$
(which follows from an intermediate inequality of \eqref{eq:tau_bound2}) gives
\begin{align}
\frac{x+\epsilon_1}{y-\epsilon_2}
\leq
\frac{\bigl(1+\kappa_P(\bSigma,\epsilon)\bigr)^2}
     {\bigl(1-\kappa_P(\bSigma,\epsilon)\bigr)^2}
\cdot
\frac{\hat{\kappa}_3(\bSigma)}
     {1+C_E\left(1-1/{\bar{\kappa}_3(\bSigma)}\right)}
\cdot (1+2\tau)
+ 2\tau. \label{eq:kappa3_est}
\end{align}

 Finally, we will prove \eqref{eq:xyepsilon_representation} by showing that the RHS of \eqref{eq:kappa3_est} is bounded by $C_{\kappa_3}$, considering two complementary cases.

 Case 1: $1\leq \bar{\kappa}_3(\bSigma)\leq 2$.
 We further bound the RHS \eqref{eq:kappa3_est} using $\kappa_P(\bSigma,\epsilon)\leq 0.1$ (see  \eqref{eq:noisyassumption20}), the trivial bound ${C_E(1-1/\bar{\kappa}_3(\bSigma))}
\geq 0$,
 $\tau\leq 1/40$ (see \eqref{eq:tau_bound2}),   $\hat{\kappa}_3(\bSigma)\leq 16  \kappa_{\mathrm{in},*}$, which follows from \eqref{eq:kappa3_bar} and  $\bar{\kappa}_3(\bSigma)\leq 2$, and the definition of $C_{\kappa_3}$ (see \eqref{eq:c_kappa3}), to obtain
 \begin{align*}
\frac{x+\epsilon_1}{y-\epsilon_2}
\leq \left(\frac{1.1}{0.9}\right)^2 \cdot 16\kappa_{\mathrm{in},*}\cdot 1.05+0.05\leq 25.1\,\kappa_{\mathrm{in},*}+0.05\leq C_{\kappa_3}.
 \end{align*}

  Case 2: $\bar{\kappa}_3(\bSigma)\geq 2$ (this case is nonvacuous, since by \eqref{eq:kappa3_bar} it corresponds to $16\kappa_{\mathrm{in},*}\leq\hat{\kappa}_3(\bSigma)\leq C_{\kappa_3}$). We bound the first term
in the RHS of \eqref{eq:kappa3_est} applying $\kappa_P(\bSigma,\epsilon)\leq C_E/20$ (see \eqref{eq:noisyassumption20}),  $\tau\leq C_E/40$ (see \eqref{eq:bound_kappa_3T_new_3_RHS}), and the assumption  $\hat{\kappa}_3(\bSigma)\leq C_{\kappa_3}$ (see \eqref{eq:overiterations_noisy}), and then using:
$(1+x/20)^2(1+x/20)(1+x/10)\leq (1-x/20)^2(1+x/2) $ for $0\leq x\leq1$:
   \[
\frac{(1+\kappa_P(\bSigma,\epsilon))^2}{(1-\kappa_P(\bSigma,\epsilon))^2} \frac{\hat{\kappa}_3(\bSigma)}{1+C_E/2}(1+2\tau)
\leq
\frac{(1+C_E/20)^2}{(1-C_E/20)^2} \frac{C_{\kappa_3}}{1+C_E/2}\left(1+\frac{C_E}{20}\right)
\leq  \frac{C_{\kappa_3}}{1+C_E/10}.
 \]
The second term on the RHS of \eqref{eq:kappa3_est} is bounded by $C_E/20$ since $\tau\leq C_E/40$. We show that the sum of these two bounds is below  $C_{\kappa_3}$ and thus conclude the proof. Indeed, using $C_{\kappa_3}\geq 1$ (see \eqref{eq:c_kappa3}) and $C_E\leq 1$: $C_{\kappa_3}-\frac{C_{\kappa_3}}{1+C_E/10}=\frac{C_E}{10}\cdot \frac{C_{\kappa_3}}{1+C_E/10} > \frac{C_E}{10}\cdot \frac{1}{2}=\frac{C_E}{20}$.
\qed

\subsection{A one-step upper bound on the growth of \texorpdfstring{$\hat{\kappa}_1$}{κ̂1}}\label{sec:growth_upper}
The following lemma complements \eqref{eq:overiterations_noisy}: while $\hat{\kappa}_1$ grows by a factor of at least $C_0$ per iteration, it cannot grow by more than a constant factor $G$, which depends on the data but not on the iterate. It is used in the proof of part (b) of Theorem~\ref{thm:noisy} in Section~\ref{sec:reduce_noisy-thm}.
\begin{lemma}\label{lemma:growth_upper}
If $\bSigma$ satisfies the hypotheses of \eqref{eq:overiterations_noisy}, then
\[
\hat{\kappa}_1(T(\bSigma))\leq G\,\hat{\kappa}_1(\bSigma), \qquad \text{where } \ G:=\dssnr\Big(\frac{9\,C_{\kappa_3}\calR}{\gamma\calA}+0.11\Big).
\]
\end{lemma}
\begin{proof}
We recall the definitions of $\hat{\kappa}_1$, $\hat{\kappa}_2$ and $\hat{\kappa}_3$ in \eqref{eq:def_kappa_1_2_new}, \eqref{eq:def_kappa_1_2} and \eqref{eq:def_kappa_3_new}, and the decomposition $\bSigma_+=\bSigma_{\diamond,\mathrm{in}}+\bSigma_{\diamond,\mathrm{out}}$ of Section~\ref{sec:proof_idea_or_overiter_noisy}. By the hypotheses of \eqref{eq:overiterations_noisy}, Lemma~\ref{lemma:constants_noisy} applies, so $\kappa_P(\bSigma,\epsilon)\leq 0.1$ and $c_4(\bSigma,\epsilon)\leq 0.1$.

\textbf{Upper bound on the numerator of $\hat{\kappa}_1(T(\bSigma))$.} Since $[g_1(\bM)]_{L_*,L_*}\psdleq \bM_{L_*,L_*}$ for any $\bM\in S_+$ (see \eqref{eq:defg1}) and $\phi$ preserves the Loewner order,
\[
\sigma_d(\phi([g_1(T(\bSigma))]_{L_*,L_*}))\leq\sigma_1(\phi([T(\bSigma)]_{L_*,L_*}))\leq\frac{n_1}{d}\frac{\sigma_1(\phi(\bSigma_{L_*,L_*}))}{(1-\kappa_P(\bSigma,\epsilon))^2}+\frac{2\|\bSigma_{\diamond,\mathrm{out}}\|}{\sigma_d(\bSigma_{\mathrm{in},*})},
\]
where the last inequality follows from \eqref{eq:bound_kappa_3T_new_2} and \eqref{eq:bound_kappa_3T_new_1b}. By \eqref{eq:def_kappa_1_2_new} and \eqref{eq:def_kappa_3_new},
\[
\sigma_1(\phi(\bSigma_{L_*,L_*}))=\hat{\kappa}_3(\bSigma)\,\sigma_d(\phi([g_1(\bSigma)]_{L_*,L_*}))=\hat{\kappa}_3(\bSigma)\,\hat{\kappa}_1(\bSigma)\,\sigma_1(\bSigma_{L_*^\perp,L_*^\perp}).
\]

\textbf{Lower bound on the denominator of $\hat{\kappa}_1(T(\bSigma))$.} By \eqref{eq:sigmaD-d}, whose proof only uses \eqref{eq:xout3} and thus holds verbatim in the noisy setting, and by the definition of $\hat{\kappa}_2$,
\[
\sigma_1([T(\bSigma)]_{L_*^\perp,L_*^\perp})\geq \sigma_{D-d}([T(\bSigma)]_{L_*^\perp,L_*^\perp})\geq \gamma\,\sigma_D(\bSigma)\,\calS=\frac{\gamma\,\calS\,\sigma_1(\bSigma_{L_*^\perp,L_*^\perp})}{\hat{\kappa}_2(\bSigma)}.
\]

\textbf{Conclusion.} Combining the two bounds and applying $\hat{\kappa}_2(\bSigma)\leq 7$, $\hat{\kappa}_3(\bSigma)\leq C_{\kappa_3}$ and $\kappa_P(\bSigma,\epsilon)\leq 0.1$ yields
\begin{equation}\label{eq:growth_two_terms}
\hat{\kappa}_1(T(\bSigma))\leq \frac{7\,C_{\kappa_3}}{0.81\,\gamma\,\calS}\cdot\frac{n_1}{d}\,\hat{\kappa}_1(\bSigma)+\frac{14\,\|\bSigma_{\diamond,\mathrm{out}}\|}{\gamma\,\calS\,\sigma_d(\bSigma_{\mathrm{in},*})\,\sigma_1(\bSigma_{L_*^\perp,L_*^\perp})}.
\end{equation}
For the first term we note that $\calS=\tilde{\calA}/\calR$ and $\tilde{\calA}=n_0\calA/(D-d)$, so that $(n_1/d)/\calS=\dssnr\calR/\calA$; hence this term equals $(7/0.81)\,C_{\kappa_3}\dssnr\calR\,\hat{\kappa}_1(\bSigma)/(\gamma\calA)\leq 8.65\,C_{\kappa_3}\dssnr\calR\,\hat{\kappa}_1(\bSigma)/(\gamma\calA)$. For the second term, \eqref{eq:xout2_noisy} and $c_4(\bSigma,\epsilon)\leq 0.1$ give $\|\bSigma_{\diamond,\mathrm{out}}\|\leq 1.1\,\tilde{\calA}\,\sigma_1(\bSigma_{L_*^\perp,L_*^\perp})$, so this term is at most $15.4\,\calR/(\gamma\sigma_d(\bSigma_{\mathrm{in},*}))$. On the other hand, the hypothesis $\hat{\kappa}_1(\bSigma)\geq\tilde{\kappa}_1$, the definition \eqref{eq:tildekappa1} of $\tilde{\kappa}_1$, and the facts $\kappa_2\geq 1$ and $1+\kappa_{\mathrm{in},*}\geq 2$ imply $\hat{\kappa}_1(\bSigma)\geq 2C_{\mathrm b}\calA\calR/(\gamma\dssnr\sigma_d(\bSigma_{\mathrm{in},*}))$, so that the second term is at most $7.7\,\dssnr\,\hat{\kappa}_1(\bSigma)/(C_{\mathrm b}\calA)\leq 0.11\,\dssnr\,\hat{\kappa}_1(\bSigma)$, where we used $C_{\mathrm b}\geq 70$ and $\calA\geq 1$ (see \eqref{eq:calA>1}). Plugging both estimates into \eqref{eq:growth_two_terms} concludes the proof.
\end{proof}

\subsection{Proof of Proposition~\ref{prop:noisy_high_init}}\label{sec:proof_noisy_high_init}
We first give an explicit choice of the accuracy constant. Define
\begin{align*}
c_*&:=\frac{1}{100C_{\kappa_3}}\min\Big\{\frac17\min\{\dssnr/\gamma-1,1,C_E/2\}^2,\frac{1}{d\,\dssnr}\min\{\dssnr/\gamma-1,1\}\Big\}>0,\\
M&:=\max\Big\{2(1+\sqrt7)^2\Big(\frac{\kappa_{\mathrm{in},*}^2\calA}{\dssnr\,c_*}+2dC_{\kappa_3}\Big),(1+\sqrt7)\sqrt{\kappa_{\mathrm{in},*}\Big(d+\frac{D-d}{\dssnr}\Big)}\Big\}.
\end{align*}
The identities
\[
\frac{\tilde\calA}{n_1}=\frac{\calA}{d\,\dssnr},\qquad
\frac{n_0}{n_1(D-d)}=\frac{1}{d\,\dssnr},\qquad
\frac{dN}{n_1}=d+\frac{D-d}{\dssnr},
\]
and $C_{\kappa_3}=26\kappa_{\mathrm{in},*}+1$ show that $M$ has the asserted dependence.

Write $\bSigma:=\bSigma^{(0)}$ in blocks as
\[
\bSigma=\begin{pmatrix}\bA&\bB\\\bB^\top&\mathbf{C}\end{pmatrix},\qquad
\mathbf{G}:=\bA-\bB\mathbf{C}^{-1}\bB^\top,\qquad t:=\sigma_1(\mathbf{C}).
\]
Thus $\sigma_d(\mathbf{G})=\kappa_1t$, $\sigma_1(\bA)=\kappa_3\kappa_1t$, and $\sigma_D(\bSigma)=t/\kappa_2$.
For an inlier $\bx=(\mathbf{z},\mathbf{e})$, where $\|\mathbf{e}\|\leq\epsilon\|\mathbf{z}\|$, the triangle inequality applied to $\bSigma^{-1/2}\bx$ and the identity $[\bSigma^{-1}]_{L_*,L_*}=\mathbf{G}^{-1}$ give
\begin{equation}\label{eq:high_init_den_upper}
\bx^\top\bSigma^{-1}\bx
\leq\left(\frac{\|\mathbf{z}\|}{\sqrt{\kappa_1t}}+\frac{\|\mathbf{e}\|}{\sqrt{\sigma_D(\bSigma)}}\right)^2
\leq\frac{\|\mathbf{z}\|^2}{t}(\kappa_1^{-1/2}+\epsilon\sqrt7)^2.
\end{equation}
Let $\bS:=\bSigma_{\mathrm{in},*}$. Its fixed-point equation is
$\sum_{\mathrm{in}}\mathbf{z}\mathbf{z}^\top/(\mathbf{z}^\top\bS^{-1}\mathbf{z})=(n_1/d)\bS$.
Since $\mathbf{z}^\top\bS^{-1}\mathbf{z}\geq\|\mathbf{z}\|^2/\sigma_1(\bS)$,
\begin{equation}\label{eq:high_init_permeance}
\sum_{\mathrm{in}}\frac{\mathbf{z}\mathbf{z}^\top}{\|\mathbf{z}\|^2}
\psdgeq\frac{n_1\bS}{d\sigma_1(\bS)}\psdgeq\frac{n_1}{d\kappa_{\mathrm{in},*}}\bI_d.
\end{equation}
Set $\bQ:=T_1(\bSigma)$ and let $\bS_0$ contain its inlier $L_*,L_*$ block and zero elsewhere. Equations~\eqref{eq:high_init_den_upper}--\eqref{eq:high_init_permeance} imply
\begin{equation}\label{eq:high_init_signal_min}
\sigma_d(\bS_0)\geq\frac{tn_1}{d\kappa_{\mathrm{in},*}(\kappa_1^{-1/2}+\epsilon\sqrt7)^2}.
\end{equation}
Two further denominator bounds are useful. By completing the square in $\bx^\top\bSigma^{-1}\bx$, its minimum over $\mathbf{e}$ with $\mathbf{z}$ fixed is $\mathbf{z}^\top\bA^{-1}\mathbf{z}$, and its minimum over $\mathbf{z}$ with $\mathbf{e}$ fixed is $\mathbf{e}^\top\mathbf{C}^{-1}\mathbf{e}$. Consequently,
\begin{equation}\label{eq:high_init_den_lower}
\bx^\top\bSigma^{-1}\bx\geq\frac{\|\mathbf{z}\|^2}{\sigma_1(\bA)},\qquad
\bx^\top\bSigma^{-1}\bx\geq\frac{\|\mathbf{e}\|^2}{t}.
\end{equation}
The second inequality implies $\|\bQ_{L_*^\perp,L_*^\perp}\|\leq Nt$, because each point contributes a matrix of norm at most $t$ to this block.

If $\kappa_1\geq\epsilon^{-2}$, then \eqref{eq:high_init_signal_min} gives
$\sigma_d(\bQ_{L_*,L_*})\geq tn_1/[d\kappa_{\mathrm{in},*}\epsilon^2(1+\sqrt7)^2]$.
Lemma~\ref{lemma:subspace}, with its square-root block-ratio bound, therefore yields
\[
\sin\angle(L^{(1)},L_*)\leq
\epsilon(1+\sqrt7)\sqrt{\frac{d\kappa_{\mathrm{in},*}N}{n_1}}
\leq M\epsilon.
\]
Here $L^{(1)}$ is the top-$d$ eigenspace of $\bQ$: the STE spectral step leaves these eigenvectors unchanged.

It remains to consider $C_{\kappa_1}/\kappa_{\mathrm{in},*}<\kappa_1<\epsilon^{-2}$.
In this range,
\[
\sigma_d(\bS_0)\geq\frac{tn_1\kappa_1}{d\kappa_{\mathrm{in},*}(1+\sqrt7)^2}.
\]
The first bound in \eqref{eq:high_init_den_lower} and the cone assumption show that the inlier off-diagonal block has norm at most $\epsilon n_1\sigma_1(\bA)$ and the inlier orthogonal block has norm at most $\epsilon^2 n_1\sigma_1(\bA)$.
The symmetric matrix with these two off-diagonal blocks and zero diagonal blocks has the same norm as either off-diagonal block. By the triangle inequality, and the outlier bound \eqref{eq:xout2},
\[
\|\bQ-\bS_0\|\leq t\tilde\calA+(\epsilon+\epsilon^2)n_1\kappa_3\kappa_1t
\leq t\tilde\calA+2\epsilon n_1\kappa_3\kappa_1t.
\]
Since $0<\epsilon\leq1$, the definition \eqref{eq:C_kappa_1} gives
$C_{\kappa_1}\geq c_*/\epsilon$, and hence
$\kappa_1>c_*/(\kappa_{\mathrm{in},*}\epsilon)$.
Using $\kappa_3\leq C_{\kappa_3}/\kappa_{\mathrm{in},*}$, we obtain
\begin{equation}\label{eq:high_init_relative_residual}
\rho:=\frac{\|\bQ-\bS_0\|}{\sigma_d(\bS_0)}\leq d\kappa_{\mathrm{in},*}(1+\sqrt7)^2\Big(\frac{\tilde\calA}{n_1\kappa_1}+2\kappa_3\epsilon\Big)\leq (1+\sqrt7)^2\Big(\frac{\kappa_{\mathrm{in},*}^2\calA}{\dssnr\,c_*}+2dC_{\kappa_3}\Big)\epsilon\leq\frac{M\epsilon}{2}.
\end{equation}
If $M\epsilon<1$, then $\rho<1/2$. The signal block of $\bQ$ is augmented only by the positive semidefinite outlier signal block. The orthogonal and off-diagonal blocks have norm at most $\|\bQ-\bS_0\|$. Therefore,
\[
\begin{aligned}
\sigma_d(\bQ_{L_*,L_*})-\sigma_1(\bQ_{L_*^\perp,L_*^\perp})
&\geq\sigma_d(\bS_0)(1-\rho)>0,\\
\|\bQ_{L_*,L_*^\perp}\|&\leq\rho\,\sigma_d(\bS_0).
\end{aligned}
\]
Lemma~\ref{lemma:blockDK} gives
\[
\sin\angle(L^{(1)},L_*)\leq\frac{\rho}{1-\rho}\leq2\rho\leq M\epsilon.
\]
If $M\epsilon\geq1$, the same conclusion follows from
$\sin\angle(L^{(1)},L_*)\leq1$. This proves the proposition.

\section{Proof of Auxiliary Lemmas}\label{sec:proof_lemmas}

\begin{proof}[Proof of Lemma~\ref{lemma:TME_contraction}]

We will prove the following general statement: For a dataset $\calX\subset\R^p$, which is not contained in a union of two $(p-1)$-subspaces, with $|\calX|=n$ and TME solution $\bSigma_* \in S_{++}(p)$,
there exists a constant $0<C_E\leq1$ such that for all $\bSigma\in S_{++}(p)$,
\begin{equation}\label{eq:noisy_assumption1_general}
\frac{\sigma_p(\bSigma_{*}^{-0.5}T_1(\bSigma)\,\bSigma_{*}^{-0.5})}{\sigma_p(\bSigma_{*}^{-0.5}\bSigma\,\bSigma_{*}^{-0.5})}\geq\frac{n}{p} \left(1+C_E \cdot\left(1- \frac{\sigma_p(\bSigma_{*}^{-0.5}\bSigma\,\bSigma_{*}^{-0.5})}{\sigma_1(\bSigma_{*}^{-0.5}\bSigma\,\bSigma_{*}^{-0.5})}\right)\right).
\end{equation}
Lemma~\ref{lemma:TME_contraction} is a special case of \eqref{eq:noisy_assumption1_general} with $p=d$ and $\calX$ replaced by $\{\bU_{L_*}^\top\bx\mid\bx\in\calX_{\mathrm{in}}\}$.

We note that for an invertible $\bA \in \R^{p\times p}$, the TME operator $T_1(\bSigma)  \equiv \sum_{\bx\in\calX}{\bx\bx^\top}/{\bx^\top\bSigma^{-1}\bx}$ and the TME solution, $\bSigma_*$,  are $\bA$-equivariant in the following way: If $\tilde{\calX}=\{\bA\bx:\bx\in\calX\}$, and $\tilde{T}_1$ and $\tilde{\bSigma}_*$ denote the corresponding TME operator and solution for $\tilde{\calX}$, then  $\tilde{T}_1(\bA\bSigma \bA^\top)=\bA T_1(\bSigma)\bA^\top$ for any $\bSigma \in S_{++}(p)$, and consequently
\begin{equation}
\label{eq:tme_equivariance}
\tilde{\bSigma}_*=\bA\bSigma_*\bA^\top.
\end{equation}
Using this $\bA$-equivariance property, we  may assume without loss of generality that $\bSigma_{*}=\bI$ by transforming the data. We thus reduce \eqref{eq:TME_definition} and \eqref{eq:noisy_assumption1_general} to the following equations, respectively:
\begin{equation}\sum_{\bx\in\calX}{\bx\bx^\top}/{\|\bx\|^2}=\frac{n}{p}\bI,\label{eq:TME_contraction_fixed}
\end{equation}
\begin{equation}
\label{eq:noisy_assumption1_general_simple}
\frac{\sigma_p(T_1(\bSigma))}{\sigma_p(\bSigma)}\geq \frac{n}{p}\left( 1+C_E \cdot\left( 1-\frac{\sigma_p(\bSigma)}{\sigma_1(\bSigma)}\right)\right).
\end{equation}

To estimate $\sigma_p(T_1(\bSigma))$, we denote by $\bv_i$ the $i$-th unit eigenvector of $\bSigma$ corresponding to $\sigma_i(\bSigma)$, where the eigenvalues are ordered in decreasing order.
We note that for $\bx \in \calX$,
\begin{multline*}
\bx^\top\bSigma^{-1}\bx\!=\!\sum_{i=1}^p \sigma_i(\bSigma)^{-1}|\bx^\top \bv_i|^2 \!\leq\! \sigma_1(\bSigma)^{-1}|\bx\!^\top\!\bv_1|^2\!+\!\sum_{i=2}^p \sigma_p(\bSigma)^{-1}|\bx\!^\top \! \bv_i|^2\!=\! \sigma_p(\bSigma)^{-1}\|\bx\|^2\\+(\sigma_1(\bSigma)^{-1}-\sigma_p(\bSigma)^{-1})|\bx^\top \bv_1|^2  \!=\!\sigma_p^{-1}(\bSigma) (\|\bx\|^2-|\bx^\top \bv_1|^2(1-\sigma_p(\bSigma) \sigma_1(\bSigma)^{-1})).
\end{multline*}
Consequently,
\begin{equation}
T_1(\bSigma) \equiv \sum_{\bx\in\calX}\frac{\bx\bx^\top}{\bx^\top\bSigma^{-1}\bx}\psdgeq
\sigma_p(\bSigma)\sum_{\bx\in\calX}\frac{\bx\bx^\top}{{\|\bx\|^2}-|\bx^\top\bv_1|^2(1-\frac{\sigma_p(\bSigma)}{\sigma_1(\bSigma)})}.
\label{eq:assist_T1_bound}
\end{equation}
Set $\rho_{\bx}:=|\bx^\top\bv_1|^2/\|\bx\|^2\in[0,1]$ and $z:=1-\sigma_p(\bSigma)/\sigma_1(\bSigma)\in[0,1)$. The elementary inequality $1/(1-t)\geq1+t$ for $0\leq t<1$ gives
\[
\frac{1}{\|\bx\|^2-|\bx^\top\bv_1|^2z}
=\frac{1}{\|\bx\|^2(1-\rho_{\bx}z)}
\geq\frac{1+\rho_{\bx}z}{\|\bx\|^2}.
\]
Define
\[
C_E:=\frac{p}{n}\min_{\|\bv_1\|=1}
\sigma_p\left(\sum_{\bx\in\calX}\rho_{\bx}
\frac{\bx\bx^\top}{\|\bx\|^2}\right).
\]
Combining \eqref{eq:assist_T1_bound}, the preceding inequality, and \eqref{eq:TME_contraction_fixed}
yields
\begin{align}\nonumber
&\sigma_p(T_1(\bSigma))\geq \sigma_p(\bSigma)\sigma_p\left(\sum_{\bx\in\calX}\frac{\bx\bx^\top}{{\|\bx\|^2}-|\bx^\top\bv_1|^2\left(1-\frac{\sigma_p(\bSigma)}{\sigma_1(\bSigma)}\right)}\right)
\\
\label{eq:lemma2sigmad}
&\geq  \sigma_p(\bSigma)\sigma_p\left( \sum_{\bx\in\calX}(1+\rho_{\bx}z)\frac{\bx\bx^\top}{\|\bx\|^2}\right)
\\
\nonumber
&\geq \sigma_p(\bSigma)\left(\frac{n}{p}+z\sigma_p\left( \sum_{\bx\in\calX}\rho_{\bx}\frac{\bx\bx^\top}{\|\bx\|^2}\right)\right)
\geq \sigma_p(\bSigma)\frac{n}{p}\left(1+C_E\left(1-\frac{\sigma_p(\bSigma)}{\sigma_1(\bSigma)}\right)\right).
\end{align}

The minimized eigenvalue is a continuous function of $\bv_1$ on the compact unit sphere, so its minimum is attained. We next show that $C_E$ is strictly positive.
Observe that $\rho_{\bx}=0$ only when $\bx \perp \bv_{1}$.
Hence, $C_E=0$ only if all $\bx \in \calX$ with $\bx \not\perp \bv_{1}$ lie on a $(p-1)$-subspace of $\R^p$. Since all $\bx \in \calX$ with $\bx \perp \bv_{1}$ lie on a $(p-1)$-subspace of $\R^p$ and we assumed that $\calX$ is not contained in the union of two $(p-1)$-dimensional subspaces of $\R^p$, $C_E>0$. In addition, we may assume \(C_E \leq 1\), since otherwise we can replace \(C_E\) by \(\min(C_E,1)\). Therefore, \eqref{eq:noisy_assumption1_general_simple} is established and Lemma~\ref{lemma:TME_contraction} is proved.
\end{proof}

\begin{proof}[Proof of Lemma \ref{lemma:g1}]
(a) A direct application of~\cite[Theorem 1.3.3]{bhatia2009positive} implies that $g_1(\bSigma) \in S_+$ whenever $\bSigma \in S_+$.  It is also obvious by the definition of $g_1(\bSigma)$ in \eqref{eq:defg1} that $\rank(g_1(\bSigma))\leq d$.
In addition,
we claim that $g_2(\bSigma)$ is also positive semidefinite and its rank is at most $D-d$. This can be seen by following its definition in \eqref{eq:defg2} and further expressing $g_2(\bSigma)$ as follows:
\begin{align*}&
\begin{pmatrix}
\bSigma_{L_*,L_*^\perp}\bSigma^{-1}_{L_*^\perp,L_*^\perp}\\
 \bI
\end{pmatrix}
\bSigma_{L_*^\perp,L_*^\perp}\Big(\bSigma^{-1}_{L_*^\perp,L_*^\perp}\bSigma_{L_*^\perp,L_*},\bI\Big)\\=&\begin{pmatrix}
\bSigma_{L_*,L_*^\perp}\bSigma^{-1}_{L_*^\perp,L_*^\perp}\\
 \bI
\end{pmatrix}\bSigma_{L_*^\perp,L_*^\perp}^{0.5}\left(\begin{pmatrix}
\bSigma_{L_*,L_*^\perp}\bSigma^{-1}_{L_*^\perp,L_*^\perp}\\
 \bI
\end{pmatrix}\bSigma_{L_*^\perp,L_*^\perp}^{0.5}\right)^\top.
\end{align*}
Combining  $g_2(\bSigma) \psdgeq \bm{0}$ and the relationship   $g_1(\bSigma)=\bSigma-g_2(\bSigma)$, we conclude that $g_1(\bSigma)\psdleq \bSigma$.

If $\bSigma \in S_{++}$, then by  \eqref{eq:defg2} and \eqref{eq:defg1},
$\rank(g_2(\bSigma))=D-d$  and  $\rank(g_1(\bSigma))=d$.
Combining this with $g_1(\bSigma)\in S_+$ gives $\sigma_d(g_1(\bSigma))>0$; equivalently, its $L_*\times L_*$ block is positive definite.

(b) Since the range of $\bSigma-\bX$ is
contained in $L_*$, we can express $\bX$ as follows:
\begin{equation}
\bX=\begin{pmatrix}
\bX_{L_*,L_*} & \bSigma_{L_*,L_*^\perp}  \\
\bSigma_{L_*^\perp,L_*} & \bSigma_{L_*^\perp,L_*^\perp}
\end{pmatrix}.
\label{eq:express_X}
\end{equation}
Since $\bX \in S_+$,  \cite[Theorem 1.3.3]{bhatia2009positive} implies that $\bX_{L_*,L_*}\psdgeq \bSigma_{L_*,L_*^\perp}\bSigma^{-1}_{L_*^\perp,L_*^\perp}\bSigma_{L_*^\perp,L_*}$, and as a result,
\begin{align*}
&[g_1(\bSigma)-(\bSigma-\bX)]_{L_*,L_*}=(\bSigma_{L_*,L_*}-\bSigma_{L_*,L_*^\perp}\bSigma^{-1}_{L_*^\perp,L_*^\perp}\bSigma_{L_*^\perp,L_*})-(\bSigma_{L_*,L_*}-\bX_{L_*,L_*})\\=&\bX_{L_*,L_*}-\bSigma_{L_*,L_*^\perp}\bSigma^{-1}_{L_*^\perp,L_*^\perp}\bSigma_{L_*^\perp,L_*}\psdgeq 0.
\end{align*}
This observation, the fact $g_1(\bSigma)\in S_+$ and \eqref{eq:express_X} imply  $g_1(\bSigma) \psdgeq \bSigma-\bX$.\end{proof}
\begin{proof}[Proof of Lemma~\ref{lemma:matrixperturbation}]
(a) This part follows from the Courant-Fischer min-max theorem~\cite[Proposition~1.3.2]{tao2012topics}. For completeness we verify it as follows:
\begin{align*}
&\sigma_{D-k+1}(\bSigma)=\min_{\dim(V)=k} \ \max_{\bv\in V: \|\bv\|=1}\bv^\top\bSigma\bv=\min_{\bV\in O(D,k)} \ \max_{\bu\in\R^k: \|\bu\|=1}\bu^\top\bV^\top\bSigma\bV\bu\\=&\min_{\bV\in O(D,k)} \sigma_1(\bV^\top\bSigma\bV) \leq \sigma_1(\bU_0^\top\bSigma\bU_0).\end{align*}
Lastly, $\sigma_1(\bSigma)\geq \sigma_1(\bU_0^\top\bSigma\bU_0)$ since both $\bU_0^\top\bSigma\bU_0$ and $\bSigma$ are positive semidefinite and thus $\sigma_1(\cdot)=\|\cdot\|$.

On the other hand, \begin{align*}&\sigma_{k}(\bSigma)=\max_{\dim(V)=k} \ \min_{\bv\in V: \|\bv\|=1}\bv^\top\bSigma\bv=\max_{\bV\in O(D,k)} \ \min_{\bu\in\R^k: \|\bu\|=1}\bu^\top\bV^\top\bSigma\bV\bu\\=&\max_{\bV\in O(D,k)}\sigma_k(\bV^\top\bSigma\bV)\geq
\sigma_k(\bU_0^\top\bSigma\bU_0)\end{align*} and the part $\sigma_k(\bU_0^\top\bSigma\bU_0)\geq \sigma_{D}(\bSigma)$ follows from $\sigma_D(\bSigma)=\min_{\bv:  \|\bv\|=1}\bv^\top\bSigma\bv$.

(b) This is a well-known eigenvalue perturbation inequality, see e.g., (1.63) of~\cite{tao2012topics}.
\end{proof}
\begin{proof}[Proof of Lemma~\ref{lemma:subspace}]
Compression interlacing gives $\sigma_d(\bSigma)\geq \sigma_d(\bSigma_{L_*,L_*})$.
The spectral decomposition therefore implies
$\bSigma\psdgeq\sigma_d(\bSigma)\bP_{\hat L}\psdgeq \sigma_d(\bSigma_{L_*,L_*})\bP_{\hat L}$.
Compressing this inequality to $L_*^\perp$ gives
\[
\bSigma_{L_*^\perp,L_*^\perp}
\psdgeq \sigma_d(\bSigma_{L_*,L_*})\,\bU_{L_*^\perp}^\top\bP_{\hat L}\bU_{L_*^\perp}.
\]
Since
$\|\bU_{L_*^\perp}^\top\bP_{\hat L}\bU_{L_*^\perp}\|
=\sin^2\angle(\hat L,L_*)$,
taking spectral norms proves the claim.
\end{proof}
\begin{proof}[Proof of Lemma~\ref{lemma:sigmaD}]
Write $x=\sigma_{D-d}(\bSigma_{L_*^{\perp},L_*^{\perp}})$, $y=\sigma_{d}(\bSigma_{L_*,L_*})$ and $z=\|\bSigma_{L_*,L_*^{\perp}}\|$. Let $a=({(x+y)-\sqrt{(x-y)^2+4z^2}})/{2}$. If $z=0$, then $\bSigma$ is block diagonal, $a=\min(x,y)=\sigma_D(\bSigma)$, and the claim is immediate; we thus assume $z>0$, so that
$$
2a ={(x+y)-\sqrt{(x-y)^2+4z^2}}< {(x+y)-\sqrt{(x-y)^2}}\leq {(x+y)-(x-y)}=2y.
$$
It thus follows from the definition of $y$ and the above inequality that
\begin{equation}
\label{eq:L*-a_in_S+}
\bSigma_{L_*,L_*}-a\bI \in S_{+}.
\end{equation}
We also note that
\begin{equation}
x-\frac{z^2}{y-a}=x-\frac{2z^2}{(y-x)+\sqrt{(x-y)^2+4z^2}}=x-\frac{-(y-x)+\sqrt{(x-y)^2+4z^2}}{2}=a.
\label{eq:identity_axyz}
\end{equation}

Let
\[
\bSigma_0=\begin{pmatrix}
\bSigma_{L_*,L_*}-a\bI & \bSigma_{L_*,L_*^\perp} \\
\bSigma_{L_*^\perp,L_*} & \bSigma_{L_*^\perp,L_*}(\bSigma_{L_*,L_*}-a\bI)^{-1}\bSigma_{L_*^\perp,L_*}^\top
\end{pmatrix},
\]
so that
\begin{equation}
\bSigma-\bSigma_0=\begin{pmatrix}
a\bI & 0 \\
0 & \bSigma_{L_*^\perp,L_*^\perp}-\bSigma_{L_*^\perp,L_*}(\bSigma_{L_*,L_*}-a\bI)^{-1}\bSigma_{L_*^\perp,L_*}^\top
\end{pmatrix}.
\label{eq:block_Sigma_minus_S0}
\end{equation}
We will bound from below the eigenvalues of $\bSigma_0$ and $\bSigma-\bSigma_0$ and use these bounds to conclude that $\bSigma\psdgeq a\bI$.
We first note that $\bSigma_0 \in S_+$ as follows:
\begin{align*}
&\bSigma_0=\begin{pmatrix}
\bSigma_{L_*,L_*}-a\bI \\
\bSigma_{L_*^\perp,L_*}
\end{pmatrix} (\bSigma_{L_*,L_*}-a\bI)^{-1}\begin{pmatrix}
\bSigma_{L_*,L_*}-a\bI &
\bSigma_{L_*,L_*^\perp}
\end{pmatrix}\\=&\Bigg(\begin{pmatrix}
\bSigma_{L_*,L_*}-a\bI \\
\bSigma_{L_*^\perp,L_*}
\end{pmatrix}(\bSigma_{L_*,L_*}-a\bI)^{-0.5}\Bigg)\Bigg(\begin{pmatrix}
\bSigma_{L_*,L_*}-a\bI \\
\bSigma_{L_*^\perp,L_*}
\end{pmatrix}(\bSigma_{L_*,L_*}-a\bI)^{-0.5}\Bigg)^\top.
\end{align*}

Next, we note that the eigenvalues of $\bSigma-\bSigma_0$ consist of $a$ (with multiplicity $d$) and the eigenvalues of the bottom right matrix in \eqref{eq:block_Sigma_minus_S0},
which are bounded below by $a$ as follows
\begin{align*}
&\sigma_{D-d}(\bSigma_{L_*^\perp,L_*^\perp}-\bSigma_{L_*^\perp,L_*}(\bSigma_{L_*,L_*}-a\bI)^{-1}\bSigma_{L_*^\perp,L_*}^\top)
\geq \sigma_{D-d}\left(\bSigma_{L_*^\perp,L_*^\perp}\right)\\&-\sigma_1(\bSigma_{L_*^\perp,L_*}(\bSigma_{L_*,L_*}-a\bI)^{-1}\bSigma_{L_*^\perp,L_*}^\top))
\geq  x- \frac{\|\bSigma_{L_*^\perp,L_*}\|^2}{\sigma_{d}\left(\bSigma_{L_*,L_*}\right)-a}\geq x-\frac{z^2}{y-a}=a.
\end{align*}
The above first inequality uses Lemma \ref{lemma:matrixperturbation}(b) and the observation that \eqref{eq:L*-a_in_S+} implies that $\sigma_1$ in this inequality is applied to a positive semidefinite matrix and  thus it coincides with its spectral norm. The second inequality follows from the sub-multiplicative inequality for $\sigma_1(\cdot)=\|\cdot\|$, where we apply $\sigma_1$ to positive semidefinite matrices, and the fact that $\sigma_1((\bSigma_{L_*,L_*}-a\bI)^{-1}) = (\sigma_d(\bSigma_{L_*,L_*})-a)^{-1}$.
The last equality is \eqref{eq:identity_axyz}.

We showed that $\bSigma_0 \psdgeq 0$ and $\bSigma-\bSigma_0 \psdgeq a \bI$, and consequently $\bSigma \psdgeq a\bI$, that is, $\sigma_D(\bSigma)\geq a$. Lastly, we note that if $z\leq \sqrt{xy}/2\leq \max(x,y)/2$, then
\begin{align*}
&a=\frac{2xy-2z^2}{(x+y)+\sqrt{(x-y)^2+4z^2}}\geq \frac{2xy-2(\sqrt{xy}/2)^2}{(x+y)+|x-y|+2z} = \frac{3xy/2}{2\max(x,y)+2z}\\\geq& \frac{3xy/2}{2\max(x,y)+\max(x,y)}=\frac{\min(x,y)}{2}.
\qedhere
\end{align*}
\end{proof}
\begin{proof}[Proof of Lemma~\ref{lemma:lowrank}]
Since $\bSigma\in S_+$ and $\rank(\bSigma)=d$, $\bSigma=\bU\bU^\top$ with $\bU\in\R^{D\times d}$.
WLOG, $L_*$ is spanned by the first  $d$ standard basis vectors $\bm{e}_1$, $\ldots$, $\bm{e}_d$ in $\R^D$. We write $\bU=[\bU_1;\bU_2]$ with $\bU_1\in\R^{d\times d}$ and $\bU_2\in\R^{(D-d)\times d}$, so this assumption means that $\bU_1\bU_1^\top=\bSigma_{L_*,L_*}$ and consequently $\bU_1=\bSigma_{L_*,L_*}^{0.5}\bR$, where $\bR \in O(d)$. Since $\bU_2\bU_1^\top=\bSigma_{L_*^\perp,L_*}$, $\bU_2=\bSigma_{L_*^\perp,L_*}\bSigma_{L_*,L_*}^{-0.5}\bR$. Noting further that $\bSigma_{L_*^{\perp},L_*^{\perp}}=\bU_2\bU_2^\top$, we conclude the lemma.
\end{proof}
\begin{proof}[Proof of Lemma~\ref{lemma:matrixeigenvalue}]
To prove the first inequality we assume that both $\bA$ and $\bB$ are invertible, since otherwise the inequality is trivial. In this case it is equivalent to
\[
\sigma_1(\bA^{-1}\bB^{-1}\bA^{-1})
\leq\sigma_1^2(\bA^{-1})\sigma_1(\bB^{-1}),
\]
which follows from submultiplicativity of the norm.

To prove the other inequality we assume  that $\bA$ is invertible, since otherwise $\sigma_d(\bA \bB \bA)=0$ and the inequality is trivial. We form   $\bar{\bu}:=\bA^{-1}\bu_d(\bB)$ and note that
$\bar{\bu}^\top(\bA\bB\bA)\bar{\bu}=\sigma_d(\bB)$  and  $\|\bar{\bu}\|\geq {\sigma_1^{-1}(\bA)}$.
Consequently,
$\sigma_d\big(\bA\bB\bA\big)\leq {\bar{\bu}^\top\bA\bB\bA\bar{\bu}}/{\|\bar{\bu}\|^2}\leq \sigma_1^2(\bA)\sigma_d(\bB)$.
\end{proof}

\begin{proof}[Proof of Lemma~\ref{lemma:special}]
For this proof, write $p:=d\dssnr/(D-d+d\dssnr)$ for the limiting fraction of inliers.
(a) We first note that since $\dssnr<1$, the limiting distribution of the data assigns to every nonzero proper subspace $V$ of $\R^D$ a mass smaller than $\dim(V)/D$: it assigns mass $p<d/D\leq\dim(V)/D$ to $L_*$ and to every subspace $V$ containing $L_*$, and mass zero to every subspace not containing $L_*$ (the Gaussian components are absolutely continuous on $L_*$ and on $\R^D$, respectively). Under this condition the population TME solution exists, is full-rank and is unique up to scaling (see \cite{tyler1987distribution,zhang2016robust} and \cite[Theorem~4.9, Case~0]{DPS15}). Recall that the TME solution is the minimizer of the objective function
\[
F(\bSigma)=\Expect_{\bx}\log(\bx^\top\bSigma^{-1}\bx)+\frac{1}{D}\log\det(\bSigma).
\]
Recall also that using the basis vectors for $L_*$ and $L_*^\perp$, we may write $\bSigma$ as
\[
\bSigma=\begin{pmatrix}
\bSigma_{L_*,L_*} & \bSigma_{L_*,L_*^\perp} \\
\bSigma_{L_*^\perp,L_*} & \bSigma_{L_*^\perp,L_*^\perp}
\end{pmatrix}.
\]
We define
\begin{equation*}
\bSigma^{\sharp}=\begin{pmatrix}
\bSigma_{L_*,L_*} & -\bSigma_{L_*,L_*^\perp} \\
-\bSigma_{L_*^\perp,L_*} & \bSigma_{L_*^\perp,L_*^\perp}
\end{pmatrix}.
\end{equation*}

We note that
\[
\bSigma^{\sharp}=
\begin{pmatrix}
\bI & \bm{0} \\
\bm{0} & -\bI
\end{pmatrix}
\bSigma
\begin{pmatrix}
\bI & \bm{0} \\
\bm{0} & -\bI
\end{pmatrix}
\quad \text{and} \quad
\bSigma^{-1}=
\begin{pmatrix}
\bI & \bm{0} \\
\bm{0} & -\bI
\end{pmatrix}
(\bSigma^{\sharp})^{-1}
\begin{pmatrix}
\bI & \bm{0} \\
\bm{0} & -\bI
\end{pmatrix},\]
that is, the diagonal blocks of $\bSigma$ and $\bSigma^{\sharp}$, and also $\bSigma^{-1}$ and $(\bSigma^{\sharp})^{-1}$, are the same and the off-diagonal blocks have opposite signs. Consequently,
for $\bx=[\bx_{L_*},\bx_{L_*^\perp}]$ and ${\bx'}=[\bx_{L_*},-\bx_{L_*^\perp}]$,  $$\bx^\top\bSigma^{-1}\bx=\bx'^{\top}(\bSigma^{\sharp})^{-1}\bx' \ \text{ and } \ \bx^\top(\bSigma^{\sharp})^{-1}\bx=\bx'^{\top}\bSigma^{-1}\bx'.$$
Combining these observations with the fact that $\bx$ and $\bx'$ occur with equal probability results in $F({\bSigma})=F(\bSigma^{\sharp})$. Therefore, if $\bSigma_*$ is a TME solution, then $\bSigma_*^{\sharp}$ is also a TME solution.

Uniqueness up to scaling gives $\bSigma_*^\sharp=t\bSigma_*$ for some $t>0$.
Reflection preserves the determinant, so
$\det(\bSigma_*^\sharp)=\det(\bSigma_*)>0$ implies $t^D=1$.
Thus $t=1$ and $\bSigma_*^\sharp=\bSigma_*$, which forces both off-diagonal blocks to vanish.

(b) First rescale the outlier covariance so that its $L_*^\perp$ block $\mathbf{C}_0$ has smallest eigenvalue one; this does not change the angular distribution. Write $K_{\mathrm{orig}}$ for the condition number of the full outlier covariance before whitening, so that $\sigma_1(\mathbf{C}_0)\leq K_{\mathrm{orig}}$. We derive the following estimates in tail-whitened coordinates:
\begin{equation}
\label{eq:TME_rot_invariance}
[\bSigma^{(\mathrm{out})}]_{L_*^\perp,L_*^\perp}=\bI.
\end{equation}
If not, we may apply the invertible linear map $\bx\mapsto \bP_{L_*}\bx+\bU_{L_*^\perp}(\bSigma^{(\mathrm{out})}_{L_*^\perp,L_*^\perp})^{-0.5}\bU_{L_*^\perp}^\top\bx$ to the data so that this condition is satisfied. This map leaves the signal block unchanged and replaces the tail covariance block $\mathbf C_0$ by $\bI$. Since the covariance is block diagonal and $\sigma_{D-d}(\mathbf C_0)=1$, the eigenvalue $1$ lies in the original covariance's spectral interval. Thus tail whitening does not increase its condition number. We continue to write $\kappa_{\mathrm{out}}$ for the condition number in the whitened coordinates. The outlier eigenvalues in these coordinates lie in $[1/\kappa_{\mathrm{out}},\kappa_{\mathrm{out}}]$.

First of all, we will show that for a TME solution $\bSigma_*$, $[\bSigma_*]_{L_*^\perp,L_*^\perp}$ is a scalar multiple of $\bSigma^{(\mathrm{out})}_{L_*^\perp,L_*^\perp}$.
In view of the equivariance property of TME (see \eqref{eq:tme_equivariance}) and \eqref{eq:TME_rot_invariance}, if $\bSigma_*$ is a TME solution, then for any orthogonal matrix $\bU\in\R^{(D-d)\times (D-d)}$,
\[
\hat{\bSigma}_*=\begin{pmatrix}
[\bSigma_*]_{L_*,L_*} & \bm{0} \\
\bm{0} & \bU^\top[\bSigma_*]_{L_*^\perp,L_*^\perp}\bU
\end{pmatrix}
\]
is also a TME solution. Considering the uniqueness of the solution, $\hat{\bSigma}_*$ is a scalar multiple of $\bSigma_*$, that is, for any orthogonal matrix $\bU$,   $\bU^\top[\bSigma_*]_{L_*^\perp,L_*^\perp}\bU$ is a scalar multiple of $[\bSigma_*]_{L_*^\perp,L_*^\perp}$. As a result,  $[\bSigma_*]_{L_*^\perp,L_*^\perp}$ is a scalar matrix, and in view of \eqref{eq:TME_rot_invariance}, we proved the first statement of Lemma~\ref{lemma:special}.

Second, we will show that the condition number of  $[\bSigma_*]_{L_*,L_*}$, $\frac{\sigma_1([\bSigma_*]_{L_*,L_*})}{\sigma_d([\bSigma_*]_{L_*,L_*})}$, is bounded. We note that  \[
\bSigma_*=\begin{pmatrix}
[\bSigma_*]_{L_*,L_*} & \bm{0} \\
\bm{0} & [\bSigma_*]_{L_*^\perp,L_*^\perp}
\end{pmatrix},\,\, \bSigma_*^{-1}=\begin{pmatrix}
[\bSigma_*]_{L_*,L_*}^{-1} & \bm{0} \\
\bm{0} & [\bSigma_*]_{L_*^\perp,L_*^\perp}^{-1}\end{pmatrix}.
\]
Furthermore, we let $\tilde{\bx}=[\bSigma_*]_{L_*,L_*}^{-0.5}\bU_{L_*}^\top\bx$ (recall that $\bU_{L_*} \in \R^{D \times d}$ was defined in Section \ref{sec:notation}),
and note that if $\bx$ is an inlier, then  $\bx^\top\bSigma_*^{-1}\bx=\|\tilde{\bx}\|^2$, and if $\bx$ is an outlier, then  $\bx^\top\bSigma_*^{-1}\bx=\|\tilde{\bx}\|^2+\|([\bSigma_*]_{L_*^\perp,L_*^\perp})^{-0.5}\bU_{L_*^\perp}^\top\bx\|^2$. Conjugating both sides of the fixed-point equation in \eqref{eq:TME_definition} by $\bSigma_*^{-1/2}$, taking the expectation,  restricting it to the $d\times d$ submatrix represented by $[\cdot]_{L_*,L_*}$ (that is, using the variables $\tilde{\bx}$), and using the latter observation on $\bx^\top\bSigma_*^{-1}\bx$ for inliers and outliers, result in
\begin{equation}\label{eq:TME_cond}
p\Expect_{\bx\in\calX_{\mathrm{in}}}\frac{\tilde{\bx}\tilde{\bx}^\top}{\|\tilde{\bx}\|^2}+(1-p)\Expect_{\bx\in\calX_{\mathrm{out}}}\frac{\tilde{\bx}\tilde{\bx}^\top}{\|\tilde{\bx}\|^2+\|([\bSigma_*]_{L_*^\perp,L_*^\perp})^{-0.5}\bU_{L_*^\perp}^\top\bx\|^2}=\frac{1}{D}\bI_{d\times d}.
\end{equation}

Next, we show that $\Expect_{\bx\in\calX_{\mathrm{in}}}{\tilde{\bx}\tilde{\bx}^\top}/{\|\tilde{\bx}\|^2}$ and $\Expect_{\bx\in\calX_{\mathrm{out}}}\!\left[\frac{\tilde{\bx}\tilde{\bx}^\top}{\|\tilde{\bx}\|^2+\|([\bSigma_*]_{L_*^\perp,L_*^\perp})^{-0.5}\bU_{L_*^\perp}^\top\bx\|^2}\right]$ share an eigenbasis and have opposite eigenvalue orderings, by \eqref{eq:TME_cond}.
The proof of this claim is based on  \eqref{eq:TME_cond} and the following lemma, which is proved later on.

\begin{lemma}
\label{lemma:S_M}
Assume that $\bx\sim\mathcal{N}(\bm{0},\bM)$, where $\bM\in S_{++}(d)$, $Y$ is an independent nonnegative random variable, and let $\bS$ be the following matrix:
\[
\bS \;:=\; \Expect\!\left[\frac{\bx\bx^\top}{\|\bx\|^2+Y}\right]
\equiv \Expect_{\bx,Y}\!\left[\frac{\bx\bx^\top}{\|\bx\|^2+Y}\right].
\]
Then $\bS$ and $\bM$ have the same eigenspaces
and their eigenvalue orderings are the same.
\end{lemma}

Let $\overline{\bSigma}_*^{(\mathrm{in})}:=([\bSigma_*]_{L_*,L_*})^{-1/2} [\bSigma^{(\mathrm{in})}]_{L_*,L_*} ([\bSigma_*]_{L_*,L_*})^{-1/2}.$ We note that for $\bx \in \calX_{\mathrm{in}}$, $\tilde{\bx}$ is sampled from
$\mathcal{N}(\bm{0},\overline{\bSigma}_*^{(\mathrm{in})})$ after omitting the factor $1/d$, which is immaterial by radial invariance.
Therefore, by Lemma~\ref{lemma:S_M},
\begin{equation}
\overline{\bSigma}_*^{(\mathrm{in})} \text{ and }   \Expect_{\bx \in \calX_{\mathrm{in}}}\!\left[{\tilde{\bx}\tilde{\bx}^\top}/{\|\tilde{\bx}\|^2}\right]
\text{ have the same eigenvectors and order of eigenvalues.
}
\label{eq:S_M_apply1}
\end{equation}

Let $\overline{\bSigma}_*^{(\mathrm{out})}:=([\bSigma_*]_{L_*,L_*})^{-1/2} [\bSigma^{(\mathrm{out})}]_{L_*,L_*} ([\bSigma_*]_{L_*,L_*})^{-1/2}.$
We note that the distribution of $\bx \in \calX_{\mathrm{out}}$ can be decomposed into that of two independent components: $\bU_{L_*}^\top\bx$ and $\bU_{L_*^\perp}^\top\bx$, sampled independently from $\mathcal{N}(\bm{0},[\bSigma^{(\mathrm{out})}]_{L_*,L_*})$ and $\mathcal{N}(\bm{0},[\bSigma^{(\mathrm{out})}]_{L_*^\perp,L_*^\perp})=\mathcal{N}(\bm{0},\bI_{(D-d)})$, respectively (we omit the common scaling factor $1/D$, which is immaterial here).
Note that $\tilde{\bx}$ only depends on the first component and it can be considered as sampled from
$\mathcal{N}(\bm{0},\overline{\bSigma}_*^{(\mathrm{out})}).$
Therefore  by Lemma~\ref{lemma:S_M},
\begin{multline}
    \label{eq:S_M_apply2}
\overline{\bSigma}_*^{(\mathrm{out})} \text{ and }
\Expect_{\bx\in\calX_{\mathrm{out}}}\!\left[\frac{\tilde{\bx}\tilde{\bx}^\top}{\|\tilde{\bx}\|^2+\|([\bSigma_*]_{L_*^\perp,L_*^\perp})^{-0.5}\bU_{L_*^\perp}^\top\bx\|^2}\right]\\ \text{ have the same eigenvectors and order of eigenvalues.}
\end{multline}

Let $\bu_1$ and $\bu_d$ be the largest and smallest eigenvectors of $\overline{\bSigma}_*^{(\mathrm{in})}$, then by \eqref{eq:S_M_apply1}, \eqref{eq:S_M_apply2},
and \eqref{eq:TME_cond}, they are the smallest and the largest eigenvectors of $\overline{\bSigma}_*^{(\mathrm{out})}$ respectively. That is,
\[
\frac{\bu_1^\top \overline{\bSigma}_*^{(\mathrm{in})} \bu_1}{\bu_d^\top \overline{\bSigma}_*^{(\mathrm{in})} \bu_d}\geq 1\geq \frac{\bu_1^\top \overline{\bSigma}_*^{(\mathrm{out})} \bu_1}{\bu_d^\top \overline{\bSigma}_*^{(\mathrm{out})} \bu_d}.
\]
For any nonzero $\bu\in\R^d$, covariance comparison gives
\[
\frac{\sigma_d(\bSigma^{(\mathrm{in})})}
{\sigma_1([\bSigma^{(\mathrm{out})}]_{L_*,L_*})}
\leq
\frac{\bu^\top\overline{\bSigma}_*^{(\mathrm{in})}\bu}
{\bu^\top\overline{\bSigma}_*^{(\mathrm{out})}\bu}
\leq
\frac{\sigma_1(\bSigma^{(\mathrm{in})})}
{\sigma_d([\bSigma^{(\mathrm{out})}]_{L_*,L_*})}.
\]
Applying the upper bound to $\bu_1$ and the lower bound to $\bu_d$ yields
\[
\frac{\bu_1^\top\overline{\bSigma}_*^{(\mathrm{in})}\bu_1}
{\bu_d^\top\overline{\bSigma}_*^{(\mathrm{in})}\bu_d}
\leq\kappa_{\mathrm{in}}\kappa_{\mathrm{out}}
\frac{\bu_1^\top\overline{\bSigma}_*^{(\mathrm{out})}\bu_1}
{\bu_d^\top\overline{\bSigma}_*^{(\mathrm{out})}\bu_d}
\leq\kappa_{\mathrm{in}}\kappa_{\mathrm{out}}.
\]
Thus the condition number of $\overline{\bSigma}_*^{(\mathrm{in})}$ is at most $\kappa_{\mathrm{in}}\kappa_{\mathrm{out}}$. Its definition also gives
\[
\sigma_d(\bSigma^{(\mathrm{in})})([\bSigma_*]_{L_*,L_*})^{-1}
\psdleq\overline{\bSigma}_*^{(\mathrm{in})}
\psdleq\sigma_1(\bSigma^{(\mathrm{in})})([\bSigma_*]_{L_*,L_*})^{-1}.
\]
Comparing the extreme eigenvalues in these inequalities shows that
\begin{equation}\label{eq:lemma51b_condition}
\frac{\sigma_1([\bSigma_*]_{L_*,L_*})}
{\sigma_d([\bSigma_*]_{L_*,L_*})}
\leq\kappa_{\mathrm{in}}^2\kappa_{\mathrm{out}}.
\end{equation}

Lastly, we bound $\sigma_d([\bSigma_*]_{L_*,L_*})/\sigma_1([\bSigma_*]_{L_*^\perp,L_*^\perp})$.
Applying the TME condition \eqref{eq:TME_definition} while taking the expectation, we have
\begin{equation}\label{eq:TME_definition_asymptotic}
\frac{1}{D}\bSigma_*=\Expect_{\bx}\frac{\bx\bx^\top}{\bx^\top\bSigma_*^{-1}\bx}.
\end{equation}
In the whitened coordinates, rotation invariance and part (a) imply that $[\bSigma_*]_{L_*^\perp,L_*^\perp}=c\bI$ for some $c>0$, and $[\bSigma_*]_{L_*,L_*^\perp}=\bm{0}$. Since the solution is defined up to scaling, rescale $\bSigma_*$ so that its tail block is $\bI$ for the rest of this calculation. For  $\varrho:=1/\sigma_1([\bSigma_*]_{L_*,L_*})$, we observe \begin{align*}&\bx^\top\bSigma_*^{-1}\bx=\bx_{L_*}^\top[\bSigma_*]_{L_*,L_*}^{-1} \bx_{L_*}+\bx_{L_*^\perp}^\top[\bSigma_*]_{L_*^\perp,L_*^\perp}^{-1} \bx_{L_*^\perp}\geq \frac{\|\bx_{L_*}\|^2}{\sigma_1([\bSigma_*]_{L_*,L_*})}+\|\bx_{L_*^\perp}\|^2\\=&{\|\bx_{L_*^\perp}\|^2}+ \varrho{\|\bx_{L_*}\|^2}.\end{align*} As a result, restricting \eqref{eq:TME_definition_asymptotic} to the $(D-d)\times (D-d)$ submatrix represented by $[\cdot]_{L_*^\perp,L_*^\perp}$ gives
\[
\frac{1}{D}\bI_{(D-d)\times (D-d)}=\frac{1}{D}([\bSigma_*]_{L_*^\perp,L_*^\perp})=\Expect_{\bx}\frac{\bx_{L_*^\perp}\bx_{L_*^\perp}^\top}{\bx^\top\bSigma_*^{-1}\bx}\psdleq \Expect_{\bx}\frac{\bx_{L_*^\perp}\bx_{L_*^\perp}^\top}{{\|\bx_{L_*^\perp}\|^2}+ \varrho{\|\bx_{L_*}\|^2}}.
\]
Taking traces from both sides, we have
\begin{equation}\label{eq:5.1_b3}
\Expect_{\bx}\frac{\|\bx_{L_*^\perp}\|^2}{{\|\bx_{L_*^\perp}\|^2}+ \varrho{\|\bx_{L_*}\|^2}}\geq \frac{(D-d)}{D}.
\end{equation}

Note that in the LHS of \eqref{eq:5.1_b3}, $\bx_{L_*^\perp}$ is only nonzero when $\bx\in\calX_{\mathrm{out}}$, which happens with probability $1-p$. In addition, $[\bSigma^{(\mathrm{out})}]_{L_*^\perp,L_*^\perp}=\bI_{(D-d)\times (D-d)}$, which implies $\sigma_1(\bSigma^{(\mathrm{out})})\geq 1$ and  $\sigma_D(\bSigma^{(\mathrm{out})})\geq 1/\kappa_{\mathrm{out}}$, and consequently,    $[\bSigma^{(\mathrm{out})}]_{L_*,L_*}\psdgeq \bI_{d\times d}/{\kappa_{\mathrm{out}}} $. We thus note that if $\bx\in\calX_{\mathrm{out}}$, then $\|\bx_{L_*^\perp}\|^2$ follows a $\chi^2_{D-d}$ distribution, and $\|\bx_{L_*}\|^2$ is bounded below by a $\chi^2_d$ distribution scaled by $1/\kappa_{\mathrm{out}}$. Let $X_1\sim \chi^2_{D-d}$ and $X_2\sim \chi^2_{d}$, then the LHS of \eqref{eq:5.1_b3} can be bounded above by
\[
\Expect_{\bx}\frac{\|\bx_{L_*^\perp}\|^2}{{\|\bx_{L_*^\perp}\|^2}+ \varrho{\|\bx_{L_*}\|^2}}\leq (1-p)
\Expect_{X_1\sim \chi_{D-d}^2,X_2\sim \chi_{d}^2}\frac{X_1}{X_1+\frac{\varrho}{\kappa_{\mathrm{out}}}  X_2}.
\]
The above equation and \eqref{eq:5.1_b3} imply
\begin{equation}\label{eq:eta0}
\Expect_{X_1\sim \chi_{D-d}^2,X_2\sim \chi_{d}^2}\frac{X_1}{X_1+\frac{\varrho}{\kappa_{\mathrm{out}} } X_2}\geq \frac{D-d}{D(1-p)}=1-(1-\dssnr)\frac{d}{D}.
\end{equation}

Next, we use \eqref{eq:eta0} to upper bound  $\varrho=1/\sigma_1([\bSigma_*]_{L_*,L_*})$. 
An equivalent reformulation of \eqref{eq:eta0} is
\begin{equation}\label{eq:eta0_upperbound}
\Expect_{X_1\sim \chi_{D-d}^2,X_2\sim \chi_{d}^2}\frac{ {\varrho}X_2}{{\kappa_{\mathrm{out}}}X_1+ {\varrho} X_2}\leq (1-\dssnr)\frac{d}{D}.
\end{equation}
Set $q:=D-d$ and $t:=1-\dssnr$. By Cauchy--Schwarz,
using the independence of $X_1$ and $X_2$,
\[
\Expect\frac{\varrho X_2}{\kappa_{\mathrm{out}}X_1+\varrho X_2}\geq\frac{\varrho(\Expect X_2)^2}{\Expect[X_2(\kappa_{\mathrm{out}}X_1+\varrho X_2)]}=\frac{\varrho d}{\kappa_{\mathrm{out}}q+\varrho(d+2)}.
\]
Indeed, $\Expect X_1=q$, $\Expect X_2=d$, and
$\Expect X_2^2=d(d+2)$. Whenever $D-t(d+2)>0$, combining this bound with
\eqref{eq:eta0_upperbound} gives
\[
\varrho\leq
\frac{\kappa_{\mathrm{out}}qt}{D-t(d+2)}.
\]
For the bound required by Lemma~\ref{lemma:special}, take $t\leq D/(4d)$. Since $(d+2)/(4d)\leq3/4$ for every $d\geq1$, the denominator satisfies $D-t(d+2)\geq D/4>0$. More generally, it is positive for every $0<t<1$ when $q>2$, as in the block-isotropic model of Appendix~\ref{sec:block_isotropic_details}.

Combining the direct bound on $\varrho$ with \eqref{eq:lemma51b_condition}, and using our normalization of the tail block, gives
\[
\frac{\sigma_d([\bSigma_*]_{L_*,L_*})}{\sigma_1([\bSigma_*]_{L_*^\perp,L_*^\perp})}=\sigma_d([\bSigma_*]_{L_*,L_*})\geq\frac{1}{\kappa_{\mathrm{in}}^2\kappa_{\mathrm{out}}\varrho}\geq\frac{D-t(d+2)}{\kappa_{\mathrm{in}}^2\kappa_{\mathrm{out}}^2qt}\geq\frac{D}{4\kappa_{\mathrm{in}}^2\kappa_{\mathrm{out}}^2(D-d)(1-\dssnr)}.
\]
Thus, in the whitened coordinates, part (b) holds for $\dssnr\in[1-D/(4d),1)\cap(0,1)$ with
\[
c_{\dssnr,0}:=1-\frac{D}{4d},\qquad
c_{\dssnr}:=\frac{D}{4\kappa_{\mathrm{in}}^2\kappa_{\mathrm{out}}^2(D-d)}.
\]
Returning to the original coordinates leaves the signal block unchanged and replaces the scalar tail block $c\bI$ by $c\mathbf C_0$. Its largest eigenvalue is at most $cK_{\mathrm{orig}}$, so the lower-bound constant is divided by at most $K_{\mathrm{orig}}$. Since the whitened outlier condition number is at most $K_{\mathrm{orig}}$ and the signal-block condition number is at most $\kappa_{\mathrm{in}}^2\kappa_{\mathrm{out}}$, a permissible original-coordinate constant is
\[
c_{\dssnr}:=\frac{D}{4\kappa_{\mathrm{in}}^2K_{\mathrm{orig}}^3(D-d)}.
\]
This proves the stated dependence of the constants after undoing the whitening transformation. The uniform-cutoff argument uses the same condition-number bound.
\end{proof}

\begin{proof}[Proof of Lemma~\ref{lemma:S_M}]
We first prove that $\bS$ and $\bM$ have the same  eigenspaces. That is, we prove the existence of an orthogonal $\bQ \in O(d)$ and scalars $\{\alpha_i\}_{i=1}^d$ such that
\begin{equation}
\label{eq:S_M}
\bS = \bQ \, \mathrm{diag}(\alpha_1,\dots,\alpha_d)\, \bQ^\top
\quad\text{and}\quad
\bM = \bQ \, \mathrm{diag}(\lambda_1,\dots,\lambda_d) \,\bQ^\top.
\end{equation}
WLOG assume that $\bM$ is a diagonal matrix with diagonal entries $\lambda_1\geq\cdots\geq \lambda_d$ and $\bQ$ is the identity. Then it is sufficient to prove that $\bS$ is diagonal with nonincreasing diagonal entries.

For $i,j=1$, $\ldots$, $d$, $x_i\sim \mathcal{N}(0,\lambda_i)$, and
$
\bS_{ij}=\Expect\!\left[{x_ix_j}/({\sum_{k=1}^dx_k^2+Y})\right]$.
For $i\neq j$, the integrand is odd in $x_i$, while its Gaussian density is even. Hence $\bS_{ij}=0$, which proves \eqref{eq:S_M}.

If $\lambda_i=\lambda_j$, exchanging $x_i$ and $x_j$ gives $\alpha_i=\alpha_j$. To show strict ordering, assume $\lambda_i>\lambda_j>0$ and prove $\alpha_i>\alpha_j$.
We note that the coordinates \(x_k\) are independent, with \(x_k\sim\mathcal{N}(0,\lambda_k)\). We fix the values of the remaining coordinates \(\{x_k\}_{k\neq i,j}\) and of \(Y\ge0\), and set
\[
E \;:=\; \sum_{k\neq i,j} x_k^2 + Y \ge 0,\,\, g(a,b) \;:=\; \frac{a^2-b^2}{E + a^2 + b^2},\qquad (a,b)\in\mathbb R^2,
\]
and denote by \(f_k(x)\) the marginal Gaussian density, $f_k(x)=e^{-x^2/(2\lambda_k)}/{\sqrt{2\pi\lambda_k}}$.
We will prove the following positivity result:
\begin{equation}
\label{eq:cond_exp_withE}
\Expect_{x_i,x_j}\!\left[\frac{x_i^2-x_j^2}{E + x_i^2+x_j^2}\right]
= \iint_{\mathbb R^2} g(a,b)\,f_i(a)f_j(b)\,da\,db \;>\; 0.
\end{equation}
This means that the conditional expectation of $(x_i^2-x_j^2)/(\sum_{k=1}^d x_k^2 + Y)$ (given \(\{x_k\}_{k\neq i,j}\) and \(Y\)) is strictly positive.
Taking the outer expectation over \(\{x_k\}_{k\neq i,j}\) and \(Y\),
the overall expectation of the same random variable, which coincides with $\alpha_i-\alpha_j$, is positive, and thus $\alpha_i > \alpha_j$.

Pairing $(a,b)$ with $(b,a)$ and using $g(b,a)=-g(a,b)$ gives
\begin{equation}\label{eq:Igf}
\iint_{\mathbb R^2}g(a,b)f_i(a)f_j(b)\,da\,db
=\iint_{a^2>b^2}g(a,b)
\bigl(f_i(a)f_j(b)-f_i(b)f_j(a)\bigr)\,da\,db.
\end{equation}
For $a^2>b^2$, we have $g(a,b)>0$ and
\[
\log\frac{f_i(a)f_j(b)}{f_i(b)f_j(a)}
=(a^2-b^2)\frac{\lambda_i-\lambda_j}{2\lambda_i\lambda_j}>0.
\]
The integrand on the right of \eqref{eq:Igf} is therefore strictly positive. This proves \eqref{eq:cond_exp_withE} and the lemma.

\end{proof}

\section{List-Decodable Initialization for STE}
\label{sec:initialization}
This application separates the cost of entering the STE basin from the cost of subsequent refinement. We apply the list-decodable algorithm of Bakshi and Kothari~\cite[Theorem~5.1]{Bakshi2021} at an accuracy fixed before sampling. When the realized data and this accuracy satisfy the basin condition, STE refines the successful candidate without increasing the SoS relaxation level or the initial list size.

Here $\alpha$ denotes the fraction of designated clean observations, not necessarily the fraction of all observations lying on $L_*$; arbitrary additional observations may also lie on $L_*$.

\begin{definition}[Sum-of-Squares Proof]
A Sum-of-Squares (SoS) proof of a polynomial inequality
$q(\bx) \ge 0$ is a certificate demonstrating that
$q(\bx) = \sum_i f_i(\bx)^2$
for some finite set of polynomials $f_i$.
Equivalently, an SoS proof for an inequality
$p_1(\bx) \ge p_2(\bx)$
consists of polynomials $f_i$ such that
$p_1(\bx) - p_2(\bx) = \sum_i f_i(\bx)^2.$
A degree-$m$ SoS proof means that each squared term $f_i(\bx)^2$ has total degree at most $m$.
\end{definition}

Intuitively, an anti-concentrated distribution is one where very little probability mass is concentrated near any hyperplane through the origin.
\begin{definition}[Certifiable Anti-Concentration]
\label{def:anti_concentration}
We use $k$-certifiable $(\delta,C\delta)$-anti-concentration in the
sense of \cite[Definition~1.6 and Section~6.1]{Bakshi2021}. Here $k$
is the degree of the sum-of-squares certificates, $\delta$ is the
small-ball scale, and $C\delta$ bounds the corresponding expected
polynomial penalty. The polynomial certificates are understood in the
homogeneous form used in Section~6.1 of that reference. For a
distribution supported on a subspace, the condition is imposed in
intrinsic coordinates on that subspace.
\end{definition}

This condition provides a polynomial certificate that the distribution
assigns little mass to narrow neighborhoods of hyperplanes through the origin within its
support. The next notion controls the moments of quadratic measurements
\cite[Definition~1.3]{Bakshi2021}.
\begin{definition}[Certifiable Hypercontractivity]
\label{def:hypercontractive}
A distribution $\mathcal{D}$ on $\R^d$ is $(C,2h)$-certifiably hypercontractive if there is a degree-$2h$ SoS proof in the $d \times d$ matrix-valued indeterminate $\bQ$ of the following inequality:
\[
\Expect_{\bx \sim \mathcal{D}}
\left(
\bx^\top \bQ\bx - \Expect_{\bx \sim \mathcal{D}} \bx^\top \bQ\bx
\right)^{2h}
\leq
(Ch)^{2h}
\left(
\Expect_{\bx \sim \mathcal{D}}
\left(
\bx^\top \bQ\bx - \Expect_{\bx \sim \mathcal{D}} \bx^\top \bQ\bx
\right)^2
\right)^h.
\]
\end{definition}
The above inequality bounds the $2h$-th central moment of a quadratic measurement by its variance raised to the $h$-th power. The cited SoS procedure also requires a bound on that variance itself.
\begin{definition}[Certifiable bounded variance]\label{def:bounded_variance}
A zero-mean distribution $\mathcal D$ with covariance $\bP_{L_*}$ has certifiably $C_{\mathrm v}$-bounded variance of degree-2 polynomials if there is a degree-$2$ SoS proof in the matrix-valued indeterminate $\bQ$ of
\[
\Expect_{\bx\sim\mathcal D}
\bigl(\bx^\top\bQ\bx-\operatorname{tr}(\bQ\bP_{L_*})\bigr)^2
\leq C_{\mathrm v}\|\bP_{L_*}\bQ\bP_{L_*}\|_F^2.
\]
\end{definition}
This is the bounded-variance condition used in the constraint system of \cite[Section~5]{Bakshi2021}. Since the inequality is quadratic in $\bQ$, its validity already provides a degree-$2$ SoS certificate. It is separate from Definition~\ref{def:hypercontractive}, whose right-hand side uses the measurement's own variance.

The next theorem combines the high-probability initialization guarantee of \cite{Bakshi2021} with a deterministic STE refinement statement for realizations satisfying the geometric basin conditions.

\begin{theorem}[List-decodable initialization and STE refinement]\label{thm:STE_Bakshi}
Let $\calX$ consist of $N$ nonzero points in $\R^D$, of which $m=\alpha N$ are designated clean observations sampled i.i.d.~from a distribution $\mathcal D$ with mean $\bm0$, covariance $\bP_{L_*}$, and $\mathcal D(\{\bm0\})=0$, where $0<\alpha<1$ and $m$ is an integer. The remaining observations are arbitrary and may also lie on $L_*$.

For a constant $C_{\mathcal D}\geq1$ and an integer $k\geq2$, assume that $\mathcal D$ satisfies:
\begin{enumerate}[(i)]
\item $k$-certifiable $(\delta,C_{\mathcal D}\delta)$-anti-concentration, with $\delta=\alpha/(8C_{\mathcal D})$, as in Definition~\ref{def:anti_concentration}. The certificate polynomial has coefficients bounded by $k^{O(k)}$ in intrinsic isotropic coordinates, for the finite-sample certificate estimate; see \cite[Lemma~6.8]{Bakshi2021}.
\item Certifiably $C_{\mathcal D}$-bounded variance, as in Definition~\ref{def:bounded_variance}.
\item $(C_{\mathcal D},2h)$-certifiable hypercontractivity for every integer $h\geq1$, as in Definition~\ref{def:hypercontractive}.
\end{enumerate}
Fix $0<\gamma<1$ and $0<\eta<1$ before drawing the sample. Suppose that $N\geq (kD\log D/\alpha)^{O(k)}$. Then the following statements hold.

\begin{enumerate}[(a)]
\item \emph{List initialization.} Apply \cite[Theorem~5.1]{Bakshi2021} at accuracy $\eta/2$, symmetrize each matrix candidate, and take its leading $d$-dimensional eigenspace. This produces, in running time $N^{O(k+\log(1/\eta))}$, a list $\mathcal L$ of at most $(1/\alpha)^{O(\log k+\log(1/\eta))}$ $d$-subspaces. With probability at least $0.9$ over the sample and the algorithm's randomness, some $\hat L\in\mathcal L$ satisfies $\|\bP_{\hat L}-\bP_{L_*}\|_F\leq\eta$.

\item \emph{STE refinement.} Use the geometric partition $\calX_{\mathrm{in}}=\calX\cap L_*$ and $\calX_{\mathrm{out}}=\calX\setminus L_*$. Consider any realization with $n_0=|\calX_{\mathrm{out}}|>0$, $n_1=|\calX_{\mathrm{in}}|$, $\dssnr=(n_1/d)/(n_0/(D-d))>\gamma$, a positive definite TME for $\calX_{\mathrm{in}}$ in intrinsic coordinates, and finite $\calA$ and $\calR$. If $\eta\leq 1/\sqrt{2B(\calX,L_*,\gamma)}$, where $B$ is defined in Corollary~\ref{cor:angle_basin}, then every candidate $\hat L$ satisfying the preceding accuracy bound gives a successful STE initialization $\bSigma^{(0)}=\bP_{\hat L}+(\eta^2/2)\bI$. The STE subspace converges linearly to $L_*$ and satisfies
\begin{equation}\label{eq:rate_after_sos}
\sin\angle(L^{(j)},L_*)\leq
\frac{2\calA\kappa_{\mathrm{in},*}\eta^2}{\dssnr}\,C_0^{-(j-1)},
\qquad j\geq1,
\end{equation}
where $C_0$ is defined in \eqref{eq:C_value}.
\end{enumerate}
\end{theorem}

The probability in part~(a) is joint over the clean sample and the algorithm's randomness, for the accuracy fixed before sampling. Part~(b) is deterministic; no probability bound for its basin conditions is asserted. The combined statement does not identify the successful candidate in the list.

To choose the accuracy before sampling, one may use a deterministic
upper bound $\bar B>0$ on $B(\calX,L_*,\gamma)$ from
Corollary~\ref{cor:angle_basin} over a prescribed class of datasets.
The fixed choice $\eta=\min\{1/2,(2\bar B)^{-1/2}\}$ then satisfies the accuracy condition of Theorem~\ref{thm:STE_Bakshi}(b) on every realization in that class, and $\log(1/\eta)=\max\{\log 2,\log(2\bar B)/2\}$.
For fixed inlier fraction, distributional certificate degree, and
geometric upper bounds, the accuracy and runtime exponent are
independent of $N$, so the initializer runs in polynomial time in $N$.
For arbitrary outliers, the required geometric upper bounds need not
remain fixed as $N$ grows.

For Gaussian inliers, the distributional assumptions of Theorem~\ref{thm:STE_Bakshi} are known to hold with explicit parameters.
\begin{corollary}\label{cor:gaussian_init}
Assume the setting of Theorem~\ref{thm:STE_Bakshi} with $\mathcal{D}=\mathcal{N}(\bm{0},\bP_{L_*})$. Then the assumptions of Theorem~\ref{thm:STE_Bakshi} on $\mathcal{D}$ hold with $C_{\mathcal{D}}=O(1)$ and $k=O(\log^5(2/\alpha)/\alpha^2)$. Consequently, the conclusions of Theorem~\ref{thm:STE_Bakshi} hold with $N\geq (D\log D/\alpha)^{O(\log^5(2/\alpha)/\alpha^2)}$ samples, running time $N^{O(\log^5(2/\alpha)/\alpha^2+\log(1/\eta))}$ and a list of size $(1/\alpha)^{O(\log(2/\alpha)+\log(1/\eta))}$.
\end{corollary}
\begin{proof}[Proof of Corollary~\ref{cor:gaussian_init}]
By \cite[Theorem~6.2]{Bakshi2021}, Gaussian distributions satisfy the required certifiable anti-concentration condition for every $0<\delta\leq1/2$, with a constant $C_{\mathcal D}$ and $k=O(\log^5(1/\delta)/\delta^2)$.
We apply this with $\delta=\alpha/(8C_{\mathcal D})$. Gaussian distributions are also $(O(1),2h)$-certifiably hypercontractive for every $h\geq1$ by \cite[Fact~6.10]{Bakshi2021}. For Gaussian inliers, the additional bounded-variance condition follows from
$\operatorname{Var}(\bx^\top\bQ\bx)=2\|\bP_{L_*}\bQ_{\mathrm{sym}}\bP_{L_*}\|_F^2\leq2\|\bP_{L_*}\bQ\bP_{L_*}\|_F^2$, where $\bQ_{\mathrm{sym}}=(\bQ+\bQ^\top)/2$.
The Gaussian certificate construction in \cite[Fact~6.3 and Theorem~6.2]{Bakshi2021} also supplies the coefficient bound required above.
\end{proof}

Spherical symmetry and sub-Gaussian tails alone do not imply the anti-concentration condition. The corollary verifies the distributional assumptions for list initialization; the geometric conditions for STE refinement in Theorem~\ref{thm:STE_Bakshi}(b) remain separate requirements.

This two-stage procedure uses SoS only to reach a fixed basin accuracy and STE for all subsequent refinement. Unlike using the Bakshi--Kothari procedure at the final tolerance, it does not increase the SoS relaxation degree or candidate-list size as finer accuracy is requested. Fix a basin accuracy $\eta_*$ before sampling, using geometric upper bounds as above, and consider realizations satisfying the accuracy condition of Theorem~\ref{thm:STE_Bakshi}(b) with $\eta=\eta_*$. By~\eqref{eq:rate_after_sos}, for a target subspace accuracy $\tau>0$, a sufficient number of STE iterations along each successful trajectory is $1+\max\{0,\lceil\log_{C_0}(2\calA\kappa_{\mathrm{in},*}\eta_*^2/(\dssnr\tau))\rceil\}$, each requiring $O(ND^2+D^3)$ arithmetic operations. Ivkov and Kothari~\cite{IK22} already obtain exact list-decodable recovery with polylogarithmic accuracy dependence. The benefit of the present construction is a simple, linearly convergent STE refinement stage, not a stronger recovery guarantee.

\subsection{Proof of Theorem~\ref{thm:STE_Bakshi}}
\label{sec:proof_STE_Bakshi}
The certificates in Theorem~\ref{thm:STE_Bakshi} remain valid when $C_{\mathcal D}$ is replaced by $4C_{\mathcal D}$, and $\delta=\alpha/(8C_{\mathcal D})=\alpha/{(2\cdot4C_{\mathcal D})}$, so \cite[Theorem~5.1]{Bakshi2021} applies with the constant $4C_{\mathcal D}$.
Part~(a) follows from that theorem and the spectral rounding argument below. Constants in the asymptotic bounds may depend on the fixed distributional and certificate bounds in the assumptions.

Apply the cited guarantee at matrix accuracy $\eta/2$. With the stated sample size, runtime and list-size orders, and probability at least $0.9$, the matrix list contains $\bM$ satisfying
\[
\|\bM-\bP_{L_*}\|_F\leq\eta/2.
\]
Symmetrize each candidate; this cannot increase its Frobenius distance from the symmetric target $\bP_{L_*}$. Denote a symmetric candidate again by $\bM$ and let $\hat L$ be the span of its top $d$ eigenvectors. Its projector minimizes $\|\bP-\bM\|_F$ over rank-$d$ orthogonal projectors: expanding the squared norm reduces the problem to maximizing $\operatorname{tr}(\bP\bM)$. Hence
\[
\|\bP_{\hat L}-\bP_{L_*}\|_F
\leq\|\bP_{\hat L}-\bM\|_F+\|\bM-\bP_{L_*}\|_F
\leq2\|\bM-\bP_{L_*}\|_F\leq\eta.
\]
This rounding applies to every candidate and preserves the successful one without assuming that the raw matrix is a projector. Its $O(D^3)$ cost per candidate is absorbed by the stated runtime bound. Replacing $\eta$ by $\eta/2$ does not change the stated asymptotic orders.

Since $\|\bP_{\hat{L}}-\bP_{L_*}\|_F^2=2\sum_{i=1}^d\sin^2\theta_i(\hat{L},L_*)\geq 2\sin^2\theta_1(\hat{L},L_*)$, where $\theta_1(\hat{L},L_*)$ denotes the largest principal angle between $\hat{L}$ and $L_*$ (see Appendix~\ref{sec:calculate_kappa1}), this implies
\begin{equation}
\label{eq:theta_1_leq_eta}
\sin^2(\theta_1(\hat{L},L_*)) \leq \eta^2/2.
\end{equation}

Since
\(
\bSigma^{(0)} = \bP_{\hat{L}} + (\eta^2/2) \bI,
\)
the initialization $\bSigma^{(0)}$ is of the form \eqref{eq:STE_init} with regularization parameter $\eta^2/2$ (this parameter is denoted by $\alpha$ in \eqref{eq:STE_init}; it should not be confused with the inlier fraction $\alpha$ of the current theorem). By \eqref{eq:theta_1_leq_eta}, this regularization parameter is at least $\sin^2(\theta_1)$, which is the condition required in \eqref{eq:example_k1_first} and \eqref{eq:example_k1_second}.
Applying these two equations, we obtain
\begin{equation}
\label{eq:kappa_1_and_2_requirements}
\kappa_1 \geq \frac{1}{4\cdot\eta^2/2}=\frac{1}{2\eta^2}
\ \text{ and } \
\kappa_2 \leq 2.
\end{equation}
For part~(b), the assumption $\eta\leq1/\sqrt{2B(\calX,L_*,\gamma)}$ gives $\eta^2/2\leq1/(4B(\calX,L_*,\gamma))$. Together with \eqref{eq:theta_1_leq_eta}, this places the regularization parameter $\eta^2/2$ in the interval required by Corollary~\ref{cor:angle_basin}. That corollary yields \eqref{eq:kappa1} and linear convergence to $L_*$. Finally, substituting $\kappa_1\geq1/(2\eta^2)$ from \eqref{eq:kappa_1_and_2_requirements} into \eqref{eq:explicit_rate} proves \eqref{eq:rate_after_sos}.

\bibliographystyle{imsart-number}
\bibliography{RSR-ref}

\end{document}